\documentclass{article}

\usepackage{amsfonts}
\usepackage{amsmath}
\usepackage{a4wide}
\usepackage{amssymb}

\newtheorem{lem}{Lemma}[section]
\newtheorem{theo}[lem]{Theorem}
\newtheorem{coro}[lem]{Corollary}
\newtheorem{propo}[lem]{Proposition}
\newtheorem{rema}[lem]{Remark}
\newtheorem{defi}[lem]{Definition}

\newenvironment{lemma}
{\begin{lem}\sl } {\end{lem}}

\newenvironment{remark}
{\begin{rema}\rm } {\end{rema}}

\newenvironment{proof}{\paragraph*{Proof}}
{\par}

\newcommand\GL{{\mathrm{GL}}}
\newcommand\SL{{\mathrm{SL}}}

\newcommand\eps\varepsilon
\newcommand\ph\varphi
\newcommand\C{{\mathbb C}}
\newcommand\F{{\mathbb F}}
\newcommand\Q{{\mathbb Q}}
\newcommand\A{{\mathbb A}}
\newcommand\PPP{{\mathbb P}}
\newcommand\T{{\mathbb T}}
\newcommand\Z{{\mathbb Z}}

\newcommand\N{{\mathbb N}}
\newcommand\R{{\mathbb R}}

\newcommand\height{{\mathrm h}}

\newcommand\NS{{\mathrm {NS}\,}}

\newcommand\HH{{\mathcal H}}

\title{Heights on squares of modular curves}

\author{Pierre Parent \\
\ \\
with an appendix by Pascal Autissier}

1

\begin{document}

\maketitle

\begin{abstract}
We develop a strategy for bounding from above the height of rational points of modular curves with values in 
number fields, by functions which are polynomial in the curve's level. Our main technical tools come from effective 
Arakelov descriptions of modular curves and jacobians. We then fulfill this program in the following 
particular case:

  If~$p$ is a not-too-small prime number, let~$X_0 (p )$ be the classical modular curve of level $p$ over 
$\Q$. Assume Brumer's conjecture on the dimension of winding quotients of $J_0 (p)$. We prove that there is 
a function $b(p)=O(p^{5} \log p )$ (depending only on $p$) such that, for any quadratic number field $K$, 
the $j$-height of points in $X_0 (p ) (K)$ which are not lifts 
of elements of $X_0^+ (p) (\Q )$, is less or equal to~$b(p)$. 
 
\medskip

AMS 2000 Mathematics Subject Classification  11G18 (primary), 14G40, 14G05 (secondary). 
\end{abstract}

\tableofcontents

\section{Introduction}

Let $N$ be an integer, $\Gamma_N$ a level-$N$ congruence subgroup of $\GL_2 (\Z )$, and $X_{\Gamma_N}$ the 
associated modular curve over some subfield of $\Q (\mu_N )$ which, to simplify the discussion, we assume from now on to be $\Q$. 
The genus $g_N$ of $X_{\Gamma_N}$ grows roughly as a polynomial  function of $N$. So if $N$ 
is not too small, $X_{\Gamma_N}$ has only a finite number of rational points with values in any 
given number field, by Mordell-Faltings. If one is interested in explicitly determining the set of rational 
points however, finiteness is of course not sufficient; a much more desirable control would be 
provided by upper bounds, for some handy height, on those points. Proving such an ``effective Mordell'' is 
known to be an extremely hard problem for arbitrary algebraic curves on number fields.

In the case of modular curves, however, the situation is much better. Indeed, whereas the jacobian of a 
random algebraic curve should be a somewhat equally random simple abelian variety, it is well-known that the 
jacobian $J_{\Gamma_N}$ of $X_{\Gamma_N}$ decomposes up to isogeny into a product of quotient abelian
varieties defined by Galois orbits of newforms for $\Gamma_N$. Moreover, in many cases, a nontrivial part 
of those factors happen to have rank zero over $\Q$. Our rustic starting observation is therefore the 
following: if $J_{\Gamma_N ,e}$ is the ``winding quotient'' of $J_{\Gamma_N}$, that 
is the largest quotient $J_{\Gamma_N ,e}$  with trivial $\Q$-rank, and
$$ 
X_{\Gamma_N}  \stackrel{\iota}{\hookrightarrow} J_{\Gamma_N} \stackrel{\pi_e}{\twoheadrightarrow} 
J_{\Gamma_N ,e}
$$
is some  Albanese map from the curve to its jacobian followed by the projection to $J_{\Gamma_N ,e}$, then 
any rational point on $X_{\Gamma_N}$ has an image which is a torsion point (because rational) on 
$J_{\Gamma_N ,e}$, hence has $0$ normalized height. The pull-back of some invertible sheaf defining the 
(say) theta height on $J_{\Gamma_N ,e}$ therefore defines a height on $X_{\Gamma_N}$ which is trivial on 
rational points. That height in turn necessarily compares to any other natural one, for instance the modular 
$j$-height. Therefore the $j$-height of any rational point on $X_{\Gamma_N}$ is also zero ``up to error 
terms". Making those error terms explicit would give us the desired upper bound for the height of rational points on $X_{\Gamma_N}$. 

   That approach can in principle be generalized to degree-$d$ number fields, by considering rational points on symmetric powers $X_{\Gamma_N}^{(d)}$ of $X_{\Gamma_N}$ 
(at least if $\dim J_{\Gamma_N ,e} \geq d$). To be a little bit more precise in the present case of symmetric squares, let us associate to a quadratic 
point $P$ in $X_0 (p)$ the $\Q$-point 
$Q :=(P,{}^\sigma P)$ of $X_0 (p)^{(2)}$. Its image $\iota (Q)$ via some appropriate Albanese embedding in $J_0 (p)$ lies above a torsion point $a$ in $J_e$: assume for simplicity $a=0$. 
We therefore know $\iota (Q)$ belongs to the intersection of $\iota (X_0 (p)^{(2)} )$ with the kernel $\tilde{J}_e^\perp$ of the projection 
$$
\pi_e \colon J_0 (p)\twoheadrightarrow J_e .
$$
To improve the situation we can further remark that $\iota (Q)$ actually lies at the intersection of $\iota (X_0 (p)^{(2)} )$ with the ``projection'', in some appropriate sense, of the latter surface
on $\tilde{J}_e^\perp$. Then one can show that this intersection is $0$-dimensional (but here we need to assume Brumer's conjecture, see below) so that its theta height is 
controlled, via some arithmetic B\'ezout theorem, in terms of the degree and height of the two surfaces we intersect. Using an appropriate version of Mumford's repulsion principle one
derives a bound for the height of $\iota (P)$ too (and not only for its sum $\iota (Q)$ with its Galois conjugate).  Then one makes the translation again from theta height to $j$-height on $X_0 (p)$.

   Nontrivial technical work is of course necessary to give sense to the straightforward strategy sketched above. The aim of this article is thus to show the possibility of that  
approach,  by making it work in what we feel to be the simplest non-trivial case: that of quadratic points of the classical modular curve $X_0 (p)$ as above (or $X_0 (p^2 )$, for technical 
reasons), for $p$ a prime number\footnote{Larson and Vaintrob have proven, under the Generalized Riemann Hypothesis, the asymptotic 
triviality of rational points on $X_0 (p)$ with values in any given number field which does not contain the Hilbert class field of some quadratic imaginary   field (see \cite{LV14},
Corollary~6.5). Independently of any conjecture, Momose had already proven the same result in the case where $K$ is a given quadratic number field (\cite{Mo95}). 
Our method however provides bounds which do not depend on the field, and should generalize to some other congruence 
subgroups.}. In the course of the proof we are led to assume the already mentioned conjecture of Brumer, which asserts that the winding quotient of $J_0 (p):= J_{\Gamma_0 (p)}$ has dimension 
roughly half that of $J_0 (p)$. 
That hypothesis is actually used in only one, technical, but crucial place, where we prove that a morphism between two curves is a generic isomorphism (see last point of 
Lemma~\ref{afterBrumer}). Note that a lower bound of $1/4$ (instead of the desired $1/2$) for the asymptotic ratio $\dim J_e /\dim J_0 (p)$ has been proven by Iwaniec-Sarnak and 
Kowalski-Michel-VanderKam. (Actually, $(1/3 +\varepsilon )$ would be sufficient for us, see~Lemma~\ref{afterBrumer} and the proof of Theorem~\ref{theoheight} below.) 
In any case we cannot at the moment get rid of this assumption - note it can in principle be numerically checked in all specific cases.
In this setting, our main result is the following (see Theorem~\ref{theoheight}).

\begin{theo}   
\label{mainofintro}
For $w_p$ the Fricke involution, set $X_0^+ (p)  =X_0 (p)/w_p $.
Assume Brumer's conjecture (see Section~\ref{preliminaries}, (\ref{Brumer})\footnote{The weak version of that conjecture 
we actually need  is stated in~(\ref{TuM'asDitDeLeFaireMoinsFort}).}). Then the quadratic 
points of $X_0 (p )$, which are not lifts of elements of $X_0^+ (p) (\Q )$, have $j$-height bounded from 
above by $O (p^{5}  \log p)$. 

  The same holds true for quadratic points of $X_0 (p^2 )$, without the restriction about $X_0^+ (p)$. 
\end{theo}

   Needless to say, this result cries for both sharpening and generalization.
Yet it should be possible to immediately use avatars of Theorem~\ref{mainofintro} to prove that
rational points are only cusps and CM points, for some specific modular curves of arithmetic interest. If combined with lower 
bounds for heights furnished by isogeny theorems as in~\cite{BPR13}, the above theorem already has 
consequences on rational points (see Corollary~\ref{corolaBrum}).  

  Regarding past works about rational points on modular curves, one can notice that most of them use, at least in parts, some variants of Mazur's method, which 
can  very roughly be divided into two steps:  first, map modular curves to winding quotients as described above; then prove some quite delicate properties about 
completions of that map to $J_e$ (formal immersion criteria). The second step is probably the most difficult to carry over to great generality. The method 
we here propose therefore allows one to use only the first and crucial fact - the mere existence of nontrivial winding quotients. In many cases, the existence 
of such quotients is known to be a deep result of Kolyvagin-Logachev-Kato, \`a la Birch-Swinnerton-Dyer Conjecture which, again, seems to reflect, from the arithmetic point of view, the quite 
special properties of the image locus (in the moduli space of principally polarized abelian varieties) of modular curves, among all algebraic curves, 
under Torelli's map.

      The methods used in this paper are mainly explicit Arakelov techniques for modular curves and abelian varieties. Such techniques and results
have been pioneered, as far as we know, by Abbes, Michel and Ullmo at the end of the 1990s (see in particular \cite{AU95}, \cite{MU98} and \cite{Ul00},
whose results we here eagerly use). They have subsequently been revisited and extended in the work developed by Edixhoven and his school, as 
mainly (but not exhaustively) presented in the orange book~\cite{CE11}. That work was motivated by algorithmic Galois-representations 
issues, but its tools are well suited to our rational points questions, as we wish to show here. We similarly hope that the effective Arakelov results 
about modular curves and jacobians we work out in the present article shall prove useful in other contexts\footnote{For recent investigations related to more general questions 
of effective bounds of algebraic points on curves, one can check~\cite{CVV16}.}.

\medskip

  The layout of this article is as follows. In Section~2 we start gathering classical instrumental facts on quotients of modular jacobians and 
regular models of $X_0 (p)$ over rings of algebraic integers. In Section~3 we make a precise description of the arithmetic Chow group of $X_0 (p)$.
Section~4 provides an explicit comparison theorem between $j$-heights and pull-back of normalized theta height on the jacobian. 
Section~5 computes the degree and Faltings height of the image of symmetric products within modular jacobians. 
In Section~6 we prove our arithmetic B\'ezout theorem (in the sense of~\cite{BGS94}) for 
cycles in $J_0 (p)$, relative to cubist metrics
(instead of the more usual Fubini-Study metrics). This seems more natural, and has the advantage of being quantitatively much more efficient;
that constitutes the technical heart of the present paper. Then we apply that 
arithmetic B\'ezout to our modular jacobian after technical computations on metric comparisons. 
Section~7 concludes the computations of the height bounds for quadratic rational points on $X_0 (p )$
by making various intersections, projections and manipulations for which to refer to loc. cit.

\bigskip
   
\noindent  {\bf Convention.}
In order to avoid numerical troubles, we safely assume in all what follows that primes are by definition strictly larger than $17$.

\section{Curves, jacobians, their quotients and subvarieties}
\label{preliminaries}

\subsection{Abelian varieties}

\subsubsection{Decompositions}

Let $K$ be a field, $J$ an abelian variety of dimension $g$ over $K$ and ${\cal L}$ an ample invertible sheaf defining a
polarization of $J$. Assume $J$ is $K$-isogenous to a product of two (nonzero) subvarieties, that is, there are abelian subvarieties 
\begin{eqnarray}
\label{desimmersions}
\iota_A \colon A \hookrightarrow J, \  \iota_B \colon B \hookrightarrow J
\end{eqnarray}
endowed with the polarization $\iota_A^* ({\cal L})$ and $\iota_B^* (B)$ respectively, such that~$\iota_A +\iota_B  \colon A\times B\to J$ 
is an isogeny. (Recall that by convention here, all abelian (sub)varieties are assumed to be connected.) Then $\pi_A \colon J\to A' := J\  {\mathrm{mod}} \ B$, 
and similarly $\pi_B \colon J\to B'$, are called {\it optimal quotients} of $J$.

To simplify things we also assume from now on that~${\mathrm{End}}_{\overline{K}} (A,B)=\{ 0\}. $
The product isogeny~$\pi :=\pi_A \times \pi_B \colon J\to A' \times B'$ induces isogenies 
$A\to A'$ and  $B\to B'$. We write
$$
\Phi \colon A\times B \to J \to A' \times B' 
$$
for the obvious composition. Taking for instance dual isogenies of $A\to A'$ and $B\to B'$, we also define an endomorphism
\begin{eqnarray}
\label{Psi}
\Psi \colon J \to A' \times B' \to A\times B \to J.
\end{eqnarray}

When $K=\C$, the above constructions are transparent. There is a $\Z$-lattice $\Lambda$ in $\C^g$, endowed with a symplectic pairing,
such that $J(\C )\simeq \C^g /\Lambda$ and one can find a direct sum decomposition $\C^g =\C^{g_A} \oplus \C^{g_B}$ such that
if $\Lambda_A =\Lambda \cap \C^{g_A}$ and $\Lambda_B =\Lambda \cap \C^{g_B}$, then
$$
A(\C ) \simeq \C^{g_A} /\Lambda_A {\ {\mathrm{and}}\ } B(\C ) \simeq \C^{g_B} /\Lambda_B .
$$
If $p_A \colon \C^g \to \C^{g_A}$ and $p_B \colon \C^g \to \C^{g_B}$ are the $\C$-linear projections 
relative to that decomposition, the analytic description of $\pi_{A,\C} \colon J(\C ) \to A' (\C )$ is then
$$
z {\ \mathrm{mod}\ } \Lambda \mapsto z {\ \mathrm{mod}\ } (\Lambda +\Lambda_B \otimes \R ) =p_A (z) {\ \mathrm{mod}\ } (p_A (\Lambda )).
$$

Summing-up we have lattice inclusions: $\Lambda_A \subseteq p_A (\Lambda )$,  $\Lambda_B \subseteq p_B (\Lambda )$, with finite indices, 
in~$\C^g$ such that our isogenies are induced by
$$
\Lambda_A \oplus \Lambda_B \subseteq \Lambda \subseteq p_A (\Lambda )\oplus p_B (\Lambda ).
$$

The isogeny ${I'_A \colon A\to A' }$ deduced from the inclusion~$\Lambda_A \subseteq p_A (\Lambda )$ has 
degree~${\mathrm{card}} (p_A (\Lambda )/\Lambda_A )$. If~$N_A$ is a multiple of the exponent of the quotient~$p_A (\Lambda )/\Lambda_A$,
there is an isogeny ${{I}_{A,N_A} \colon  A' \to A}$ such that~${I}_{A,N_A} \circ I'_A$ and~$I'_A \circ {I}_{A,N_A}$ both 
are multiplication by~$N_A$. The analytic descriptions of the above clearly are: 
\begin{eqnarray}
\label{OnTangentSpaces}
\left\{
\begin{array}{rcl}
A(\C ) \simeq \C^{g_A} /\Lambda_A & \stackrel{I'_A}{\longrightarrow} &{A'} (\C )\simeq \C^{g_A} /p_A (\Lambda ) \\
z & \mapsto & z 
\end{array}
\right.    
\ {\mathrm{\ and\ }} \ 
\left\{
\begin{array}{rcl}
\C^{g_A} /p_A (\Lambda ) & \stackrel{{I}_{A,N_A}}{\longrightarrow} & \C^{g_A} /\Lambda_A \\
z & \mapsto & N_A z .
\end{array}
\right.    
\end{eqnarray}

\begin{remark}
\label{RemarkIndex}  {\rm Instead of considering two immersions as in~(\ref{desimmersions}), suppose only $A \hookrightarrow J$ is given,
and $K$ is a number field. One might apply~\cite{GR12}, Th\'eor\`eme~1.3, to deduce the existence of an abelian variety~$B$ over $K$ 
such that, with our previous notations, the degree of~$A \times B \stackrel{+}{\longrightarrow} J$:
$$
\vert A \cap B\vert =\vert \Lambda /\Lambda_A \oplus \Lambda_B \vert
$$ 
is bounded from above by an explicit function $\kappa (J)$ of the stable Faltings' height~$\height_F (J)$:
$$
\kappa (J)=\left( (14g)^{64 g^2} [K:\Q] \max (h_F (J), \log [K:\Q ] ,1)^2  \right)^{2^{10} g^3}
$$
and this does not depend on the choice of the embedding $K\hookrightarrow \C$. Note that 
when~$A$ and~$(J \, 
{\mathrm{mod}} \, A)$ are not isogenous (which will be the case for us), then there is actually no choice for that~$B\hookrightarrow J$: it 
has to be the Poincar\'e complement to~$A$. The isogeny~$J\to {A'} \times {B'}$ given by the two projections has 
degree~$\vert p_A (\Lambda )\oplus p_B  (\Lambda )/ \Lambda \vert$, which also is~$\vert A \cap B\vert :=N$. One can therefore take the~$N_A$ appearing
in~(\ref{OnTangentSpaces}) as equal to $N$, and
$$
N\le \kappa (J) .
$$
Making the same for ${B'} \to B$, the above morphism $\Psi$ (see~(\ref{Psi})) is then simply the multiplication $J\stackrel{[N\cdot ]}{\longrightarrow} J$ by 
the integer~$N$. Although we will not need numerical estimates for those quantities in what follows, it is straightforward, using~\cite{Ul00}, to make 
them explicit in our setting of modular curves and jacobians.} 
\end{remark}

\subsubsection{Polarizations and heights}
\label{PolarizationsAndHeights}

    Keeping the above notations and hypothesis, consider in addition now an ample sheaf $\Theta$ on $J$, and let {${I}_A := {I}_{A,N} \colon {A'} \to A$}
(respectively, ${I}_{B,N}$) be as in~(\ref{OnTangentSpaces}). We pull-back $\Theta$ along the composed morphism:
\begin{eqnarray}
\label{chain}
\varphi_A \colon J \stackrel{\pi_A}{\longrightarrow} {A'}  \stackrel{{I}_A}{\longrightarrow} A \stackrel{\iota_A}{\longrightarrow} J
\end{eqnarray}
so that the immersion~${\imath_A \colon A\hookrightarrow J}$ defines a 
polarization~$\Theta_A :=\imath_A^* (\Theta )$ on~$A$, whence a polarization $\Theta_{{A'}} :={{I}_A}^* (\Theta_A )$ on ${A'}$, and finally
an invertible sheaf $\Theta_{J,A} :=\pi_{A}^* (\Theta_{{A'}})$ on $J$. 
Composing the morphisms:
\begin{eqnarray}
\label{chain2}
J \stackrel{\pi_A \times \pi_B}{\longrightarrow} {A'} \times {B'} \stackrel{{I}_A \times {I}_B}{\longrightarrow} A\times B 
\stackrel{\iota_A + \iota_B}{\longrightarrow} J
\end{eqnarray}
gives the multiplication-by-$N$ map: $J \stackrel{ [\cdot N]}{\longrightarrow} J$. Assuming for simplicity $\Theta$ is symmetric one therefore has
\begin{eqnarray}
\label{produitsTheta}
[\cdot N]^* \Theta \simeq \Theta^{\cdot \otimes N^2} \simeq \Theta_{J,A} \otimes_{{\cal O}_J} \Theta_{J,B} .
\end{eqnarray}
If $K$ is a number field, the N\'eron-Tate normalization process associates with $\Theta$ a system of compatible Euclidean 
norms $\height_\Theta =\| \cdot \|^2_{\Theta}$ 
on the finite-dimensional $\Q$-vector spaces $J(F) \otimes_\Z \Q$, for $F/K$ running through the number field extensions of $K$, and similarly 
Euclidean norms $\height_{\Theta_A} := \| \cdot \|^2_{\Theta_A^{\cdot \otimes \frac{1}{N^2}}} :=\frac{1}{N^2} \| \cdot \|^2_{\Theta_A}$ on $A (F) \otimes_\Z \Q$ 
and $\height_{\Theta_B} 
:=   \frac{1}{N^2} \| \cdot \|^2_{\Theta_B}$ on $B (F) \otimes_\Z \Q$ such that, under the isomorphisms $J(F) \otimes_\Z \Q \simeq \left( A(F) \otimes_\Z \Q 
\right) \oplus 
\left( B(F) \otimes_\Z \Q \right)$, one has
\begin{eqnarray}
\label{decomposeHeights}
\height_\Theta =\height_{\Theta_A} +\height_{\Theta_B} .
\end{eqnarray}

Recall from~(\ref{OnTangentSpaces}) the definition of $N_A$, that of the maps $A' \stackrel{{I}_{A ,N_A}}{\longrightarrow} A$ and 
$A\stackrel{\iota_A}{\hookrightarrow} J$.  Denote by $[N_A ]_A$ the multiplication by $N_A$ restricted to $A$. If $V$ is a closed algebraic subvariety of $J$, 
define 
\begin{eqnarray}  
\label{pseudoprojection}
{\cal P}_A (V) := \left( \iota_A [N_A ]_A^{-1} {I}_{A ,N_A} \pi_A  \right) (V)
\end{eqnarray}
as the reduced closed subscheme with relevant support. That map ${\cal P}_A$ would simply be the projection of $V$ on $A$ if $J$ were {\em isomorphic} to the 
product $A\times B$ of subvarieties, and is the best approximation to that projection in our case when $J$ is only isogenous to $A\times B$.

Note that ${\cal P}_A (V)$ is a priori highly non-connected.  All its irreducible geometric components are however obtained from each other by translation 
by a $N_A$-torsion point of $A(\overline{\Q})$. For our later purposes (see proof of Theorem~\ref{theoheight}), we will have the possibility to 
replace ${\cal P}_A (V)$ by one of its components containing a specific point, say $P_0$: we shall denote that component by ${\cal P}_A (V)_{P_0}$, and
refer to it as the ``pseudo-projection'' of $V$ on $A$ containing $P_0$.

\medskip

   Suppose now $J\sim A\times B$ as above is the jacobian of an algebraic curve~$X$ on~$K$ with positive genus~$g$. For $P_0$ a point of $X(K)$ (or 
more generally a $K$-divisor of degree $1$ on $X$) let
\begin{eqnarray}
\label{albanese}
\imath_{P_0} \colon
\left\{
\begin{array}{rcl}
X & \hookrightarrow & J \\
   P & \mapsto &  (P)-(P_0 )
\end{array}
\right.
\end{eqnarray}
be the Albanese embedding associated with~$P_0$. We define the classical Theta divisor~$\theta$ on~$J$ which is the image of~$\imath_{P_0}^{g-1} 
\colon X^{g-1} \to J$ and its symmetric version
\begin{eqnarray}
\label{symmetricTheta}
\Theta :=\left( \theta \otimes_{{\cal O}_J} [-1]^* \theta \right)^{\cdot \otimes \frac{1}{2}} 
\end{eqnarray}
(which is a translate of $\theta$ obtained as $\imath_{\kappa_0}^{g-1} (X^{g-1})$, where $\imath_{\kappa_0} =t_{\kappa_0}^* \imath_{P_0}$ for $t_{\kappa_0}$ the translation by
some $\kappa_0$ with $(2g-2)\kappa_0 =\kappa$: the  canonical divisor on $X$. Of course $\Theta$ does not need to be defined over $K$). 
Our first aim will be to compare the height functions $\| \imath_{P_0} (\cdot )\|_{{\Theta_A}^{\cdot \otimes\frac{1}{N^2}}}$ on $X(F)$, when $X$ is a modular curve, 
with another natural height given by the modular $j$-function.

     We will discuss in Section~\ref{Chow} an Arakelov description of N\'eron-Tate height. We 
conclude this paragraph by a few remarks  as a preparation. Let  ${B_2} :=\{ \omega_1 , \dots ,
\omega_g \}$ be a basis of $H^0 (X(\C ),\Omega^1_{X/\C}) \simeq H^0 (J(\C ),\Omega^1_{J/\C} )$, which is
orthogonal with respect to the norm 
$$
\| \omega \|^2 =\frac{i}{2} \int_{X(\C )} \omega \wedge \overline{\omega} .
$$
The transcendent writing-up of the Abel-Jacobi map $\iota_{P_0} \colon P\mapsto (\int_{P_0}^P \omega_i )_{1\leq i\leq g}$ shows that the pull-back to $X(\C )$ of the 
translation-invariant measure on $J(\C )$, normalized to have total mass $1$, is 
\begin{eqnarray}
\label{flat}
\mu_0 =\frac{i}{2g} \sum_{B_2} \frac{\omega \wedge \overline{\omega}}{\| \omega \|^2} .
\end{eqnarray}

    More generally, $\pi_A \circ \iota_{P_0}$ is, over $\C$, the map $P\mapsto (\int_{P_0}^P \omega )_{\omega \in {B}_2^A}$, 
where ${B}_2^A$ is some orthogonal basis of $H^0 ({A'} (\C ),\Omega^1_{{A'}/\C} )\simeq H^0 (J(\C ),\pi_A^* (\Omega^1_{{A'}/\C}  )) \subseteq  H^0 (J(\C ),
\Omega^1_{J/\C})$. Therefore, writing $g_A :=\dim ({A'})=\dim (A)$ (we assume $A\neq 0$), the pull-back to $X(\C )$ of the translation-invariant measure 
on ${A'}(\C )$ (normalized so to have total mass $1$ on the curve again) is 
\begin{eqnarray}
\label{flat_A}
\mu_A =\frac{i}{2g_A} \sum_{{B}_2^A} \frac{\omega \wedge \overline{\omega}}{\| \omega \|^2} .
\end{eqnarray}

\subsection{Modular curves}
\label{ModularCurves}

    Here we recall a few classical facts on the minimal regular model of the modular curve $X_0 (p)$, for $p$ a prime number, over a ring of algebraic
integers. The first general reference on this topic is~\cite{DR73};  see also \cite{CE11} or \cite{Me08}, \cite{Me11b}.

\subsubsection{The $j$-height}

  The quotient of the completed Poincar\'e upper half-plane $\HH \cup \PPP^1 (\Q )$ by the classical congruence subgroup $\Gamma_0 (p)$ defines
a Riemann surface $X_0 (p)(\C )$ which is known to have a geometrically connected smooth and proper model over $\Q$. All through  this paper, we
denote its genus by~$g$.

   The first technical theme of this article is the explicit comparison of various heights on $X_0 (p)({\overline{\Q}})$. When $V$ is an algebraic 
variety over a number field $K$, any finite $K$-map $\varphi \colon V\to \PPP^N_K$ to some projective space %
defines a naive Weil height on $V(\overline{K})$. This applies in particular when $V$ is a curve and $\varphi$ is the finite morphism defined
by an element of the function field of $V$, and in the case of a modular curve $X_\Gamma$ associated with some congruence 
subgroup $\Gamma$, say, a natural 
height to choose on~$X_\Gamma (\overline{\Q})$ is precisely Weil's height $\height (P)= \height (j(P) )$ relative to the classical~$j$-function. 
The degree of the associated map $X_\Gamma \to X (1)\simeq \PPP^1$ is~$[{\mathrm{PSL}}_2 (\Z ) :\Gamma ]$, so that number is the class of 
our Weil height in the N\'eron-Severi group $\NS (X_\Gamma )$ identified with $\Z$. More explicitly if~$X =X_\Gamma$ is defined over the number 
field~$K$, say, the $j$-morphism is
$$
\left\{
\begin{array}{rcrcl}
X & \stackrel{\jmath}{\to} & \PPP^1_K ={\mathrm{Proj}} (K[X_0 ,X_1 ]) & \hookleftarrow {{\A^1_K}}  & ={\mathrm{Spec}}  (K[X_1 /X_0 ] ) \\ 
P & \mapsto & (1,j(P)) =(1/j(P) ,1) & \leftarrow j(P) &=\frac{X_1}{X_0} (P) ,
\end{array}
\right.
$$  
and the Weil height of a point~$P\in X(K)$ is therefore the naive height of its~$j$-invariant as an algebraic number: 
$$
\height (P)=\height (j(P)) =\frac{1}{[K:\Q ]} \sum_{v\in M_K} [K_v :\Q_v ] \log ({\mathrm{max}} (1,\vert j(P)\vert_v ))
$$
which is also Weil's projective height $\height (\jmath (P))$ with respect to the above basis~$(X_0 ,X_1 =X_0 j)$ 
of global sections of ${\cal O}_{\PPP^1_K} (1)$. Our Weil height on~$X$ is  associated with the linear equivalence 
classes of divisors~$D$ corresponding to~$\jmath^* ({\cal O}_{\PPP^1_K} (1))$, so that
$$
D \sim ({\mathrm{poles\ of\, }}j{\mathrm{\, on\, }X})(\sim ({\mathrm{zeroes\ of\ }}j))\sim \sum_{c\in \{ 
{\mathrm{cusps\, of\, }X\} }} e_c .c 
$$
where each~$e_c$ is the ramification index of~$c$ via~$\jmath$. 

  Those considerations lead to explicit comparisons with other heights. Indeed, a more intrinsic way to define 
heights on algebraic varieties is provided by Arakelov theory. Defining this properly in the case of our 
modular curves demands a precise description of regular models for them, which we now recall.

\subsubsection{Regular models}
\label{minimalregular}
  
    The normalization of the $j$-map $X_0 (p) \to X(1)_{/\Z} \simeq \PPP^1_{/\Z}$ over  $\Z$ defines a model for $X_0 (p)$, that we call the modular model,
it is smooth over $\Z [1/p]$.  

   We fix a number field $K$, write ${\cal O}_K$ for its ring of integers, and deduce by base change a model for $X_0 (p)$ over ${\cal O}_K$. We know
its only singularities are normal crossing, so after a few blow-ups if necessary we obtain a regular model of $X_0 (p)$ over ${\cal O}_K$:
see Theorem 1.1.d) of the Appendix of~\cite{Ma77}. 
We denote it from now on by~${\cal X}_0 (p)_{/{\cal O}_K}$, or simply ${\cal X}_0 (p)$ if the context prevents confusion. We stress here that  for $F/K$ 
a field extension, ${\cal X}_0 (p)_{/{\cal O}_F}$ is {\it not} the base change to ${\cal O}_F$ of ${\cal X}_0 (p)_{/{\cal O}_K}$ if $F/K$ ramifies above $p$. 
Let $v$ be a place of ${\cal O}_K$ above $p$, with residue field $k(v)$. The dual graph of~${\cal X}_0 (p)$ at $v$ is made of two extremal vertices, which 
we label~$C_0$ and~$C_\infty$, containing the cusps~$0$ and~$\infty$ respectively (see Figure~1). Those two vertices, which correspond to irreducible
components of genus $0$, are linked by
$$
s:=g+1
$$ 
branches. Each branch corresponds to a singular point $S$ in~${\cal X}_0 (p) (\F_{p^2} )$, which in turn parameterizes an 
isomorphism class of supersingular elliptic curve~$E_S$ in characteristic~$p$. 

   The Fricke involution $w_p$ acts on the dual graph as the continuous
isomorphism which exchanges $C_0$ and $C_\infty$ and acts on the branches as a generator of ${\mathrm{Gal}} (\F_{p^2} /\F_p )$.

We list the supersingular points as $S(1),\dots S(s)$  and for each one define
\begin{eqnarray}
\label{doubleVn}
w_n :=\# {\mathrm{Aut}} (S(n)) /\langle \pm 1\rangle :=\# {\mathrm{Aut}}_{\F_{p^2}}  (E_{S(n)} ) /\langle \pm 1\rangle 
\end{eqnarray}
which is equal to $1$ except in the (at most two) cases when the underlying supersingular elliptic curve has $j$-invariant $1728$ or $0$, 
where it is equal to $2$ or $3$ respectively. Now each path, or branch, on our dual graph at~$v$ passes through~$(w_n e-1)$ vertices (for~$e$ the ramification 
index of $K$ at~$v$), that is, again, equal to $e-1$ except for at most two branches: one of length~${2e-1}$ (obtained by blowing-up 
the supersingular point of moduli~${j=1728\ {\mathrm{mod}\  v}}$, if it exists), and a path of length~${3e-1}$  (obtained by blowing-up, if needed, at the 
supersingular point of moduli~$j=0\ {\mathrm{mod} \ v}$). We enumerate the vertices~$\{ C_{n,m} \}_{1\leq m
\leq w_n e-1}$ in the $n^{\mathrm{th}}$ path.  
We also denote by $w({\mathrm{Eis}} )$ the familiar quantity $\sum \frac{1}{w_n}$, the sum being taken over 
the set of all supersingular points of~${\cal X}_0 (p)_{/{\cal O}_{K,v}}$. The well-known Eichler mass formula says that
\begin{eqnarray}
\label{w(Eis)}
w({\mathrm{Eis}} ) =\sum_{1\le n\le s} \frac{1}{w_n} = \frac{p-1}{12} 
\end{eqnarray}
(see for instance~\cite{Gr87}, p.~117). Recall this implies that the genus~$g$ of $X_0 (p)$ is asymptotically equivalent 
to $p/12$ (the exact formula depending on the residue class of $p$ mod $12$) and in any case:
\begin{eqnarray}
\label{genussize}
\frac{p-13}{12} \leq g \leq \frac{p+1}{12} 
\end{eqnarray}
(see for instance p.~117 of~\cite{Gr87} again).


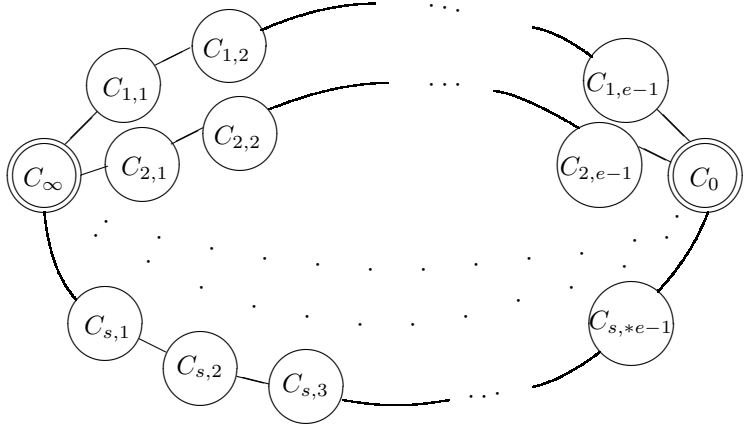
\begin{figure}
\label{figure}
\begin{center}
\begin{picture}(320,100)(-30,-80)

\put(-8,-4){$C_\infty$}
\put(0,0){\circle{25}}
\put(0,0){\circle{29}}


\put(8,12){\line(1,1){12}}

\put(22,30){$C_{1,1}$}
\put(30,34){\circle{28}}

\put(42,42){\line(2,1){13}}

\put(60,45){$C_{1,2}$}
\put(70,49){\circle{28}}

\qbezier(82,55)(105,65)(125,65)

\put(145,62){$\cdot$}
\put(150,61){$\cdot$}
\put(155,60){$\cdot$}

\qbezier(185,56)(195,54)(207,45)

\put(205,32){$C_{1,e-1}$}
\put(220,36){\circle{30}}

\put(232,25){\line(1,-1){12}}


\put(14,0){\line(3,1){10}}

\put(29,0){$C_{2,1}$}
\put(37,4){\circle{28}}

\put(48,12){\line(2,1){12}}

\put(64,12){$C_{2,2}$}
\put(74,16){\circle{28}}

\qbezier(85,25)(110,35)(130,35)

\put(145,32){$\cdot$}
\put(150,32){$\cdot$}
\put(155,32){$\cdot$}

\qbezier(170,32)(180,31)(202,18)

\put(195,0){$C_{2,e-1}$}
\put(210,4){\circle{30}}

\put(225,11){\line(2,-1){12}}


\qbezier(0,-14)(2,-36)(12,-47)

\put(15,-60){$C_{s,1}$}
\put(23,-56){\circle{28}}

\put(36,-62){\line(2,-1){10}}

\put(50,-75){$C_{s,2}$}
\put(59,-73){\circle{28}}

\put(73,-76){\line(4,-1){12}}

\put(90,-82){$C_{s,3}$}
\put(99,-80){\circle{28}}

\qbezier(113,-85)(133,-89)(153,-85)

\put(160,-86){$\cdot$}
\put(165,-86){$\cdot$}
\put(170,-85){$\cdot$}

\qbezier(185,-80)(195,-78)(210,-67)

\put(206,-58){$C_{s,*e-1}$}
\put(222,-56){\circle{32}}

\qbezier(232,-44)(245,-30)(251,-14)


\put(22,-20){$\cdot$}
\put(42,-26){$\cdot$}
\put(62,-30){$\cdot$}
\put(82,-34){$\cdot$}
\put(102,-36){$\cdot$}
\put(122,-38){$\cdot$}
\put(142,-38){$\cdot$}
\put(162,-36){$\cdot$}
\put(182,-34){$\cdot$}
\put(202,-32){$\cdot$}
\put(222,-27){$\cdot$}
\put(238,-18){$\cdot$}

\put(18,-25){$\cdot$}
\put(38,-35){$\cdot$}
\put(58,-45){$\cdot$}
\put(78,-52){$\cdot$}
\put(98,-56){$\cdot$}
\put(118,-59){$\cdot$}
\put(138,-59){$\cdot$}
\put(158,-56){$\cdot$}
\put(178,-50){$\cdot$}
\put(198,-45){$\cdot$}
\put(218,-37){$\cdot$}
\put(234,-26){$\cdot$}


\put(244,-4){$C_0$}
\put(250,0){\circle{25}}
\put(250,0){\circle{29}}

\put(-8,5){}

\end{picture}
\end{center}
\caption{Dual graph of ${\cal X}_0 (p)_{/{\cal O}_K}$ at $v$.}
\label{faut}
\end{figure}


  Abusing a bit notations, $C_\infty$ will sometimes be also denoted as~$C_{n,0}$, and similarly~$C_0$ 
might be written as~$C_{n,w_n e}$. We choose as a basis for~$\oplus_C \Z \cdot C$ the ordered set
\begin{eqnarray}
\label{basis}
{\cal B}= \left( C_\infty ,(C_{1,1} , C_{1,2} , \cdots ,C_{1,e-1} ) ,( C_{2,1} ,\cdots ,C_{2,e-1} ) , \dots ,(C_{s,1} ,\cdots ,C_{s,w_s e -1} ), C_0\right)
\end{eqnarray}
(that is, we enumerate the vertices by running through each branch successively, and put the possible branches of length twice or thrice the generic length at the end).
At bad places~$v$ the intersection matrix restricted to each submodule~$\oplus_{m=1}^{w_n e-1}  \Z \cdot C_{n,m}$ (for some fixed branch of index~$n$) is 
then~$(\log (\# k(v)) \cdot {\cal M}_0 $, where
\begin{eqnarray}
\label{M_0}
{\cal M}_0 =
\left(
\begin{array}{cccccc}
-2 & 1 & 0 & 0 & \cdots  & 0  \\ 
1 & -2 & 1 & 0 & \cdots  & 0 \\
0 & 1 & -2 & 1 & \cdots & 0  \\
\vdots &   &   \ddots & \ddots & \ddots  &\vdots  \\
 \vdots  &   &   & 1 & -2  & 1 \\
0 & 0 & \cdots  & 0 & 1 & -2
\end{array}
\right) ,
\end{eqnarray}
whose only dependence on~$n$ is that its type is $(w_n e -1)\times (w_n e -1)$. That matrix has determinant $(-1)^{w_n e -1} w_n e$. 
Define the row vectors: 
$$
L :=(1 \  0 \  0 \  \cdots \  0     ),  \  L' := (0 \  0 \  0 \  \cdots \  1     )
$$
(with length implicitly defined by the next lines) and the transpose column vectors:
$$
V:= L^t ,\ V' :={L'}^t .
$$
The intersection matrix on the whole space $\Z^{\cal B}$ is finally~$(\log (\# k(v)) \cdot {\cal M} )$ for
\begin{eqnarray}
\label{M}
{\cal M} =
\left(
\begin{array}{cccccc}
-s & L                &   L                & \cdots & L & 0 \\
V & {\cal M}_0 &  0                  & \cdots  & 0 & V' \\
V & 0                &  {\cal M}_0  & \cdots  & 0 & V' \\
\vdots    &   \vdots       &    \vdots     &        \ddots              &  \vdots   &   \vdots     \\
V   & 0                 & 0          &   \cdots                               & {\cal M}_0 & V' \\
0 & L'                      &   L'                &   \cdots    &   L'  &  -s    
\end{array}
\right)   .
\end{eqnarray}
(This has to be modified in the obvious way when $e_v =1$.)

\subsubsection{Winding quotients, their dimension}
\label{Winding}
    
    We denote as usual the jacobian of $X_0 (p)_{\Q}$ by $J_0 (p)$. As follows from section~\ref{minimalregular}, ${\cal X}_0 (p)$ is semistable over $\Z$, 
and the neutral component of the N\'eron model ${\cal J}_0 (p)$ of $J_0 (p)$ is  a semi-abelian scheme over $\Z$ (and an abelian scheme over $\Z [1/p]$).
Its neutral component represents the neutral component ${\mathrm{Pic}}^0_\Z ({\cal X}_0 (p))$ of the relative Picard functor of ${\cal X}_0 (p)$
over $\Z$.

   We know from Shimura's theory that the natural decomposition of cotangent 
spaces into Hecke eigenspaces induces a corresponding decomposition over $\Q$ of abelian varieties up to isogenies:
\begin{eqnarray}
\label{Shimuradecomposition}
J_0 (p) \sim \prod_{f\in B_2/{\mathrm{Gal}} (\overline{\Q}/\Q )} J_f 
\end{eqnarray}
indexed by Galois orbits in some set $B_2$ of newforms. A first useful sorting of this decomposition comes from the sign of
the functional equations for the $L$-functions of eigenforms $f$, that is, whether $w_p (f)$ equals $f$ or $-f$. One accordingly writes $J_0 (p)^-$
for the optimal quotient abelian variety associated with $\prod_{f, w_p (f)=-f} J_f$ in~(\ref{Shimuradecomposition}), and similarly $J_0 (p)^+$, so that
$J_0 (p)^- =J_0 (p)/(1+w_p )J_0 (p)$ and $J_0 (p)^+ =J_0 (p)/(1-w_p )J_0 (p)$. One
knows that
$$
\dim J_0 (p)^- =(\frac{1}{2} +o(1)) \dim J_0 (p)
$$ 
(see e.g.~\cite{Ro01}, Lemme~3.2).

    A more subtle object is  the {\it winding quotient} $J_e$, defined as the optimal quotient of $J_0 (p)$ corresponding to $\prod_{f, L(f,1)\neq 0} J_f$ in 
decomposition (\ref{Shimuradecomposition}). One can write 
\begin{eqnarray}
\label{defineWinding}
J_e =J_0 (p)/I_e J_0 (p)
\end{eqnarray} 
for some ideal $I_e$ of the Hecke algebra $\T_{\Gamma_0 (p)}$. Similarly, $J_e^\perp =J_0 (p)/I_e^\perp J_0 (p)$ will denote the optimal quotient corresponding 
to $\prod_{f, L(f,1)= 0} J_f$. 
For obvious reasons regarding signs of functional equations, $J_e$ is contained in $J_0 (p)^-$. But more is  expected: in line with the principle that
``{the vanishing order of a (modular) $L$ functions at the critical point should generically be as small as allowed by parity}'', 
Brumer (\cite{Br95}) conjectured that, as $p$ tends to infinity,
\begin{eqnarray}
\label{Brumer}
({\bf {?}})   \hspace{1cm}  \dim J_e =(1 -o(1)) \dim J_0 (p)^-  . \hspace{1cm}   ({\mathrm{Brumer}})   
\end{eqnarray}
Equivalently, it is conjectured that $\dim J_e =(\frac{1}{2} +o(1)) \dim J_0 (p)$, or that the dimensions of $J_e$ and $J_e^\perp$ 
should be, asymptotically in $p$, of equal size.  Note that~(\ref{Brumer}) above is also implied by the ``Density Conjecture'' of~\cite{ILS00}, p.~56 et seq., 
see also Remark~F on p.~65\footnote{Quoting Olga Balkanova (private communication), {\it ``Theorem 1.1 in \cite{ILS00} is proved for the test 
function $\phi$, whose Fourier transform is supported on the interval $[-2,2]$. The density conjecture claims that the same results are true 
without restriction on Fourier transform of $\phi$, see formula 1.9 [of loc. cit.].''}}. Actually, what we eventually need in this article (see Section~\ref{7}) is a 
weaker form of (\ref{Brumer}), which is: 
\begin{eqnarray}
\label{TuM'asDitDeLeFaireMoinsFort}
({\bf {?}})   \hspace{1cm}    \dim J_e > \frac{\dim J_0 (p)}{3} +\frac{2}{3} 
\end{eqnarray}
for large enough $p$. An important theorem of Iwaniec-Sarnak and Kowalski, Michel and Vanderkam asserts something nearly as good,  that is :
\begin{eqnarray}
\label{KM}
(\frac{1}{4}  -o(1) )\dim J_0 (p)  \le \dim J_e (\leq (\frac{1}{2} +o(1)) \dim J_0 (p)  )
\end{eqnarray}
as~$p$ goes to infinity (so that $(\frac{1}{2} -o(1)) \dim J_0 (p)  \le \dim J_e^\perp  \leq (\frac{3}{4} +o(1))\dim J_0 (p) $,  see 
\cite{IS00}, Corollary~13 and \cite{KMK00}). Breaking that $\frac{1}{4}$ is known to be closely linked to the Landau-Siegel zero problem.
Assuming the Generalized Riemann Hypothesis for $L$-functions of modular forms, Iwaniec, Luo and Sarnak prove one can 
improve $\frac{1}{4}$ to $\frac{9}{32}$ (\cite{ILS00}, Corollary~1.6, (1.54))... That seems to be all for the moment.  

\medskip 
 
     The central object of this paper will eventually be the maps 
$$
X_0 (p)^{(d)} \to J_0 (p) \to  J_e 
$$   
from symmetric products of $X_0 (p)$ (mainly the curve itself and its square) to the winding quotient. 

\bigskip

\section{Arithmetic Chow group of modular curves}
\label{Chow}

   We now give a description of the Arakelov geometry of $X_0 (p)$, relying on the work of many people: that topic has  been pioneered 
by Abbes, Ullmo and Michel (\cite{AU95}, \cite{MU98}, \cite{Ul00}) and notably developed by Edixhoven, Couveignes and their coauthors (see~\cite{CE11}). 
We shall also use the work of Bruin (\cite{Br14}), Jorgenson-Kramer (\cite{JK06}) and Menares  (\cite{Me08}, \cite{Me11b}) 
among others. We refer to those articles and their bibliography for general facts on  Arakelov theory (see~\cite{Ch86}, \cite{EJ11a}).

    Let ${\cal X}$ be any regular and proper arithmetic surface over the integer ring ${\cal O}_K$ of a number field $K$. Fixing in general 
smooth hermitian metrics~$\mu$ on the base changes of $\cal X$ to $\C$, it follows from the basics of Arakelov theory that for any horizontal 
divisor $D$ on ${\cal X}$ over ${\cal O}_K$ there are Green functions $g_{\mu ,D}$ on each Archimedean completion~${\cal X} (\C )$ satisfying 
the differential equation 
$$
\Delta g_{\mu ,D} = -\delta_D + \deg (D) \mu 
$$   
for~$\Delta =\frac{1}{i\pi} \partial \overline{\partial}$ the Laplace operator and $\delta_D$ the Dirac distribution relative to $D_{\C}$ on ${\cal X} (\C )$. 
The function $g_{\mu ,D}$  is integrable on the compact Riemann surface ${\cal X} (\C )$ endowed with its measure $\mu$, and uniquely determined 
up to an additive constant which is often fixed by imposing the normalizing condition that 
\begin{eqnarray}
\label{normalizing}
\int_{{\cal X} (\C )} g_{\mu ,D} \mu =0.
\end{eqnarray}   
When the horizontal divisor $D$ is a section $P_0$ in ${\cal X}({\cal O}_K )$, one will sometimes also use the notation $g_\mu ({P_0},z)$ 
for $g_{\mu ,P_0} (z)$. The Green functions relative to fixed smooth $(1,1)$-forms $\mu$ allow one to define an Arakelov intersection product relative 
to the $\mu$, which will be denoted by $[\cdot ,\cdot ]_\mu$, or $[\cdot ,\cdot ]$ if there is no ambiguity about the implicit form. In particular
the index will often be dropped for divisors intersections  of which one at least is vertical, where the choice of $\mu$ does not intervene.

  We shall denote by~$\mu_0$ the canonical Arakelov $(1,1)$-form on the Riemann surface ${\cal X} (\C )$ (assumed to have positive genus), inducing 
the ``flat metric''. It corresponds to the pullback, by any Albanese morphism ${\cal X} (\C) \to {\mathrm{Jac}} ({{\cal X}}_K )(\C )$, of 
the ``cubist"  metric in the sense of Moret-Bailly~(\cite{MB85}, more about this shortly) on the jacobian~${\mathrm{Jac}} ({{\cal X}}_K )$, associated with 
the N\'eron-Tate normalized height~$\height_{\Theta}$. 

   We now specialize to the case of ${\cal X}_0 (p)$ as in Section~\ref{ModularCurves}. If~$f$ is a modular form of weight 2 for~$\Gamma_0 (p)$, 
let~$\| f\|^2$ be its Petersson norm. Because newforms are orthogonal in prime level we have, as in (\ref{flat}):
\begin{eqnarray}
\label{mu}
\mu_0 :=\frac{i}{2 \dim (J_0 (p))} \sum_{f \in B_2} \frac{f\frac{dq}{q} \wedge \overline{f\frac{dq}{q}} }{\| f\|^2 } .
\end{eqnarray}

   We shall also need to consider N\'eron-Tate heights $\height_A$ for subabelian varieties $A\hookrightarrow J_0 (p)$
as in section~\ref{PolarizationsAndHeights} (recall $A\neq 0$). The associated $(1,1)$-form $\mu_A$ is given by~(\ref{flat_A}). More specifically, we focus on   
$\height_{\Theta_e}$ on ${J}_e$ (as in~(\ref{decomposeHeights}) and around, for $A'=J_e$)
which induces a height $\height_{\Theta_e} \circ \iota_{e,P_0}$ on $X_0 (p)$ via the map $\iota_{e,P_0} \colon X_0 (p)\hookrightarrow J\twoheadrightarrow {J}_e$. 
The curvature form of the hermitian sheaf on $X_0 (p)$ defining the Arakelov height associated with $\height_{\Theta_e} \circ \iota_{e,P_0}$  is  
\begin{eqnarray}
\label{mu_e}
\mu_e := \frac{i}{2\dim (J_e )} \sum_{f \in B_2 [I_e ]} \frac{f\frac{dq}{q} \wedge \overline{f\frac{dq}{q}} }{\| f\|^2 } .
\end{eqnarray}
where $B_2 [I_e ]$ stands for the set of newforms killed 
by the ideal $I_e$ defining $J_e$ as in~(\ref{defineWinding}).

\begin{remark}
\label{wpinvariance}  {\rm  Notice that both $\mu_0$ and $\mu_e$, or any $\mu_A$ above, are invariant by pull-back $w_p^*$ by the Fricke involution. In particular
the Arakelov intersection products $[\cdot ,\cdot ]_{\mu_0}$ and  $[\cdot ,\cdot ]_{\mu_e}$, relative to $\mu_0$ and $\mu_e$ respectively, 
are $w_p$-invariant. The latter was clear already from the fact that, more generally, $w_p$ is an orthogonal symmetry on $J_0 (p)$ endowed with its quadratic 
form $\height_\Theta$, which respects the orthogonal decomposition $\prod_{f } J_f $ of~(\ref{Shimuradecomposition}).}
\end{remark}

   One can now specialize the Hodge index theorem to our modular setting (see~\cite{Me11b}, Theorem~4.16, \cite{Me08}, 
Theorem~3.26, or more generally~\cite{MB85}, p.~85 et seq.):
\begin{theo}
\label{Hodge}
Let~$K$ be a number field, $\mu$ be a smooth non-zero $(1,1)$-form on $X_0 (p)(\C )$ as given in~(\ref{flat_A}), and $\widehat{CH}(p)_{\R, \mu}^{\mathrm{num}}$ be the arithmetic 
Chow group with real coefficients up to numerical equivalence of~${\cal X}_0 (p)$ over~${\cal O}_K$, relative to $\mu$. Denote by~$\infty$ the horizontal 
divisor defined by the $\infty$-cusp on~${\cal X}_0 (p)$ over $\Z$ (which is the Zariski closure of the $\Q$-point $\infty$ in $X_0 (p)(\Q )$), compactified 
with the normalizing condition~(\ref{normalizing}). Write~$\R \cdot X_\infty$ for the line of divisors with real coefficients supported on some fixed full 
vertical fiber~$X_\infty$. Define, for all~$v\in {\mathrm{Spec}} ({\cal O}_K )$ above~$p$, the $\R$-vector space:
$$
G_v :=\bigoplus_{C\neq C_\infty} \R \cdot C
$$
where the sum runs through all the irreducible components of~${\cal X}_0 (p)\times_{{\cal O}_K}  k(v)$ except~$C_\infty$ (the one 
containing~$\infty (k(v))$). 
Identify finally $J_0 (p)  (K )/{\mathrm{torsion}}$ with the subgroup of divisor classes $D_0$ which are compactified under the normalizing condition
$g_{D_0} (\infty) =0$ (which is therefore {\em different} from~(\ref{normalizing})).  
One  has a decomposition:
\begin{eqnarray}
\label{HHodge}
\widehat{CH}(p)_{\R, \mu}^{\mathrm{num}} = (\R \cdot \infty \oplus \R \cdot X_\infty ) \oplus_{v\vert p}^\perp G_v \oplus^\perp \left( J_0 (p)  (K ) 
\otimes \R \right)
\end{eqnarray}
where the ``$\oplus^\perp$" mean that the direct factors are mutually orthogonal with respect to the Arakelov intersection product. Moreover, the 
restriction of the self-intersection product to $J_0 (p)  (K ) \otimes \R$ coincides with twice the opposite of the N\'eron-Tate pairing.
\end{theo}
\begin{proof}
The proof can be immediately adapted from that of~\cite{Me11b}, Theorem 4.16, for  $L^1_2$-admissible measures (a setting allowing to define convenient 
actions of the Hecke algebra on the Chow group). For further computational use we recall how one decomposes divisors 
in practice. %
   Take~$D$ in~$\widehat{CH}(p)_{\R, \mu}^{\mathrm{num}}$, with degree~$d$ on the generic fiber. There is a vertical divisor~$\Phi_D$, with 
support in fibres above places of bad reduction (that is, of characteristic $p$), such that~$(D-d\infty -\Phi_D )$ has a real multiple which belongs to the neutral 
component ${\mathrm{Pic}}^0 ({\cal J}_0 (p))_{/{\cal O}_K }$. That~$\Phi_D$ is well-defined up to multiple of full vertical fibres, so we can 
assume~$\Phi_D$ belongs to~$\oplus^\perp G_p$ (and is then unambiguously defined). One associates to~$(D-d\infty -\Phi_D)\in \R \cdot {\cal J}^0_0 (p) 
({\cal O}_K )$ an element~$\delta$ in~$\widehat{CH}(p)_{\R, \mu}^{\mathrm{num}}$ by imposing a compactification such that~$[\infty ,\delta ]_\mu  =0$. 
The general Hodge index theorem (see for instance~\cite{MB85}) then finally asserts that~$(D-d\infty -\Phi_D -\delta )$ can be written as an element 
in~$\R \cdot X_\infty$. $\Box$
\end{proof}

\bigskip

In order to later on interpret the N\'eron-Tate height (associated with some given (symmetric) invertible sheaf) as an Arakelov height in a suitable sense
(see~\cite{Ab97} paragraph~3, or~\cite{MB86}), we will need to compute explicitly, given~$P\in X_0 (p)(K)$, the vertical divisor~$\Phi_P =
\oplus_{v\vert p} \Phi_{P,v}$  such that
\begin{eqnarray}
\label{PhiPP}
[C,P- \infty -\Phi_P ]=0
\end{eqnarray} 
for any irreducible component of any fiber of~${\cal X}_0 (p) \to {\mathrm{Spec}} ({\cal O}_K )$, as in the proof of Theorem~\ref{Hodge}. 
\begin{lem}
\label{VerticalComput}
Consider a bad fiber~${\cal X}_0 (p)_{k(v)}$, with~$e_v$ the absolute ramification index of $v$, 
and write $\# k(v) =p^{f_v}$. Let~$P\in X_0 (p)(K)$ and let~$C_{P,v}$ be the irreducible component of~${\cal X}_0 (p)_{k(v)}$ which contains~$P (k(v))$.
As~${\cal X}_0 (p)$ is assumed to be regular, the section $P$ hits each fiber on its smooth locus, so that the component $P$~belongs to is 
unambiguously defined in each bad fiber. Write
$$
\Phi_{P,v} =\sum_{n,m} a_{n,m} [C_{n,m} ]
$$
with notations as in~(\ref{basis}). Recall that, by our convention, $a_{C_\infty} =a_{*,0} =0$.  
\begin{enumerate}
\item If $C_{P,v} = C_0$ then for all $n$ and $m$,
$$
a_{n,m}  = \frac{-12}{(p-1) \cdot w_n }  \cdot m .
$$ 
(Recall (see~(\ref{doubleVn})) that~$w_n :=\# {\mathrm{Aut}} (S (n)) /\langle \pm 1\rangle \in \{ 1, 2, 3\}$, 
with~$S(n)$ the supersingular point corresponding to the branch $\{ C_{n,.} \}$.)  

\medskip

  For further use we henceforth write $\Phi_{C_0}$ for the above vector $\Phi_{P,v} \in \Z^{\cal B}$. 
\item   If $C_{P,v} =C_{n_0 ,m_0 } \neq C_0 ,C_\infty$ then
\begin{itemize}
\item  for~$n=n_0$ and $m\in \{ 0,m_0 \}$, one has $a_{n,m}  =  \left(  \frac{m_0}{w_{n_0} e_v} (1-\frac{12}{(p-1) w_{n_0} }  )  -1 \right) \cdot m$;
\item   for~$n=n_0$ and $m\in \{ m_0 ,w_{n_0} e_v \}$, one has $a_{n,m}  = \left(  \frac{m_0}{w_{n_0} e_v} (1-\frac{12}{(p-1)w_{n_0} }  )  \right) \cdot 
m -m_0 $;
\item  for~$n\neq n_0$ and all~$m\in \{ 0 ,w_n e_v \}$, one has $a_{n,m}  =  \frac{-12m_0}{ (p-1)w_{n_0} e_v} \cdot \frac{m}{w_n}$.
\end{itemize}
\item (Of course if~$C_{P,v} =C_\infty$ then~$\Phi_{P,v} =0$.)
\end{enumerate}
\end{lem}

\medskip

\begin{rema}
\label{CoeffPhi}

{\rm We have distinguished different cases above because the proof naturally leads to doing so, and it will be of interest below to have the 
simpler case $(a)$  explicitly displayed. Note however that all outputs are actually  covered by the formulae of case $(b)$. 
Notice also that, in case $(a)$, all coefficients of~$\Phi_{P,v}$ satisfy
$$
0\geq  a_{n,m} \geq a_0 :=a_{C_0} =a_{n,w_n m} =-12e_v / (p-1) .
$$
As for case $(b)$, all coefficients of~$\Phi_{P,v}$ satisfy
$$
0\geq  a_{n,m} \geq  a_{n_0 ,m_0} =\left(  \frac{m_0}{w_{n_0} e_v } (1-\frac{12}{(p-1) w_{n_0} }  )  -1 \right) \cdot m_0 
$$
(remember $0\leq m\leq w_n e_v$ for all~$m$). Computing the minimum of the above right-hand as a polynomial in~$m_0$ gives
\begin{eqnarray}
\label{sizePhi}
0\geq  a_{n,m} \geq  \frac{-e_v w_{n_0}}{4(1-\frac{12}{(p-1) w_{n_0}  }   )}   \ge \frac{-e_v w_{n_0}}{4-\frac{3}{w_{n_0}  }} \geq -3e_v
\end{eqnarray}
(recalling we always assume  $p\geq 17$). 
}
\end{rema}
\begin{proof}
Given the intersection matrix~(\ref{M}) and condition~(\ref{PhiPP}): ${[C,P-\infty -\Phi_{P,v} ]=0}$ for all~$C$ in the fiber at $v$ gives the matrix equation:
\begin{eqnarray}
\label{Phi_0}
\log (\# k(v)) {\cal M}\cdot \Phi_{P,v} =  \log (\# k(v)) (-1, 0, \cdots ,1,0, \cdots ,0)^t 
\end{eqnarray}
where the coefficient~$1$ (respectively, $-1$) in the right-hand column vector is at the place corresponding to ${C_{P,v} =C_{n,m}}$ (respectively, to ${C_{\infty} =C_{n,0}}$)
in the ordering of our component basis~(\ref{basis}). That is however more easily solved by running through the dual graph of~${\cal X}_0 (p)_{k(v)}$ 
``branch by branch" as follows. Suppose first that~$C_{P,v} =C_0$, and recall  $a_{C_\infty} =0$ by convention. Equation~(\ref{PhiPP}) translates into: 
\begin{itemize}
\item $(-1-\sum_{n=1}^s a_{n,1} =0)$  for~$C=C_\infty$;
\item $(1+sa_0 -\sum_{n=1}^s a_{n,w_n e_v -1} =0)$ for~$C=C_0$;
\item $(a_{n,m-1} -2a_{n,m} +a_{n,m+1} =0)$ for all others~$C=C_{n,m}$.  
\end{itemize}
The equations of the third line in turn define, for each branch (that is, for fixed~$n$), a sequence defined by linear double induction with solution 
$a_{n,m}  = m \cdot \alpha_n$ for some~$\alpha_n$ which is easily computed to be~$\frac{-1}{w({\mathrm{Eis}} ) \cdot w_n} =\frac{-12}{(p-1)w_n}$ (see~(\ref{w(Eis)})). 
(Note this is true even for $e_v =1$.)

  For case~(b), the intersection equations become: %
\begin{itemize}
\item $(-1-\sum_{n=1}^s a_{n,1} =0)$ for~$C=C_\infty$;
\item $(sa_0 -\sum_{n=1}^s a_{n,w_n e_v -1} =0)$ for~$C=C_0$;
\item $(1-a_{n_0 ,m_0 -1} +2a_{n_0 ,m_0} -a_{n_0 ,m_0 +1} =0)$ if~$C=C_{P,v} =C_{n_0 ,m_0}$;
\item $(a_{n,m-1} -2a_{n,m} +a_{n,m+1} =0)$ for all others~$C=C_{n,m}$.
\end{itemize}
As above, solving these equations in all branches not containing~$C_{P,v}$ gives $a_{n,m} =m\beta_n$ and the same is true in the branch 
containing~$C_{P,v}$ for~$m\in \{ 0,\dots , m_0 \}$. We also see that~$a_{n_0 ,m_0 +1} =(m_0 +1)\beta_{n_0} +1$, and then $a_{n_0 ,m} =
m(\beta_{n_0} +1)  -m_0$ for~$m\in \{ m_0 +1 ,w_{n} e_v \} $. We have~$a_0 =w_{n} e_v \beta_n$ for all~$n\neq n_0$, so let~$\beta$ be the 
common value of the $\beta_n$ for $n\neq n_0$ with $w_{n}=1$. (There is always such an $n$ as we assumed $p>13$. Note also those 
computations still cover the case~$e_v =1$.) From~$\beta =a_0 /e_v$ and~$a_0 =w_{n_0} e_v (\beta_{n_0}  +1) -m_0$ we 
derive
$$
\beta_{n_0} =(a_0 +m_0 -w_{n_0}e_v  )/w_{n_0} e_v  =  \frac{\beta}{w_{n_0}} + \frac{m_0}{w_{n_0} e_v} -1.
$$
Hence, because of the first equation~$(-1-\sum_{n=1}^s a_{n,1} =0)$,
$$
0= -1-  \beta_{n_0} -\sum_{1\leq n\leq s, n\neq n_0} \beta /w_{n} =-\beta w({\mathrm{Eis}})  -\frac{m_0}{w_{n_0}e_v } 
$$ 
so that
$$
\beta = \frac{-m_0}{ w({\mathrm{Eis}})w_{n_0} e_v } = \frac{-12\, m_0}{ (p-1)w_{n_0} e_v } .
$$
$\hspace{14cm} \Box$
\end{proof}

\medskip

\begin{lem}
\label{0-infty}  
Let $\mu$ be some $(1,1)$-form on $X_0 (p) (\C )$ as in Theorem~\ref{Hodge}.
\begin{enumerate}
\item 
The class in $\widehat{CH}(p)_{\R, \mu}^{\mathrm{num}}$ of the cuspidal divisor $(0)-(\infty )$ satisfies
\begin{eqnarray}
\label{(0)-(infty)}
(0)-(\infty ) \equiv \Phi^0_{C_0} := \Phi_{C_0} +\sum_{v\vert p} \frac{6e_v}{p-1} (\sum_C [C] ) =\sum_{v\vert p} \sum_{n,m}   \frac{6}{(p-1)}    
(e_v -  \frac{2m}{w_n } )[C_{n,m} ] 
\end{eqnarray}
with notations as in Lemma~\ref{VerticalComput} (a). This is an eigenvector of the Fricke $\Z$-automorphism  $w_p$
with eigenvalue $-1$.
\item One has $[\infty ,\infty ]_{\mu} =[0,0]_{\mu} = [0,\infty ]_{\mu}  -\frac{6\log p}{p-1}$. If $\mu$ is the Green-Arakelov measure $\mu_0$ then
$0\geq [\infty ,\infty ]_{\mu_0} =O(\log p /p)$ and similarly $[0,\infty ]_{\mu_0} =O(\log p/p)$ with $[0,\infty ]_{\mu_0}$ non-positive too, at least 
for large enough $p$. If $\mu =\mu_e$ (see (\ref{mu_e})) - or more generally any sub-measure of $\mu_0$ - then $[0,\infty ]_{\mu_e} =O(p\log p)$. 
\end{enumerate}
\end{lem}
\begin{proof}
By the Manin-Drinfeld theorem, $(0)-(\infty )$ is torsion as a divisor in the generic fiber ${\cal X}_0 (p)\times_\Z \Q$. One therefore has
$$
(0)-(\infty ) \equiv \Phi +cX_\infty 
$$
in the decomposition~(\ref{HHodge}) of $\widehat{CH}(p)_{\R, \mu}^{\mathrm{num}}$, for~$\Phi$ some vertical divisor with support in the fibers above~$p$. 
This divisor is determined by the same equations~(\ref{PhiPP}) as $\Phi_{C_0}$ in Lemma~\ref{VerticalComput} {\it (a)}. 
For each  ${v\vert p}$ the full $v$-fiber $\sum_C [C] $ is numerically equivalent to some real multiple of the archimedean fiber $X_\infty$; there is
therefore a real number $a$ such that 
$$
\Phi^0_{C_0} := \Phi_{C_0} +\sum_{v\vert p} \frac{6e_v}{p-1} (\sum_C [C] ) \equiv \Phi_{C_0} +aX_\infty .
$$
Now $w_p$ switches the cusps $0$ and $\infty$ so the divisor $(0)-(\infty )$ is anti-symmetric for $w_p$: 
$$
w_p^* ((0)-(\infty ))  =-((0)-(\infty )) 
$$
and clearly $w_p^* (\Phi^0_{C_0} )=-\Phi^0_{C_0} $. The fact that $w_p$ preserves 
the archimedean fiber concludes the proof of~{\em (a)}. 

   To prove {\em (b)} we compute
\begin{eqnarray}
0&=& [0-\infty -\Phi_{C_0 }^0 , \infty  ]_\mu   = [0,\infty ]_\mu  -[\infty ,\infty ]_\mu   -\frac{6}{p-1} \log p \nonumber 
\end{eqnarray}
and
\begin{eqnarray}
0&=& [0-\infty -\Phi_{C_0 }^0 , 0  ]_\mu   = [0,0 ]_\mu  -[0 ,\infty ]_\mu   +\frac{6}{p-1} \log p \nonumber 
\end{eqnarray}
so that $[\infty ,\infty ]_\mu = [0 ,0 ]_\mu = [0,\infty ]_\mu   -\frac{6\log p}{p-1}$.
The cusps $0$ and $\infty$ are known not to intersect on ${\cal X}_0 (p)_{/\Z}$ so that $[0,\infty ]_\mu =-g_{{\mu}} (0,\infty )$. When~$\mu =\mu_0$,  
this special value of the Arakelov-Green function has been computed by Michel and Ullmo: it satisfies   
$$
g_{\mu_0} (0,\infty )=\frac{1}{2g} \log p (1+O(\frac{\log \log p}{\log p})) =  O(\frac{ \log p}{ p}) 
$$
by~\cite{MU98}, formula~(12) on p.~650. 
Finally, using~\cite{Br14}, Theorem~7.1~(c) and paragraph~"8, and plugging into Bruin's method the estimates of~\cite{MU98} regarding the comparison 
function~$F(z) =O((\log p)/p)$ between Green-Arakelov and Poincar\'e measures, we obtain a bound of shape $O(p\log p)$ 
for $\vert g_{\mu_e} (0,\infty ) \vert$ (see also Remark~\ref{CorrectOrder}). 
This completes the proof of {\em (b)}.  $\Box$
\end{proof}

\medskip

Instrumental in the sequel will be the explicit decomposition of the relative dualizing sheaf~$\omega$ in the arithmetic Chow group.

\begin{propo}
\label{decomposeomega}
The relative dualizing sheaf $\omega$ of the minimal regular model ${\cal X}_0 (p)\to {\cal O}_K$ can be written,
in the decomposition~(\ref{HHodge}) of $\widehat{CH}(p)_{\R, \mu_0}^{\mathrm{num}}$ relative to the canonical
Green-Arakelov $(1,1)$-form~${\mu_0}$, as:
\begin{eqnarray}
\label{decomposeomegaequation}
\omega = (2g-2) \infty +\sum_{v\vert p} \Phi_{\omega ,v} +\omega^0 + [K:\Q ] c_\omega X_\infty 
\end{eqnarray}
where the above components satisfy the following properties.

\begin{itemize}
\item   The number $c_\omega$ is equal to~$\frac{(1-2g)}{[K:\Q ]} [\infty ,\infty ]_{\mu_0}$, so that $0\leq  c_\omega \le   O(\log p )$.  \\
\item  
Set
$$
H_4 :=\frac{1}{2} \ {\sum_{P\in {\cal H}_4}} (P-\frac{1}{2} (0+\infty ) ), \ H_3 :=\frac{2}{3} \  {\sum_{p\in {\cal H}_3}}  (P-\frac{1}{2} (0+\infty ) )
$$ 
where the sums run over the sets ${\cal H}_4$ and ${\cal H}_3$, whose number of elements can be $0$ or $2$, of Heegner points of~$X_0 (p)$ with $j$-invariant $1728$
and $0$ respectively. Define 
$$
H_4^0 :=H_4  +[K:\Q ] c_4 X_\infty \hspace{0.4cm} {\mathrm{and}}  \hspace{0.4cm} H_3^0 :=H_3 +[K:\Q ] c_3 X_\infty
$$ 
for two numbers~$c_3$ and $c_4$ with 
${c_3  = O( \log p )}$, and the same for~$c_4$. (Recall this means the $H_*$ are compactified with the normalizing condition~(\ref{normalizing}), whereas
the $H_*^0$ are the orthogonal projections on~$\left( J_0 (p) (K ) \otimes \R \right) \subseteq \widehat{CH}(p)_{\R, \mu_0}^{\mathrm{num}} $ of 
the~$H_*$, so that $[\infty ,H_*^0 ]_{\mu_0} =0$, for $*=3$ or $4$.) One sets  $\omega^0 :=-{H}_4^0 -{H}_3^0$, which can be chosen in $J_0 (p)^0 (\overline{\Q} )$.

\item  Finally, the component~$\Phi_{\omega ,v}$ in each~$G_v$ for~$v\vert p$ is
\begin{eqnarray}
\label{Phi_omega000}  
\Phi_{\omega ,v} =-12 \frac{ (g-1) }{(p-1) }   \sum_{n,m}     \frac{ m}{ w_n}     C_{n,m} 
\end{eqnarray}
with notations as in~(\ref{basis}). We therefore have~$\Phi_{\omega ,v} =(g-1)\Phi_{C_0}$ using notations of Lemma~\ref{VerticalComput}. In particular, 
recalling~$e_v$ is the ramification index of $K/\Q$ at $v$,  
the coefficients~$\omega_{n,m}$ of $\Phi_{\omega ,v}$ in~(\ref{Phi_omega000}) satisfy
\begin{eqnarray}
\label{sizePhiOmega}
0\geq \omega_{n,m} \geq -e_v . 
\end{eqnarray}
\end{itemize}

\end{propo}

\medskip

\begin{proof}
Many parts of those statements are deduced from~\cite{MU98}, Section~6, and results of Edixhoven et al. from~\cite{EJ11c}. See also~\cite{Me11b}, Section~4.4.

  We start by estimating~$c_\omega$. By Arakelov's adjunction formula, 
\begin{eqnarray}
\label{infinityintersection}
-[\infty , \infty]_{\mu_0} =[\infty , \omega]_{\mu_0} & = 
 (2g-2)[\infty ,\infty ]_{\mu_0}   + [K:\Q ] c_\omega   \nonumber
\end{eqnarray}
because of the orthogonality of the decomposition~(\ref{HHodge}). Lemma~\ref{0-infty} therefore implies
$$
0 \leq c_\omega =\frac{(1-2g)}{[K:\Q ] } [\infty ,\infty ]_{\mu_0}  =O(\log p) .
$$

\medskip

The computations of the $J_0 (p)$-part $\omega^0 := -(H_3^0 +H_4^0 )$ follows from the Hurwitz formula, as explained in~\cite{MU98}, paragraph~6, 
p.~670. One indeed checks that, on the generic fiber~$X_0 (p)_{/\Q} ={\cal X}_0 (p) \times_\Z \Q$, the canonical divisor is linearly equivalent to 
$$
(2g-2)\infty - \left( \frac{1}{2} {\sum_{j (P) =e^{i\pi /2} }}' (P-\infty ) +\frac{2}{3} {\sum_{j (P) =e^{2i\pi /3} }}' (P-\infty ) \right) 
$$ 
where the sums $\sum'$ are here restricted to points~$P$ at which~$X_0 (p)\to X (1)$ is unramified (these are the Heegner points alluded to in our 
statement). It follows from the modular interpretation that in each of those sums there are two Heegner points  (if any), which are then ordinary at~$p$ 
(recall we assume~$p>13>3$). This proves that the $J_0 (p) (K )\otimes_\Z \R$-part of $\omega$ is indeed $-(H_4^0 +H_3^0 )$ with $H_4^0 =H_4  +[K:\Q ]
c_4 X_\infty$ and $H_3^0 =H_3 +[K:\Q ] c_3 X_\infty$ for some real numbers~$c_3$ and $c_4$. 
(Note that, as Heegner points are preserved by the Atkin-Lehner involution (\cite{Gr84}, paragraph~5, p.~90), their specializations 
above~$p$ share themselves between the two components~$C_0$ and~$C_\infty$ of~${\cal X}_0 (p)_{/\F_p}$, so that $2H_3^0 ={\sum_{j (P) =e^{i\pi /2} }}' 
(P-\infty )$ and $\frac{2}{3} H_4^0 = {\sum_{j (P) =e^{2i\pi /3} }}' (P-\infty )$ belong to the neutral component $J_0 (p)^0 ({\cal O}_K )$.) The estimates on $c_3$ 
and $c_4$ will be justified at  the end of the proof.      

\medskip

The bad fibers divisors~$\Phi_{\omega ,v} :=\sum_{n,m} \omega_{n,m} [C_{n,m} ]$ can be computed with the ``vertical" adjunction formula (\cite{Li02} 
Chapter~9, Theorem~1.37) as in~\cite{Me11b}, Lemma~4.22. Indeed, for each irreducible component~$C$ in the~$v$-fiber having genus~$0$, one has
$$
[C,C+\omega ] =-2 \log (\# k(v)). 
$$
If~${\cal M}$ is the intersection matrix displayed in~(\ref{M}), and $\delta_{*,*}$ is Kronecker's symbol, we therefore have 
\begin{eqnarray}
\label{PhiOmegaBost}
C\cdot {\cal M} \cdot \Phi_{\omega ,v}  = -2 -\frac{1}{ \log (\# k(v))} [C,C] -(2g-2)\delta_{C,C_\infty } =
\left\{
\begin{array}{lcl}
0 &  {\mathrm{\ if\ }} & C\neq C_\infty ,C_0 \\
s-2g &    {\mathrm{\ if\ }}  & C= C_\infty  \\
s-2  & {\mathrm{\ if\ }} & C= C_0 
\end{array}
\right.
\end{eqnarray}
that is, as~$s=g+1$:
$$
{\cal M}\cdot \Phi_{\omega ,v} = (g-1) (-1, 0, \cdots ,0,1)^t .
$$
That equation is~(\ref{Phi_0}) (up to a multiplicative scalar), which has been solved in the first case of Lemma~\ref{VerticalComput}. Therefore
\begin{eqnarray}
\label{PhiPPhiOmega}
\Phi_{\omega ,v} =(g-1)\Phi_{C_0}, {\mathrm{\ that\ is:\ }}  \omega_{n,m} = \frac{12(1-g)}{(p-1) } \cdot \frac{m}{w_n  }  .
\end{eqnarray} 
As noted in~Remark~\ref{CoeffPhi} and using~(\ref{genussize}), this implies the coefficients~$\omega_{n,m}$ of $\Phi_{\omega ,v}$ satisfy
$$
0\geq \omega_{n,m} \geq \frac{12(1-g)}{p-1} e_v >-e_v . 
$$

\medskip

We finally estimate the intersection products
$$
c_3 =\frac{-1}{[K:\Q ]} [\infty ,H_3 ]_{\mu_0}
\hspace{1cm} {\mathrm{and}} \hspace{1cm} c_4 =\frac{-1}{[K:\Q ]} [\infty ,H_4 ]_{\mu_0} .
$$
By the
adjunction formula and Hriljac-Faltings' theorem (\cite{Ch86}, Theorem~5.1 (ii)) we compute
that for any~$P\in X_0 (p)(K)$, 
\begin{eqnarray}
-2[K:\Q ]  \height_{\Theta}  (P-\frac{1}{2g-2} \omega ) & = & [P-\frac{1}{2g-2} \omega -\Phi_\omega (P) ,P-\frac{1}{2g-2} \omega -\Phi_\omega 
(P) ]_{\mu_0}     \nonumber \\
   & = & \frac{1}{(2g-2)^2}  [\omega, \omega ]_{\mu_0} + \frac{g}{g-1} [P,P]_{\mu_0} -\Phi_\omega (P)^2 \nonumber  
\end{eqnarray}
where here~$\Phi_\omega (P)$ is a vertical divisor supported at bad fibers such that
\begin{eqnarray}
\label{PhiP}
[C,P-\frac{1}{2g-2} \omega -\Phi_\omega (P)]=0
\end{eqnarray} 
for any irreducible component $C$ of any bad fiber of ${\cal X}_0 (p)_{/{\cal O}_K}$. Hence
\begin{eqnarray}
\label{Szpiro}
\frac{1}{(2g-2)^2} \omega^2 + \frac{g}{g-1} [P,P]_{\mu_0}  -\Phi_\omega (P)^2  =  -2[K:\Q ] \height_{\Theta}  ((P-\infty ) +\frac{1}{2g-2} (H_3 +H_4 ))  . 
\end{eqnarray}

  We specialize to the case when~$P=P_*^*$ (where the upper star is $1$ or $2$ and the lower star is $4$ or $3$) is one of the Heegner 
points occurring in~$H_4$ or~$H_3$, respectively. We replace for now the base field $K$ by $F :=\Q (P_*^* ) =\Q (\sqrt{-1})$ (respectively,~$\Q(\sqrt{-3})$). 
The right-hand of~(\ref{Szpiro}), if non-zero, is then
\begin{eqnarray}
\label{michelullmo}
-8\log (p)(1+o(1)) \hspace{0.5cm} {\mathrm{or}}  \hspace{0.5cm} -12\log (p)(1+o(1)),  \hspace{0.5cm} {\mathrm{respectively,}} 
\end{eqnarray}
by~\cite{MU98}, p.~673. If those Heegner points occur we know that $p$ splits in $F$, so there are two bad primes $v$ and $v'$ on ${\cal O}_F$ 
(therefore two bad fibers on ${\cal X}_0 (p)_{/{\cal O}_F}$ and two~$G_v$, $G_{v'}$) to take into account.  We compute~$\Phi_\omega (P_*^* )$ 
and $\Phi_\omega (P_*^* )^2$. As mentioned at the beginning of the proof, $P_*^* $ specializes to the component~$C_0$ at a place, say~$v$, of $F$ 
above $p$, and to~$C_\infty$ at the conjugate place~$v'$. Conditions~(\ref{PhiP}) therefore give that, for any irreducible component~$C$ of 
the fiber at~$v$, 
$$
0= [C,P_*^* -\frac{1}{2g-2} \omega -\Phi_\omega (P_*^* )_v ]=  [C, 0- \infty  -\frac{1}{2g-2} {\Phi_{\omega ,v}} -\Phi_\omega (P_*^* )_v ] 
$$
and using Lemma~\ref{VerticalComput}, Lemma~\ref{0-infty} and~(\ref{PhiPPhiOmega}) one obtains
$$
\Phi_\omega (P_*^* )_v  = -\frac{1}{2g-2} {\Phi_{\omega ,v}}  +\Phi_{C_0 ,v} =\frac{1}{2}  \Phi_{C_0 ,v} 
$$
whereas, at~$v'$:
$$
\Phi_\omega (P_*^* )_{v'}  =  -\frac{1}{2g-2} {\Phi_{\omega ,v'}} = -\frac{1}{2}  \Phi_{C_0 ,v'} .
$$
Using Lemmas~\ref{VerticalComput} and~\ref{0-infty} again we therefore have
\begin{eqnarray}
\label{PhiPsquare}
\Phi_\omega (P_*^* )^2 =\sum_{w\vert p} \frac{1}{4}  \Phi_{C_0 ,w}^2 =\sum_{w\vert p} \frac{1}{4}    [\Phi_{C_0 ,w}   ,0-\infty ] =
\frac{1}{2}  a_{0} \log p =-\frac{6\log (p)}{p-1} .
\end{eqnarray}
As for the self-intersection of $\omega$ one knows from~\cite{Ul00}, Introduction, that 
$$
\omega_{{\cal X}_0 (p)_{/\Z}}^2 =3 g\log (p)(1+o(1)).
$$ 
As the quantity~$\frac{1}{[F:K]} [\omega ]^2$ is known to be
independent from the number field extension~$F/K$, the dualizing sheaf $\omega_{{\cal X}_0 
(p)_{/{\cal O}_F}}$ of ${\cal X}_0 (p)$ over ${\cal O}_F$ 
(instead of~$\Z$) satisfies $\omega^2 = 6g\log (p)(1+o(1))$. 
Summing-up, equation~(\ref{Szpiro}) implies  that
\begin{eqnarray}
\label{PP}
[P_*^* ,P_*^* ]_{\mu_0}  =O (\log (p))
\end{eqnarray}
for each Heegner point~$P_*^*$. Now, on the other hand, the vertical divisor~$\Phi_{P^*_*}$ in the sense of~(\ref{PhiPP}) and Lemma~\ref{VerticalComput}
is $\Phi_{P^*_*} =\Phi_{C_0 ,v}$ for the place $v$ of $F$ where $P_*^*$ specializes on $C_0$ and not $C_\infty$. Therefore
\begin{eqnarray}
\label{overF}
-4\height_{\Theta} (P_*^* -\infty )  &=& [P_*^* -\infty -\Phi_{P^*_*} ,P_*^* -\infty -\Phi_{P^*_*} ]_{\mu_0}       \nonumber \\
 & = &- 2[P_*^* , \infty ]_{\mu_0}  +[P_*^*  , P_*^* ]_{\mu_0}  + [\infty ,\infty ]_{\mu_0}   -(\Phi_{P^*_*} )^2 
\end{eqnarray}
whence, using~(\ref{michelullmo}), (\ref{PhiPsquare}), (\ref{PP}) and Lemma~\ref{0-infty}(b):
$$
[P_*^* , \infty ]_{\mu_0}  =\frac{1}{2} \left(  [P_*^*  , P_*^* ]_{\mu_0}  + [\infty ,\infty ]_{\mu_0}   -(\Phi_{C_0 ,v} )^2   + 4\height_{\Theta} (P_*^* -\infty )  \right) 
=  O(\log p).
$$
Putting everything together  and using Lemma~\ref{0-infty} once more we conclude that
\begin{eqnarray}
\label{HiHi}
c_4  = - \frac{1}{[K:\Q ]} [\infty ,H_4 ]_{\mu_0} = \frac{1}{2[K:\Q ]} \left( -[\infty , P_4^1 +P_4^2 ]_{\mu_0}  +[\infty ,0 +\infty ]_{\mu_0}  \right) = 
    O(\log p )
\end{eqnarray}
and similarly for $c_3$. (Note that the Arakelov intersection products, in the computations around~(\ref{overF}), were performed over $F=\Q (P_*^* )$ and not $K$,
although we did not indicate this in the notations in order to keep it from becoming too heavy. We however want quantities over~$K$ for the statement 
of the theorem, so 
we need considering Arakelov products over $K$ in~(\ref{HiHi}) above.)   \hspace{3cm} $\Box$
\end{proof}

\bigskip

\begin{rema}
\label{symmetricOmega}
{\rm It may be convenient to write, with notations as in~(\ref{decomposeomegaequation}), a more symmetric $\omega$ as
\begin{eqnarray}
\label{decomposesymmetricomega}
\omega = (g-1) (\infty  +0) +(-{H}_4^0 -{H}_3^0 ) + [K:\Q ] c_\omega X_\infty 
\end{eqnarray}
which yields an element with no vertical component at bad fibers.}  
\end{rema}

\section{$j$-height and {$\Theta$-height}}
\label{jETtheta}
In this section we compare two natural heights on $X_0 (p)(\overline{\Q})$, namely the $j$-height and the 
one induced from the N\'eron-Tate $\Theta$-height on $J_0 (p)(\overline{\Q})$. We start with an explicit 
description of the latter, for which it is actually convenient to use a bit of Zhang's language about 
``adelic metrics'' (see~\cite{Zh93}) which, in our modular setting, has a very concrete form. 

  Using notations and results from Section~\ref{minimalregular} we therefore consider the limit, as $e_v$ 
goes to $\infty$, of the dual graph of the special fiber of ${\cal X}_0 (p)$ at a place $v$ of a $p$-adic local field 
with ramification index $e_v$ at $p$ (see Figure~\ref{figure}). Here we normalize the length of the $s=g+1$ 
edges from $C_\infty$ to $C_0$ to be $1$, so that the vertex $C_{n,m}$ corresponds to the point of the 
$n^{\mathrm{th}}$ edge with distance $\frac{m}{e_v w_n}$ from the origin $C_\infty$. Now associate to any 
edge $n\in \{ 1,\cdots ,s\}$ the quadratic polynomial function
\begin{eqnarray}
\label{ZhangMetric1}
g_n (x) \colon [0,1] \to \R , \ x\mapsto \frac{1}{2} x\left( (w_n -\frac{12}{(p-1)} ) x -w_n - 12\frac{(g-1)}{(p-1)}\right)   .
\end{eqnarray}
For $K$ any number field, $P$ in $X_0 (p)(K)$, and $v$ a place of $K$ whose ramification degree and residual degree are still denoted by $e_v$ and $f_v$ respectively,
let 
\begin{eqnarray}
\label{ZhangMetric2}
G(P (K_v  ))= e_v f_v  \log (p) \cdot g_n (C_{P(k(v))} )
\end{eqnarray}
where $C_{P(k(v))}$ is the component to which the specialization of $P$ belongs at $v$, identified to a point of the $n^{\mathrm{th}}$ edge where it lives.
\begin{theo}  
\label{ThetaZhang}
For any number field $K$, there is an element 
\begin{eqnarray}
\label{L'Omega_Theta0} 
\tilde{\omega}_{\Theta ,K} =\left( g\cdot \infty   +\Phi_{\Theta ,K} + c_{\Theta ,K} X_\infty \right) 
\end{eqnarray}
of $\widehat{CH}(p)_{\R, \mu_0}^{\mathrm{num}}$ such that, for any $P\in X_0 (p)(K)$ one has, with notations as in Proposition~\ref{decomposeomega}, 
\begin{eqnarray}
\label{L'Omega_Theta} 
\height_\Theta (P-\infty +\frac{1}{2} \omega^0 ) = \frac{1}{[K:\Q ]} [P,\tilde{\omega}_{\Theta ,K} ]_{\mu_0}
\end{eqnarray}
and the terms of~(\ref{L'Omega_Theta0}) satisfy:
\begin{eqnarray}
\label{ToutesEstimees} 
0  \geq [P,\Phi_{\Theta ,K} ]\geq  -2 [K:\Q ]\log (p)  \hspace{0.5cm}   {\mathrm{and}} \hspace{0.5cm}   c_{\Theta ,K} = [K:\Q ]O(\log p ) .
\end{eqnarray}

  Passing to the limit on all number fields, the height induced on $X_0 (p)(\overline{\Q})$ by pulling-back N\'eron-Tate's $\Theta$-height on $J_0 (p)(\overline{\Q})$ via
the embedding $P\mapsto P-\infty +\frac{1}{2} \omega^0$ can be written as: 
\begin{eqnarray}
\label{adelicZhang}
\height_\Theta (P-\infty +\frac{1}{2} \omega^0 ) & = & \frac{1}{[K :\Q ]} \left( g[P,\infty ]_{\mu_0} +\sum_{v\in M_{K } , v\vert p} G(P (K_v  )) +c_{\Theta ,K}  \right)  
\end{eqnarray}
where Zhang's Green function $G$ at bad fibers is defined in~(\ref{ZhangMetric1}) and~(\ref{ZhangMetric2}).  

  In any case one has that the height satisfies 
\begin{eqnarray}
\label{lesestimees} 
\height_{\Theta} (P-\infty +\frac{\omega^0}{2} ) =\frac{1}{[K :\Q ]} [P,g\cdot \infty ]_{\mu_0} +O(\log p )   .
\end{eqnarray}
\end{theo}

\begin{proof}
We prove~(\ref{L'Omega_Theta}) and~(\ref{ToutesEstimees}); from there reformulation~(\ref{adelicZhang}) and~(\ref{lesestimees}) are straightforward. 

Recall~${\cal X}_0 (p)$ denotes the minimal regular model of~$X_0 (p)$ on~${\mathrm{Spec}} ({\cal O}_K )$, 
that~${\cal J}_0 (p)$ is the N\'eron model of $J_0 (p)$ on the same base, and ${\cal J}_0 (p)^0$ stands for its neutral component. 
Let~$\delta$ be an element of $J_0 (p)(K)$, seen as a degree $0$ divisor on $X_0 (p)$. Up to making a base extension we can assume $\delta$ 
is linearly equivalent to a sum of points in  $X_0 (p)(K)$. We shall denote by $\tilde{\delta} =\delta +\Phi_\delta$ (for $\Phi_\delta$ some vertical divisor on~${\cal X}_0 (p)$, 
with multiplicity $0$ on the component  containing $\infty$, following our running conventions) the associated element of the neutral component ${\cal J}_0 (p)^0 ({\cal O}_K )$ (that is, 
the one whose associated divisor has degree zero on each irreducible component, in any fiber, of ${\cal X}_0 (p)$, and therefore defines a point 
of ${\cal J}_0 (p)^0 ({\cal O}_K )$). For any point $P$ in $X_0 (p)(K)\hookrightarrow  {\cal X}_0 (p) ({\cal O}_K )$ let similarly~$\Phi_P$ be the vertical 
divisor on~${\cal X}_0 (p)$, with support on the bad fibers, such that~$(P-\infty -\Phi_P )$ has divisor class belonging to the neutral component ${\cal J}_0 (p)^0 ({\cal O}_K )$
and, again,~$\Phi_P$ has everywhere trivial~$\infty$-component, see~(\ref{PhiPP}). Recall we can compute~$\Phi_P$ explicitly by Lemma~\ref{VerticalComput}. 
We write~$\Phi_P =\sum_{v\in M_K ,v| p} \sum_{C_v}  a_{C_v} [C_v ]$ where the sum is taken on 
irreducible components $C_v$ of vertical bad fibers of~${\cal X}_0 (p)$.  Using notations of Lemma~\ref{VerticalComput}~{\em (b)} we also define the following 
new vertical divisor at bad fibers:
\begin{eqnarray}
\label{PhiTheta}
\Phi_{\vartheta ,K} :=  \sum_{v\in M_K ,v| p} \sum_{Q_v} a_{C_{Q_v}} C_{Q_v}   = \sum_{v| p} \sum_{(n_0 ,m_0 )} a_{n_0 ,m_0}^v C_{n_0 ,m_0} 
\end{eqnarray}
so that
$$
a_{n_0 ,m_0}^v  =  \left(  \frac{m_0}{w_{n_0} e_v} (1-\frac{12}{(p-1) w_{n_0} }  )  -1 \right) \cdot m_0 \mathrm{.}
$$
Our very definitions imply
\begin{eqnarray}
\label{PhiP-PP}
\Phi_P^2 =[P,\Phi_P ]=[P,\Phi_{\vartheta ,K} ] 
\end{eqnarray}
for any~$P \in X_0 (p)(K)$. Using Faltings' Hodge index theorem we can write the N\'eron-Tate height~$\height_{\Theta} (P-\infty +\delta )$ as:
\begin{eqnarray}
\label{ThetaHeight}
\height_{\Theta} (P-\infty +\delta ) & =& \frac{-1}{2 [K:\Q ]} [P-\infty +\tilde{\delta} -\Phi_P , P-\infty +\tilde{\delta} -\Phi_P ]_{\mu_0}  \nonumber \\
  & = &  \frac{1}{2 [K:\Q ]} ( [P,\omega +2\infty -2\tilde{\delta} ]_{\mu_0}  +2[P,\Phi_P ]_{\mu_0}   -[\Phi_P ,\Phi_P ]_{\mu_0}    \nonumber \\
  & & + [\tilde{\delta} , 2\infty  -\tilde{\delta} ]_{\mu_0}    -[\infty ,\infty ]_{\mu_0}  ) \nonumber \\     
  & = &  \frac{1}{2 [K:\Q ]} ( [P,\omega +2\infty -2\tilde{\delta} +\Phi_{\vartheta ,K} ]_{\mu_0}  + [\tilde{\delta} , 2\infty -\tilde{\delta}  ]_{\mu_0}    -[\infty ,\infty ]_{\mu_0} ) \nonumber \\     
  & = &  \frac{1}{ [K:\Q ]} [P,\tilde{\omega}_{\delta} ]_{\mu_0}
\end{eqnarray}
with
\begin{eqnarray}
\label{ThetaDivisor}
\tilde{\omega}_{\delta} :=\left( \frac{1}{2} (\omega +\Phi_{\vartheta ,K} )+\infty  -\tilde{\delta} \right) +c_{\delta } X_\infty
\end{eqnarray}
for~$X_\infty$ some fixed archimedean fiber of~${\cal X}_0 (p)$ and~$c_{\delta }$
is the real number
\begin{eqnarray}
\label{OmegaInfiniteComponent}
c_{\delta } =  \frac{1}{2} \left( -[\infty ,\infty ]_{\mu_0} +[\tilde{\delta} , 2\infty -\tilde{\delta}   ]_{\mu_0}  \right) .
\end{eqnarray}

      Note that~$\tilde{\omega}_{\delta}$ does not depend on $P$ (as $\Phi_{\vartheta ,K}$ was  introduced to that aim).   
    
\medskip    
    
       Let us now take~$\delta = \omega^0 /2 =-(H_3 +H_4  )/2 \in \frac{1}{12} \cdot J_0 (p)^0 (\Q )$, as defined in~Proposition~\ref{decomposeomega}. (This is 
Riemann's characteristic (the ``$\kappa$'' of~\cite{HS00}, p.~138 for instance, that is the generic fiber of the $J_0 (p) (\Q )\otimes \R$-part of~$\omega$ in the 
decomposition~(\ref{decomposeomegaequation}).) Set $\Phi_{\Theta ,K} :=\frac{1}{2} (\Phi_\omega +\Phi_{\vartheta ,K} ) $.
Then
\begin{eqnarray}
\label{ThetaHeight2}
\tilde{\omega}_{\Theta} :=\tilde{\omega}_{\delta} =\left( g\cdot \infty   +\Phi_{\Theta ,K}  + c_{\Theta ,K} X_\infty \right)
\end{eqnarray}
for~$c_{\Theta ,K}$ which, still using notations of Proposition~\ref{decomposeomega} and its proof, is explicitly given by:
\begin{eqnarray}
\frac{1}{[K:\Q ]} c_{\Theta ,K} &= &\frac{1}{2} \left( c_\omega -c_4 -c_3 +\frac{1}{2} \height_{\Theta} (H_3 +H_4 )   - \frac{1}{[K:\Q ]}  ([\infty ]_{\mu_0}^2  +[\infty ,H_3 +H_4 ]_{\mu_0} ) \right)  \nonumber \\
 & = & \frac{1}{2} \left( c_\omega - \frac{1}{[K:\Q ]} [\infty ]_{\mu_0}^2 +\frac{1}{2} \height_{\Theta} (H_3 +H_4 )   \right) .  \nonumber
\end{eqnarray}

   As in the proof of Proposition~\ref{decomposeomega} we invoke p.~673 of \cite{MU98} to assert~$\height_{\Theta}  (H_3 +H_4 )=O(\log (p) )$. We 
moreover know from the same Proposition and from Lemma~\ref{0-infty} that both~$\vert c_\omega \vert = O(\log p )$ and $[\infty ,\infty ]_{\mu_0} = [K:\Q ] O(\log p/p )$, 
so that
\begin{eqnarray}
\label{c_Theta}
c_{\Theta ,K} = [K:\Q ]O(\log p ). 
\end{eqnarray}

The contribution of $\Phi_{\Theta ,K} $ is controlled by Lemma~$\ref{VerticalComput}$ and Remark~\ref{CoeffPhi}: on one hand,
\begin{eqnarray}
\label{Phi^2}
0  \geq [P,\Phi_{\vartheta ,K} ]  = [P,\Phi_P ]  = \sum_{v\in M_K ,v| p} a_{C_P ,v} \log (\# k_v )    \geq   \sum_{v\in M_K ,v| p}  -3e_v \log (p^{f_v} ) \nonumber \\
             \geq  -3 [K:\Q ]\log (p)     
\end{eqnarray}
On the other hand, by 
(\ref{sizePhiOmega}), the coefficients of the vertical components $\Phi_{\omega ,v}$ satisfy $0\geq \omega_{n,m} \geq -e_v$, so 
writing~$\omega_{n_P ,m_P ,v}$ for the coefficient in~$\Phi_{\omega ,v}$ of the component containing~$P(k(v))$ we have:
\begin{eqnarray}
\label{Phi_omega} 
0  \geq  [P,\Phi_\omega ]  = \sum_{v\vert p} \omega_{n_P  ,m_P ,v} \log (\# k(v)) \geq  \sum_{v\vert p} -e_v \log (p^{f_v}  ) =-[K:\Q ]\log (p) .
\end{eqnarray}
Putting  (\ref{c_Theta}), (\ref{Phi^2}) and (\ref{Phi_omega}) together completes the proof of~(\ref{L'Omega_Theta}) and~(\ref{ToutesEstimees}) and the proof. $\hspace{1cm} \Box$

\end{proof}

\medskip
 
\begin{rema}
{\rm Estimates on the Green-Zhang function on $X_0 (p)$ as in the above theorem will be extended below 
to the N\'eron model over $\overline{\Z}$ of the whole jacobian $J_0 (p)$, 
see Proposition~\ref{padicmetric0}.}
\end{rema}

\begin{rema}
\label{symmetricHodge}
{\rm As already noticed, the involution $w_p$ acts as an isometry (actually, an orthogonal symmetry) with respect to the 
quadratic form $\height_\Theta$ on $J_0 (p)(K)\otimes_\Z \R$. Indeed $w_p$ acts as multiplication by $\pm 1$ on each factor of Shimura's decomposition
up to isogeny:
$$
J_0 (p) \sim \prod_{f\in {G_\Q} \cdot S_2 (\Gamma_0 (p))^{\mathrm{new}}} J_f 
$$  
whose factors are $\height_\Theta$-orthogonal subspaces. (See also~\cite{Me08}, Corollaire~4.3, or \cite{Me11b},
Theorem~4.5 (3).) As $w_p (\omega^0 )=\omega^0$ (see the proof of Proposition~\ref{decomposeomega}) this implies 
$$
\height_\Theta (P-\infty +\frac{1}{2} \omega^0 ) =\height_\Theta (w_p (P-\infty +\frac{1}{2} \omega^0 )) = \height_\Theta (w_p (P)-0 +\frac{1}{2} \omega^0 )=
\height_\Theta (w_p (P)-\infty +\frac{1}{2} \omega^0 )
$$
using once more that $(0)-(\infty )$ is torsion, so that
\begin{eqnarray}
\label{w-}
[P,\tilde{\omega}_{\Theta} ]_{\mu_0} =[w_p (P),\tilde{\omega}_{\Theta} ]_{\mu_0} = [P,w_p^* (\tilde{\omega}_{\Theta} )]_{w_p^* (\mu_0 )} 
=[P,w_p^* (\tilde{\omega}_{\Theta} )]_{\mu_0} 
\end{eqnarray}
(see Remark~\ref{wpinvariance}). This suggests it could sometimes be convenient to write $\tilde{\omega}_{\Theta}$ in a $w_p$-eigenbasis of $\widehat{CH}(p)_{\R, 
\mu}^{\mathrm{num}}$ instead of that of Theorem~\ref{Hodge}, for instance
\begin{eqnarray}
\label{ouioui}
\widehat{CH}(p)_{\R, \mu_0}^{\mathrm{num}} =\R \cdot \frac{1}{2} (0+\infty ) \oplus \R \cdot  X_\infty \oplus_{v\vert p} \Gamma_v \oplus \left( J_0 (p) (K)\otimes  \R \right)
\end{eqnarray}
where now the $\Gamma_v$ decompose as the direct sum of eigenspaces $\Gamma_v^{w_p =-1}$ and $\Gamma_v^{w_p =+1}$, with bases: 
\begin{eqnarray}
\label{withbasis}
\{C_{n,m}^- :=C_{n,m} -w_p (C_{n, m} ) \}_{{1\leq n\leq s} \atop {0\leq m\leq ew_n /2}} {\mathrm{\ and}\ }  \{C_{n,m}^+ :=C_{n,m} +w_p (C_{n, m} )-C_0 -
C_\infty \}_{{1\leq n\leq s} \atop {1\leq m\leq ew_n /2}}
\end{eqnarray}
respectively. Using Lemma~\ref{0-infty} and Proposition~\ref{decomposeomega}, a lengthy but easy computation allows one to check that 
$$
\tilde{\omega}_{\Theta} =g\cdot  \frac{1}{2} (0+\infty ) + \Phi_\Theta^+ +\gamma_\Theta X_\infty
$$ 
where $\Phi_{\Theta}^+$ is an explicit vertical divisor above~$p$ with $w_p^* (\Phi_\Theta^+ )=\Phi_\Theta^+$, so that indeed
$$
w_p^* (\tilde{\omega}_\Theta ) =\tilde{\omega}_\Theta 
$$
thus recovering (\ref{w-}). 

  Consider for instance the case of ${\cal X}_0 (p)$ over $\Z$, for $p\equiv 1$ mod $12$ (that is, ${\cal X}_0 (p)_{/\Z}$ is regular, so that there is no 
need to blow-up singular points of width larger than $1$). Here $\Gamma_v =\Gamma_v^- =\R \cdot C_0^- =\R \cdot ([C_\infty ]-[C_0 ])$) and one 
readily checks that 
\begin{eqnarray}
\label{p=1mod12}
{\tilde{\omega}}_\Theta =\frac{g}{2} (0+\infty ) +\gamma_\Theta X_\infty 
\end{eqnarray}
that is, there is no $\Gamma_v$-component at all in that case. Evaluating $\height_\Theta (\frac{1}{2} \omega^0 )$ as in the proof of 
Proposition~\ref{decomposeomega} and using Lemma~\ref{0-infty},
$$
\gamma_\Theta =-\frac{g}{2} [\infty ,0+\infty ]_{\mu_0} +\height_\Theta (\frac{1}{2} \omega^0 ) =gO(\log p/p) +O(\log p) =O(\log p).
$$ 
}
\end{rema}

\bigskip

   We then turn to the $j$-height, first making a comparison of $\height_j$ with the ``degree component" (in the sense of 
Theorem~\ref{Hodge}) of the hermitian sheaf~$\omega$. 

\begin{propo}
\label{jHeight}
Let~$\height_j$ be Weil's $j$-height on~$X_0 (p)$ as defined in in Section~\ref{ModularCurves}, and let~$\mu_0$ and $\mu_e$ be the $(1,1)$-forms defined in~(\ref{mu}) and (\ref{mu_e}). 
Recall $\sup_{X_0 (p) (\C )} g_{\mu}$ stands for the upper bound for {\it all} Green functions $g_{\mu ,a}$ relative to some point $a$ of $X_0 (p)(\C )$
and to the measure $\mu$. 

  If~$p$ is a prime number, $K$ is a number field, and~$P$ belongs to~$X_0 (p)(K)$, then
\begin{eqnarray}
\label{heightJ}
\height_j (P)      &  \leq & (p+1)  \left(  \frac{1}{[K:\Q ]} [P, \infty ]_{\mu_0} +  \sup_{X_0 (p)(\C )} g_{\mu_0}  +O( 1) \right) \nonumber \\
                                         & \leq &  \frac{(p+1)}{[K:\Q ]} [P, \infty ]_{\mu_0} +O(p^2 \log p)
 \end{eqnarray}
and similarly
\begin{eqnarray}
\label{heightJe}
\height_j (P)      &  \leq & (p+1)  \left(  \frac{1}{[K:\Q ]} [P, \infty ]_{\mu_e} +  \sup_{X_0 (p)(\C )} g_{\mu_e}  +O( 1) \right) \nonumber \\
                                         & \leq &  \frac{(p+1)}{[K:\Q ]} [P, \infty ]_{\mu_e} +O( p^3 ).
\end{eqnarray}
\end{propo}
\medskip
\begin{rema}
\label{CorrectOrder}
{\rm As explained in the proof below, the function $O(p^2 \log p)$ of~(\ref{heightJ}) comes from~\cite{Wi16}, Corollary~1.5, together with~\cite{Ul00}, Corollaire~1.3 for the estimate of Faltings' $\delta$
invariant for $X_0 (p)$, which imply the suprema of our functions verify:
\begin{eqnarray}
\label{CelleDeWilms}
\sup_{X_0 (p)(\C )} g_{\mu_0} \le  O(p\log p) .  
\end{eqnarray}
The function $O(p^3 )$ of~(\ref{heightJe}) in turns follows from the main result of~\cite{Br14}. Indeed this states explicitly that 
$\sup_{X_0 (p)(\C )} g_{\mu_0} \le  0.088 \cdot p^2 + 7.7 \cdot p + 1.6 \cdot 10^4$,
see~\cite{Br14}, Theorem~1.2. It follows from measures comparison (see~(\ref{CompMeasures}) below) and the 
method of P.~Bruin that this holds for~$\sup_{X_0 (p)(\C )} g_{\mu_e}$ too, so that
\begin{eqnarray}
\label{CelleDeBruin}
\sup_{X_0 (p)(\C )} g_{\mu_e} \le  O(p^2 ) .  
\end{eqnarray}
It seems that, at least in the case of $X_0 (p)$, if we plug into Bruin's method the estimates of~\cite{MU98} regarding the comparison function~$F(z)$ 
between Green-Arakelov and Poincar\'e measures, we recover 
bounds of shape $O(p\log p)$ instead of $O(p^2 )$ (see~\cite{Br14}, p.~263, and Paragraph~8 (Theorem~7.1 in particular)), and the same again
holds true for the Green function $g_{\mu_e}$. One should therefore be able to obtain the same error term $O(p^2 \log p)$ for~(\ref{heightJe}) as for~(\ref{heightJ}).

  Note that the main theorems of~\cite{JK06} and~\cite{Ar13} might even yield that the above functions $O(p^2 )$ or $O(p\log p)$ could be replaced by a uniform 
bound $O(1)$. 
}
\end{rema}
\begin{proof}
This is essentially a question of measure comparisons on $X_0 (p)(\C )$, between $j^* (\mu_{FS})$ on one hand (where $\mu_{FS}$ is the 
Fubini-Study $(1,1)$-form on $X(1)(\C )\simeq  \PPP^1 (\C )$) and the Green-Arakelov form $\mu_0$ (respectively, $\mu_e$) on the other hand. 
We adapt the main result of~\cite{EJ11b}. 

   We define first a somewhat canonical  Arakelov intersection product $[\cdot ,\cdot ]_{\mu_{FS}}$ on the projective line using $\mu_{FS}$. 
Write $\PPP^1_{/{{\cal O}_K}} ={\mathrm{Proj}} ({{\cal O}_K} [x_0 ,x_1 ]) = {\overline{\mathrm{Spec}}}^{\mathrm{Zar}}  ({{\cal O}_K} [j ])$ (with
$j=x_1 /x_0$), so that the horizontal divisor~$\infty ({{\cal O}_K} )$ is $V(x_0 )$ and, for any $P=[x_0 :x_1 ]$, let the associated Green function be
$$
g_{\mu_{FS} ,\infty } (P)= g_{\mu_{FS} ,\infty } (j(P)) =\frac{1}{2} \log \left( \frac{\vert x_0 \vert^2 }{\vert x_0 \vert^2 +\vert x_1 \vert^2} \right) 
=-\frac{1}{2} \log (1+\vert j(P)\vert^2 )
$$
at any point different from $\infty =[0:1]$. (We note in passing this ad hoc Green function does not need to
fulfill the normalization condition~(\ref{normalizing}).) Then for any $P$ in $X(1)(K)$ one easily 
checks that 
\begin{eqnarray}
\label{LeLog2}
\left\vert \height_j (P) - \frac{1}{[K:\Q ]} [j(P),\infty ]_{\mu_{FS}} \right\vert \leq \frac{1}{2} \log (2).
\end{eqnarray}  
Applying~\cite{EJ11b}, Theorem~9.1.3 and its proof to the setting described above gives, for any $P$ in $X_0 (p)(K)$,
\begin{eqnarray}
\label{ComparEdix}
[j(P),\infty ]_{\mu_{FS}} \leq [P,j^* (\infty )]_{\mu_0} +(p+1)\sum_\sigma \sup_{X_0 (p)_\sigma} g_{\mu_0} +\frac{1}{2} \sum_\sigma 
\int_{X_0 (p)_\sigma} \log (\vert j\vert^2 +1) \mu_{0} 
\end{eqnarray}
where~$\sigma$ runs through the infinite places of $K$ and $X_0 (p)_\sigma :=X_0 (p) \times_{{\cal O}_K ,\sigma} \C$. 

    We estimate the right-hand terms of~(\ref{ComparEdix}). 
As for the last integrals we recall that, on the union of disks of ray $\vert q\vert <r$ around the cusps 
(that is, on the image in $X_0 (p)(\C )$ of the open subset $D_r :=\{ z\in {\cal H}, \Im (z) > -(\log r) 
/2\pi \}$ in Poincar\'e upper-half plane~$\cal H$) for some fixed $r$ in $]0,1[$, one has
$$
\left\vert \frac{f(q)}{q} \right\vert \leq \frac{2}{(1-r)^2}  
$$ 
for any newform $f$ in $S_2 (\Gamma_0 (p))$. (See for instance~~\cite{EJ11c}, Lemma~11.3.7 and its proof.) We also know that the Petersson norm of
such an $f$ satisfies $\| f\|^2 \geq \pi e^{-4\pi }$ (\cite{EJ11c}, Lemma~11.1.2). Choose $r=1/2$ to fix ideas. On $D_{1/2}$, we have (see~(\ref{mu})):
$$
\mu_0 =\frac{i}{2 \dim (J)} \sum_{f \in B_2} \frac{f\frac{dq}{q} \wedge \overline{f\frac{dq}{q}} }{\| f\|^2 } \leq \frac{64e^{4\pi}}{\pi } \frac{i}{2} dq\wedge \overline{dq} .
$$ 
(Sharper bounds should be achievable, but the one above is good enough for our present purpose.) 
It follows that there exists some real $A$ such that, in the decomposition
\begin{eqnarray}
\label{decomposeintegral}
\int_{X_0 (p) (\C )} \log (\vert j\vert^2 +1) \mu_{0} = \int_{X_0 (p)(\C )\cap D_{1/2}} \log (\vert j\vert^2 +1)  \mu_0    +  
\int_{X_0 (p) (\C )\setminus D_{1/2}} \log (\vert j\vert^2 +1)  \mu_0   
\end{eqnarray}
the first term of the right-hand side satisfies
$$
 \int_{X_0 (p)(\C )\cap D_{1/2}} \log (\vert j\vert^2 +1)  \mu_0  \leq  \frac{64e^{4\pi}}{\pi } [\SL_2 (\Z ):\Gamma_0 (p)] 
 \int_{X (1)(\C )\cap D_{1/2}} \log (\vert j\vert^2 +1)  \frac{i}{2} dq\wedge \overline{dq} \leq (p+1) A.
$$
As for the second term, remembering that $\mu_0$ has total mass $1$ on $X_0 (p)(\C )$ we check that 
$$
\int_{X_0 (p) (\C )\setminus D_{1/2}} \log (\vert j\vert^2 +1)  \mu_0   \leq M_{1/2} :=\max_{X(1)(\C )\setminus D_{1/2}} (\log (\vert j\vert^2 +1))
$$
whence the existence of some absolute real number~$A_0$ such that
\begin{eqnarray}
\label{SecondIntegral}
\int_{X_0 (p) (\C )} \log (\vert j\vert^2 +1) \mu_{0}  \leq (p+1)A_0  .
\end{eqnarray}

\medskip

Putting this together with~(\ref{ComparEdix}) 
we obtain a constant~$C$ for which~(\ref{LeLog2}) reads
$$
\height_j (P) \leq    \frac{1}{[K:\Q ]} [P,j^* (\infty )]_{\mu_0}  +  (p+1)(\sup_{X_0 (p)(\C )} g_{\mu_0}  +A_0 ). 
$$

  With notations of Lemma~\ref{0-infty}, one further has 
\begin{eqnarray}
\label{equiv}
j^* (\infty )=p(0)+(\infty ) \equiv (p+1)\infty +p\cdot \Phi^0_{C_0}
\end{eqnarray}
as elements of~$\widehat{CH}(p)_{\R, \mu_0}^{\mathrm{num}}$. Using Lemma~\ref{0-infty} (a) we get
$$
\vert [P,\Phi_{C_0}^0 ]\vert \leq [K:\Q ] \frac{6 \log p}{p-1} 
$$
so that, with~(\ref{CelleDeWilms}),
\begin{eqnarray}
\height_j (P) &  \leq  & \frac{1}{[K:\Q ]} [P,(p+1)\infty ]_{\mu_0}  +(p+1)(\sup_{X_0 (p)(\C )} g_{\mu_0}  +A_0 ) +O(\log p) \nonumber \\
 &  \leq  & \frac{1}{[K:\Q ]} [P,(p+1)\infty ]_{\mu_0}  +C_0 \cdot  p^2 \log p \nonumber
\end{eqnarray}
which is~(\ref{heightJ}).

  The proof of~(\ref{heightJe}) proceeds along the same lines, with one more ingredient. Applying Theorem~9.1.3 of~\cite{EJ11b}
with the measure~$\mu_e$ instead of~$\mu_0$ gives the corresponding version of~(\ref{ComparEdix}). To obtain an upper bound 
for $\sup_{X_0 (p) (\C )} g_{\mu_e}$ we recall that the theorem of Kowalski, Michel and Vanderkam asserts that ${\dim (J_e )} \geq \dim (J_0 (p)) /5$ 
for large enough~$p$.
Our measure $\mu_e := \frac{1}{\dim (J_e )} \sum_{S_e}  \frac{i}{2} \frac{f\frac{dq}{q} \wedge \overline{f\frac{dq}{q}} }{\| f\|^2 }$
(see (\ref{mu_e})) therefore satisfies 
\begin{eqnarray}
\label{CompMeasures}
0\leq \mu_e \leq \frac{g}{\dim (J_e )} \mu_0 \leq 5  \mu_0 .
\end{eqnarray}
This shows that as in~(\ref{CelleDeBruin}), Bruin's theorem (\cite{Br14}, Theorem~7.1) provides a universal~$c_e$ such that 
\begin{eqnarray}
\label{MerklCe}
\sup_{X_0 (p) (\C )} g_{\mu_e}  \leq c_e \, p^2  . 
\end{eqnarray}
Using~(\ref{SecondIntegral}) we obtain: 
\begin{eqnarray}
\label{SecondIntegralMue}
\int_{X_0 (p) (\C )} \log (\vert j\vert^2 +1) \mu_{e}  \leq (p+1)A_e  . 
\end{eqnarray}
Finally, equivalence~(\ref{equiv}) remains naturally true in the Chow group $\widehat{CH}(p)_{\R, \mu_e}^{\mathrm{num}}$ 
relative to the measure $\mu_e$ instead of $\mu_0$, as remarked in Lemma~\ref{0-infty} {\em (a)}.
This completes the proof of~(\ref{heightJe}).  \hspace{9cm} $\Box$
\end{proof}

\bigskip

     We can finally relate $\height_j$ and the N\'eron-Tate height~$\height_\Theta$ relative to the $\Theta$-divisor (see~(\ref{symmetricTheta})):
\begin{theo}
\label{TheoremTheta} 
There are real numbers~$\gamma , \gamma_1$ such that the following holds. Let~$K$ be a number field and~$p$ a prime number. Let~$\omega^0 
:= - (H_4 +H_3 )$ be the $0$-component of the canonical sheaf~$\omega$ on~$X_0 (p)$ over $K$ (as in Proposition~\ref{decomposeomega} and 
Theorem~\ref{ThetaZhang}). If~$P$ is a point of~$X_0 (p)(K)$ then
\begin{eqnarray}
\label{jThetaheight}
\height_j (P)  \leq  (12 +o(1) )  \cdot \height_{\Theta} (P-\infty +\frac{1}{2} \omega^0 )  +\gamma \cdot p^2 \log p 
\end{eqnarray}
\begin{center}
and
\end{center}
\begin{eqnarray}
\label{jThetaheightinfty}
\height_j (P)  \leq  (24 +o(1) )  \cdot \height_{\Theta} (P-\infty  )  +\gamma_1 \cdot p^2 \log p .
\end{eqnarray}
\end{theo}
\begin{rema}
\label{InverseInequalityTheta-j}
{\rm Theorem~\ref{TheoremTheta} offers only one direction of inequality between $j$-height and $\Theta$-height: 
with our method of proof, it is harder to give an effective form to the reverse inequality, because of the 
metrics comparisons we use (see below). 

  Notice also that going through the above proofs using the estimate $\sup_{X_0 (p)(\C )} g_{\mu_0} =O(1)$ of~\cite{JK06} and~\cite{Ar13} (see Remark~\ref{CorrectOrder}) 
would even give an error term of shape $O(p)$ instead of $O(p^2 \log p)$ in~(\ref{jThetaheightinfty}). 

   Those results are in some sense (hopefully sharp) special cases of the main results of~\cite{Pa12}, after rewriting the $j$-function in terms of classical $\Theta$.
}
\end{rema}
\begin{proof}
Using Theorem~\ref{ThetaZhang}, (\ref{lesestimees}), Proposition~\ref{jHeight} and~(\ref{genussize}) we obtain
\begin{eqnarray}
\label{ComparisonThetaJ}
\height_j (P) & \leq &  
12 \frac{p+1}{p-13}  \height_{\Theta} (P-\infty +\frac{1}{2} \omega^0 ) +O(p^2 \log p ).   \nonumber  
\end{eqnarray}
The last estimate (\ref{jThetaheightinfty}) of the theorem comes from the fact that $\height_{\Theta}$ is a quadratic form and that
\begin{eqnarray}
\label{tailleOmega}
\height_{\Theta} (\omega^0 )=O(\log p)
\end{eqnarray}
by the results of~\cite{MU98} now many times mentioned.  $\hspace{7cm} \Box$
 
\end{proof}

\section{Height of modular curves and the various $W_d$}
\label{LaHauteurDeToutLeMonde}

We prove in this section a certain number of technical results about heights of cycles in the modular jacobian,
which will be useful in the sequel. For applications of the explicit arithmetic B\'ezout theorem displayed in next 
section (Proposition~\ref{Bezout}), we indeed first need estimates for the degree and height of the image of $X_0 (p)$, 
together with its various $d^{\mathrm{th}}$-symmetric products (usually called ``$W_d$''), within either $J_0 (p)$ or its 
quotient $J_e$, relative to the $\Theta$-polarization. (For more general considerations on this topic, we also refer to~\cite{dJ18}.)
We estimate those heights both in the normalized N\'eron-Tate 
sense and for some good (``Moret-Bailly'') projective models, to be defined shortly. 

  Let us first define the height of cycles relative to some hermitian bundle. For further details on this 
we refer to~\cite{Zh95}, or to~\cite{Ab97}, Section~2 for a more informal introduction.
\begin{defi}
\label{ZhangModels}
Let $K$ be a number field and ${\cal O}_K$ its ring of integers. Let ${\cal X}$ be an arithmetic scheme 
over ${\cal O}_K$, that is an integral scheme which is projective and flat over ${\cal O}_K$, having smooth generic
fiber $X$ over $K$. Let ${\cal F}$ be a generically ample and relatively semiample hermitian sheaf with smooth metric,
see~\cite{Zh95}, Section~5. We denote by $\hat{c}_1 ({\cal F})$ the first arithmetic Chern class 
of ${\cal F}$, and similarly by $c_1 (F)$ the first Chern class of $F$.

   Such a pair $({\cal X}, {\cal F})$ will be called a model, in the sense of Zhang, of its pull-back
$(X,F) =({\cal X}_K ,{\cal F}_K)$ to the generic fiber.    
\end{defi}

   Consider a model $({\cal X}, {\cal F})$ as in Definition~\ref{ZhangModels}, and let $Y$ be a $d$-dimensional 
subvariety of $X$. The degree of ${Y}$ with respect to ${F}$ is as usual the non-negative integer given by the $d^{\mathrm{th}}$-power 
self-intersection of ${c}_1 ( F)$ with $Y$, that is
$$
\deg_{F} ( Y) = \left( c_1 (F)^d \vert Y \right) .
$$ 
We shall sometimes also write that quantity as $\deg_{\cal F} (Y)$.

Now let ${\cal Y} \to {\cal X}$ be some ``generic resolution of singularities'' of $Y$ (that is, some good 
integral model for some desingularization of $Y$, see Section~1 of~\cite{Zh95}).
  The height of $Y$ with respect to ${\cal F}$ will similarly be the real number obtained by taking the 
the $({\dim {\cal Y}})^{\mathrm{th}}$-power self-intersection of $\hat{c}_1 ({\cal F})$ with ${\cal Y}$, 
divided by the degree of $Y$ and normalized so that:
\begin{eqnarray}
\label{lavraiehauteur}
\height_{\cal F} (Y)  =\frac{(\hat{c}_1 ({\cal F} )^{d+1} \vert {\cal Y} )}{[K:\Q ](d+1)\deg_{F} (Y)} .
\end{eqnarray}   
One can check that definition\footnote{It could have been simpler to systematically use the definition of 
height of~\cite{BGS94}, Section 3.1, which does not demand desingularization, as we do in the proof
of Proposition~\ref{Bezout} at the end of Section~\ref{ArithmeticBezout}. We could not find references 
however for Zhang's inequality (see~(\ref{Zhang})) in that setting, so we stick to the above 
definitions.} does not depend on the desingularization ${\cal Y} \to {\cal X}$.

   Instrumental to us will here be Zhang's control of heights in terms of essential minima. Recall
that the (first) essential minimum $\mu^{\mathrm{ess}}_{\cal F} (Y)$ of $Y$ is the minimum of the set
of real numbers $\mu$ such that there is a sequence of points $(x_n )$ in $Y(\overline{\Q })$ which is 
Zariski dense in $Y$ and $\height_{\cal F} (x_n )\le \mu$ for all $n$. Zhang's Theorem~(5.2) of~\cite{Zh95} then asserts that
\begin{eqnarray}
\label{Zhang}
\height_{\cal F} (Y) \leq  \mu^{\mathrm{ess}}_{\cal F} ({Y}) .
\end{eqnarray}
Note that if $\height_{\cal F} \geq 0$ on $Y(\overline{\Q })$ one also knows from~\cite{Zh95}, Theorem~5.2 the reverse inequality
\begin{eqnarray}
\label{ZhangBis}
\height_{\cal F} (Y) \geq \frac{\mu^{\mathrm{ess}}_{\cal F} (Y)}{d+1}  .
\end{eqnarray}

  If $({\cal X} ,{\cal F})$ is a model over~${\cal O}_K$, in the sense of Definition~\ref{ZhangModels}, of a polarized abelian variety $(X,F)$ 
over $K ={\mathrm{Frac}} ({\cal O}_K )$, and $Y$ again is a $d$-dimensional subvariety of the generic 
fiber $X$, we still define its normalized N\'eron-Tate height relative to $F$ as the limit
$$
{\height}_{F} (Y) :=\lim_{n\to \infty} \frac{1}{N^{2n}}  \height_{{\cal F}} ( {[N^n ] Y} )
$$
where $N$ is any fixed integer larger than $1$ and ${[N^n ] Y}$ is the image  
of $Y$ under multiplication by $N^n$ in $X$.  This normalized height, which is a direct generalization of 
the classical notion of N\'eron-Tate height for points, is known not to depend neither on the 
model ${\cal X}$ of $X$, nor the extension ${\cal F}$ of $F$, nor its hermitian structure (and not on $N$), 
so that the notation ${\height}_{F} (\cdot )$ is finally unambiguous. We refer to~\cite{Ab97}, 
Proposition-D\'efinition~3.2 of Section~3 for more details. We will actually use 
the extension of the two inequalities~(\ref{Zhang}) and~(\ref{ZhangBis}) to the case where the heights and 
essential minima are those given by the limit process defining N\'eron-Tate height (which is known to be 
non-negative on points) that is, with obvious notations
\begin{eqnarray}
\label{ZhangNT}
\frac{\mu^{\mathrm{ess}}_{F} (Y)}{d+1}  \leq \height_{F} (Y) \leq  \mu^{\mathrm{ess}}_{F} ({Y}) 
\end{eqnarray}
see Th\'eor\`eme~3.4 of~\cite{Ab97}. As we will see in Section~\ref{greenzhang} and below, Moret-Bailly
theory allows, under certain conditions, to interpret N\'eron-Tate heights as Arakelov 
projective heights (that is, without going through limit process).

\subsection{N\'eron-Tate heights}

  We shall apply the above to cycles in modular abelian varieties endowed with their symmetric theta divisor:
the notation ${\height}_{{\Theta}}$ will always stand for normalized N\'eron-Tate height of cycles.

\begin{propo}
\label{heightXThetae}
Let $X$  be the image via $\pi_A \circ \iota_\infty \colon X_0 (p)\to A$ of the modular curve $X_0 (p)$ mapped to a non-zero quotient $\pi_A \colon J_0 (p)\to A$ of its jacobian, 
endowed with the polarization~${\Theta_A}$ induced by the $\Theta$-divisor (see~(\ref{chain}), (\ref{albanese}) and around). 
The degree and normalized N\'eron-Tate height of~$X$ satisfy:  
$$
\deg_{{\Theta_A}}  (X) =\dim (A) =O(p)
$$
and
$$
{\height}_{{\Theta_A}} (X) =O(\log p ) .
$$
\end{propo}
\begin{proof}
If $(A,{\Theta_A})=({\mathrm{Jac}} (X_0 (p)) ,\Theta )$, it is well-known that the $\Theta$-degree of $X_0 (p)$ 
(or in fact any curve) embedded in its jacobian via some Albanese embedding, equals its genus. 
That can be seen in many ways, among which one can invoke Wirtinger's theorem (\cite{GH78}, p.~171), which 
yields in fact the desired result for any quotient $(A,\Theta_A )$:  using the notation before~(\ref{flat_A}) we have
$$
\deg_{\Theta_A} (X) =%
\int_{X_0 (p)} \sum_{f\in B_2^A}  \frac{i}{2} \frac{f\frac{dq}{q} \wedge \overline{f\frac{dq}{q}}}{\Vert f\Vert^2} =\dim A \leq g(X_0 (p)) . 
$$  
We then apply once more the fact  (\ref{genussize}) that the genus $g(X_0 (p))$ is roughly $p/12$.  %
(We could also have more simply say that the degree is decreasing by projection, as in the argument below.) 

As for the height, the main result of~\cite{MU98} gives that the essential minimum of the normalized N\'eron-
Tate height $\mu^{\mathrm{ess}}_{\Theta} (X_0 (p))$ is $O(\log p)$. As the height of points decreases by projection 
(see Section~\ref{PolarizationsAndHeights}, and in particular (\ref{decomposeHeights})) the same is true for 
$\mu^{\mathrm{ess}}_{{\Theta_A}} (X)$ and we conclude with Zhang's (\ref{ZhangNT}). $\Box$
\end{proof}

\medskip

Now for the N\'eron-Tate normalized height of symmetric squares and variants:

\begin{propo}
\label{heightXSquareThetae}
Assume $X:= X_0 (p)$ has gonality strictly larger than $2$ (which is true as soon as $p>71$, see~\cite{Og74}). 
Let $\iota :=\iota_\infty \colon X_0 (p) \hookrightarrow J_0 (p)$ 
be the Albanese embedding as in Proposition~\ref{heightXThetae}. Let $X^{(2)}$ be the symmetric square $X_0 
(p)^{(2)}$ embedded 
in $J_0 (p)$ via $(P_1 ,P_2 )\mapsto \iota  (P_1 ) +\iota  (P_2 )$, and similarly let  $X^{(2),-}$  be the image 
of $(P_1 ,P_2 )\mapsto \iota  (P_1 ) -\iota  (P_2 )$.
Let  $X^{(2)}_{e^\perp}$ and $X^{(2), -}_{e^\perp}$ be the projections of $X^{(2)}$ and 
$X^{(2),-}$, respectively, to $J_e^\perp$ (the ``orthogonal complement'' to the winding quotient $J_e$, see 
paragraph~\ref{Winding}). Then with notations as in Proposition~\ref{heightXThetae} taking~$A=J_0 (p)$ 
and $A= J_e^\perp$ respectively one has
$$
\deg_{\Theta}  (X^{(2)}) =  O(p^2 ) =\deg_{\Theta}  (X^{(2),-}),  {\hspace{1cm}}  {\height}_{\Theta} 
(X^{(2)}) = O(\log p) ={\height}_{\Theta} (X^{(2),-})
$$
\begin{center}
and the same holds for the quotient objects:
\end{center}
$$
\deg_{\Theta_e^\perp}  (X^{(2)}_{e^\perp} ) =  O(p^2 ) =\deg_{\Theta_e^\perp}  (X^{(2),-}_{e^\perp} ) \ ;  {\hspace{1cm}}  {\height}_{\Theta_e^\perp} 
(X^{(2)}_{e^\perp} ) =  O( \log p) = {\height}_{\Theta_e^\perp} (X^{(2),-}_{e^\perp} )  .
$$
\end{propo}
\begin{proof}
Denoting by $p_1$ and $p_2$ the obvious projections below
we factor in the common way (see \cite{Mu66}, paragraph~3, Proposition~1 on p.~320) our maps over~$\Q$ as follows:
\begin{eqnarray}
\begin{array}{cccccclc}
                &                                       &                 &                                        &                   &   &  A  &   \\
                &                                       &                 &                                        &                   & \nearrow_{p_2}   &    &   \\
X_0 (p)\times X_0 (p) & \stackrel{\pi_A \iota \times\pi_A \iota}{\longrightarrow} & A \times A  & \stackrel{M}{\longrightarrow} &  A \times A    &                                      &       \\  
                &                           & (x,y)        & \mapsto                            & (x+y,x-y)    & \searrow^{p_1}     &       &      \\
               &                             &               &                                           &                  &                            & A         &  \\
\end{array}
\end{eqnarray}
so $X^{(2)} =p_1 \circ M \circ (\pi_A \iota \times\pi_A \iota ) (X_0 (p)\times X_0 (p) )$ and $X^{(2),-} =p_2 
\circ M \circ (\pi_A \iota \times\pi_A \iota )(X_0 (p)\times X_0 (p) )$ when $A=J_0 (p)$, and the same 
with $X^{(2)}_{e^\perp}$ and $X^{(2),-}_{e^\perp}$ with  $A=J^\perp_e$. We endow $A\times A$ with the hermitian 
sheaf ${\Theta_A}^{\boxtimes 2} :=p_1^* {\Theta_A} \otimes p_2^* {\Theta_A}$. Then 
$M^* ({\Theta_A}^{\boxtimes 2} )\simeq ( {{\Theta_A}^{\boxtimes 2}} )^{\otimes 2}$ 
(\cite{Mu66}, p.~320). Therefore, writing $X$ for $\pi_A \iota (X_0 (p))$ in short and using 
Proposition~\ref{heightXThetae}, 
$$
\deg_{{\Theta_A}^{\boxtimes 2}} (M(X\times X)) =4 \deg_{{\Theta_A}^{\boxtimes 2}} (X\times X) =8 
(\deg_{\Theta_A} (X))^2 =O(g^2 ). 
$$ 
As degree decreases by our projections and $O(g^2 )=O(p^2 )$, $\deg_{\Theta_A} (X^{(2)} )$ and $\deg_{\Theta_A} (X^{(2), -} )$ are  $O   (p^2 )$.
By definition of essential minima,
$$
\mu^{\mathrm{ess}}_{{\Theta_A}^{\boxtimes 2}} (X\times X) \le 2 \mu^{\mathrm{ess}}_{\Theta_A} (X) . 
$$
This implies that $\mu^{\mathrm{ess}}_{{\Theta_A}^{\boxtimes 2}} (M(X\times X)) \le 4
\mu^{\mathrm{ess}}_{\Theta_A} (X)$. Invoking~(\ref{ZhangNT}) again 
and Proposition~\ref{heightXThetae}
together with the fact that the height of points also decreases by projection, 
$$
\mu^{\mathrm{ess}}_{{\Theta_A}} (X^{(2)} ) \leq \mu^{\mathrm{ess}}_{{\Theta_A}^{\boxtimes 2}} (M(X\times X)) 
\leq  4 \mu^{\mathrm{ess}}_{\Theta_A} (X) \leq 8{\height}_{\Theta_A} (X) \leq O(\log p). 
$$
Therefore 
$$
{\height}_{\Theta_A} (X^{(2)} ) = O( \log p). \hspace{2cm} \Box
$$ 
\end{proof}

\medskip

  Note that this proof applies more generally to any sub-quotient of $J_0 (p)$.

\subsection{Moret-Bailly models and associated projective heights}
\label{Moret}

To build-up the projective models of the jacobian (over $\Z$, or finite extensions), and associated
heights, that we shall need for our arithmetic B\'ezout, we use Moret-Bailly theory, in the sense of~\cite{MB86}, as follows.
For more about similar constructions in the general setting of abelian varieties we refer to~\cite{BI96}, 
2.4 and 4.3; see also~\cite{Pa12}.

    Let therefore $(J,L(\Theta ))$ stand for the principally polarized abelian variety $J_0 (p)$ endowed with the invertible sheaf associated with its symmetric theta divisor, 
defined over some small extension of~$\Q$ (see~(\ref{theta}) below and around for more details). Endow the complex base-changes of the associated invertible sheaf $L(\Theta )$ 
with its cubist hermitian metric. If~${\cal N}_{J,{\cal O}_K}$ is the N\'eron model of $J$ over the ring of integers ${\cal O}_K$ of a number field $K$, we know it is a semistable scheme 
over ${\cal O}_K$, whose only non-proper fibers are above primes ${\mathfrak{P}}$ of characteristic $p$, where it then is purely toric. At any such ${\mathfrak{P}}$,
with ramification index $e_\mathfrak{P}$, the group scheme ${\cal N}_{J,{\cal O}_K}$ has components group 
\begin{eqnarray}
\label{ComponentGroup}
\Phi_\mathfrak{P} \simeq (\Z /N_0 e_\mathfrak{P} \Z ) \times (\Z /e_\mathfrak{P} \Z )^{g-1}   
\end{eqnarray}   
for $g :=\dim J$ and $N_0 :={\mathrm{num}} (\frac{p-1}{12} )$ (see e.g.~\cite{LF12}, Proposition~2.11).

  We choose and fix an integer $N >0$ and a number field $K \supseteq \Q (J [2N] )$, for all this 
paragraph, so that all the 2$N$-torsion points in $J$ have values in $K$. One then observes 
from~(\ref{ComponentGroup}) that $2N$ divides all the ramification indices $e_\mathfrak{P}$, 
and Proposition~II.1.2.2 on p.~45 of \cite{MB86}   
asserts that $L(\Theta )$ has a cubist extension, let us denote it by ${\cal L}(\Theta )$, to the open 
subgroup scheme ${\cal N}_{J,N}$ of the N\'eron model~${\cal N}_{J,{\cal O}_K}$ over ${\cal O}_K$
whose fibers have component group killed by $N$.

Such an extension ${\cal L}(\Theta )$ is actually symmetric (\cite{MB86}, Remarque II.1.2.6.2) and unique
(see Th\'eor\`eme~II.1.1.i) on p.~40 of loc. cit.). Moreover~${\cal L}(\Theta )$  is ample on~${\cal N}_{J, N }$ 
(\cite{MB86}, Proposition~VI.2.1 on p.~134). Its powers ${\cal L}(\Theta )^{\otimes r}$ are even very ample 
on~${\cal N}_{J,N } \times_{{\cal O}_K} {{\cal O}_K} [1/2p]$ as soon as $r \ge 3$, as follows from the general 
theory of theta functions. Provided $N>1$, the sheaf ${\cal L}(\Theta )^{\otimes N }$ is spanned by its global 
sections on the whole of~${\cal N}_{J,N}$ (\cite{MB86}, Proposition~VI.2.2), although we shall not use that 
last fact as such.

  Picking-up a basis of {\it generic} global sections in $H^0 (J_0 (p)_K , {L}(\Theta )^{\otimes N} )$, 
with $N\geq 3$, we thus defines a map $J_0 (p)_K  \stackrel{\jmath_{N}}{\longrightarrow} \PPP^{n}_{K}$, 
for $n=N^{g} -1$. 
Assume our generic global sections extend to a set ${\cal S}$ in $H^0 ({\cal N}_{J,N} , 
{\cal L}(\Theta )^{\otimes N} )$. 
Let  ${\cal J} \stackrel{\jmath}{\hookrightarrow} \PPP^{n}_{{\cal O}_K}$ be the schematic closure 
in $\PPP^{n}_{{\cal O}_K}$ of  the generic fiber $({\cal N}_{J,N} )_K =J_K$ via the associated composed embedding 
$J_K \hookrightarrow \PPP^{n}_K \hookrightarrow \PPP^{n}_{{\cal O}_K}$. 
Define ${\cal M} =\jmath^* {\cal O}_{\PPP^{n}_{{\cal O}_K}} (1)$ on ${\cal J}$. Let on the other 
hand ${\cal M}_{{\cal N}_{J,N}} :=\left( \sum_{s\in {\cal S}}
{\cal O}_K \cdot s\right)$ %
be the subsheaf of ${\cal L} (\Theta )^{\otimes N}$ on ${\cal N}_{J,N}$ spanned by ${\cal S}$. 
Write $\nu \colon {\widetilde{{\cal N}}_{J ,N}} \to {{\cal N}_{J ,N}}$ for the blowup at base points 
for ${\cal M}_{{\cal N}_{J,N}}$ on ${{\cal N}_{J ,N}}$,
that is, the blowup along the closed subscheme of ${{\cal N}_{J ,N}}$ defined by the sheaf ${\cal L}(\Theta )^{\otimes N} /{\cal M}_{{\cal N}_{J,N}}$. 
%
%
%
We have a commutative diagram 
\begin{eqnarray}
\label{modelA}
\begin{array}{rccl}
                   &  {\widetilde{{\cal N}}_{J ,N}}  & &     \\
                              \nearrow                   &   {}_{\imath_{{\cal N}}} \downarrow     &      \searrow^{\jmath_{{\cal N}}}  &          \\
{J}_K \ \   {\hookrightarrow} & {\cal J} & \stackrel{\jmath}{\hookrightarrow}   \ \  &      \PPP^{n}_{{\cal O}_K}   
 \end{array}
\end{eqnarray}
where the only non-trivial map $\jmath_{\cal N}$ (whence $\imath_{\cal N}$) is deduced from the fundamental 
properties of blowups. Considering the complex base-changes of  the generic fiber we note that ${\cal M}$ is automatically endowed 
with a cubist hermitian structure induced by that of $L({\Theta})_\C$ (see~\cite{BI96}, (4.3.3) and following lines). 

\begin{defi}
\label{GoodModels}
Given an integer $N\geq 3$, and a number field $K$ containing $\Q (J_0 (p) [2N] )$, we define the ``good model'' for $(J_0 (p), L(\Theta )^{\otimes N} )$ 
relative to some finite set ${\cal S}$ in $H^0 ({\cal N}_{J,N} , {\cal L}(\Theta )^{\otimes N} )$,
which spans $H^0 (J_0 (p) , L(\Theta )^{\otimes N} )$, as the projective scheme ${\cal J}$ over ${\mathrm{Spec}} 
({\cal O}_{K} )$ enhanced with the hermitian sheaf ${\cal M}$ 
constructed above, and $\height_{\cal M}$ the associated height.
\end{defi}

\medskip

  Outside base points for ${\cal M}_{{\cal N}_{J,N}}$ on ${{\cal N}_{J ,N}}$ the blowup $\nu \colon {\widetilde{{\cal N}}_{J ,N}} \to {{\cal N}_{J ,N}}$ is an 
isomorphism and on that open locus we have
\begin{eqnarray}
\label{ComputeNeron}
{\cal L}(\Theta )^{\otimes N} \simeq {\cal M}_{{\cal N}_{J,N}} \simeq \imath_{\cal N}^* {\cal M} =
\jmath_{\cal N}^* {\cal O}_{\PPP^n_{{\cal O}_K} }(1)  
\end{eqnarray} 
so we dwell on the fact that the height $\height_{\cal M}$ of our ``good models'' for $(J_0 (p), L(\Theta )^{\otimes N} )$ will indeed compute ($N$ times) 
the N\'eron-Tate height of {\it certain} $\overline{\Q}$-points 
(those whose closure factorizes through ${\cal N}_{J,N}$ deprived from the base points for ${\cal S}$), but 
definitely {\it not all}. 
For arbitrary points, still, one can deduce from the work of Bost (\cite{BI96}, 4.3) the following inequality.

\begin{propo}  
\label{InegalBost}
For any point $P$ in $J_0 (p)(\overline{\Q} )$, the height $\height_{\cal M} (P)$ of 
Definition~\ref{GoodModels} satisfies
$$
\height_{\cal M} (P) \leq N \, {\height}_\Theta (P).
$$ 
\end{propo}

\begin{proof}
We briefly adapt~\cite{BI96}, 2.4 and 4.3, 
using our above notations. Of course this statement has nothing to see with modular jacobians, and holds for any abelian variety over a number field.
Let $N'$ be some integer such that $P$ defines a section of ${\cal N}_{J ,N'} ({\cal O}_F )$ for some ring of 
integers ${\cal O}_F$. Up to replacing ${\cal O}_F$ by a sufficiently ramified finite extension, we can 
assume $L(\Theta )^{\otimes N}$ has a cubist extension ${\cal L}(\Theta )^{\otimes N}$ 
to all of ${\cal N}_{J ,N'}$ over ${\cal O}_F$ (\cite{MB86}, Proposition~II.1.2.2).
One has
$$
{\height}_\Theta (P)= \frac{1}{N} \frac{1}{[F :\Q ]} \widehat{\deg} (P^* ({\cal L}
(\Theta )^{\otimes N} )) .
$$ 
As in~(\ref{modelA}) however we see that there is no well-defined map from ${{\cal N}_{J ,N'}}$ to 
$\PPP^{n}_{{\cal O}_F}$ because 
${\cal L}(\Theta )^{\otimes N}$ needs not be spanned by elements of $\cal S$ on all of ${{\cal N}_{J ,N'}}$ 
(even though it  is, by hypothesis, on the generic fiber). To remedy this we adapt the construction~(\ref{modelA}).

  If $\pi' \colon {\cal N}_{J,N'} \to {\mathrm{Spec}} ({\cal O}_F )$ is the structural morphism, we define now
${{\cal M}'}_{\cal N} :=\left( \sum_{s\in {\cal S}} {\cal O}_F \cdot s\right)$  
as the subsheaf of ${\cal L} (\Theta )^{\otimes N}$ on ${\cal N}_{J,N'}$ spanned by ${\cal S}$, still endowed with the metric induced
by that of ${\cal L} (\Theta )^{\otimes N}$.  One checks (see~\cite{BI96}, (4.3.8)) that the projective model ${\cal J}_{{\cal O}_F}$ of 
$({\cal N}_{J ,N'} )_F \simeq J_F$ in $\PPP^n_{{\cal O}_F}$ defined as in~(\ref{modelA}) yields a sheaf ${\cal M}'$ on ${\cal J}_{{\cal O}_F}$, 
whence a height $\height_{{\cal M}'}$, which coincides with  the height $\height_{{\cal M}}$ on the base change of the good 
model ${\cal J}_{{\cal O}_K}$.

  Replacing ${{\cal N}_{J ,N'}}$ by its blowup $\nu' \colon {\widetilde{{\cal N}}_{J ,N'}} \to {{\cal N}_{J ,N'}}$ at base points 
for ${{\cal M}'}_{\cal N}$ in ${\cal L} (\Theta )^{\otimes N}$ on ${{\cal N}_{J ,N'}}$, we keep on following
construction~(\ref{modelA}) to obtain maps ${\imath'}_{{\cal N}} \colon {\widetilde{{\cal N}}_{J ,N'}} 
\to {\cal J}_{{\cal O}_F}$ and ${\jmath'}_{{\cal N}} \colon {\widetilde{{\cal N}}_{J ,N'}} 
\to \PPP^{n}_{{\cal O}_F}$ such that the Zariski closure of ${{\jmath'}_{{\cal N}}} 
({\widetilde{{\cal N}}_{J ,N'}})$ identifies with ${\cal J}_{{\cal O}_F}$.
We moreover have  
$$
{\imath'}_{{\cal N}}^* ({\cal M}' ) ={\nu'}^* ({\cal L}(\Theta )^{\otimes N} )\otimes {\cal O} (-E)
$$   
where $E$ is the exceptional divisor of the blowup which is by definition effective. 
The section $P$ of ${{\cal N}_{J ,N'}} ({{\cal O}_F})$ lifts to some 
${\widetilde{P}}$ of ${\widetilde{{\cal N}}_{J ,N'}} ({{\cal O}_F})$.
Let $\varepsilon_P$ be the section of ${\cal J}({{\cal O}_F})$ defined by the Zariski closure of $P(F)$ 
in $\cal J$. One can finally compute
\begin{eqnarray}
\height_{\cal M} (P) = \height_{{\cal M}'} (P) & = &\frac{1}{[F :\Q ]} \widehat{\deg} 
(\varepsilon_P^* ({{\cal M}'} )) = \frac{1}{[F :\Q ]} \widehat{\deg} (\tilde{P}^* 
({\imath'}_{{\cal N}}^* ({\cal M}' ) )) \nonumber \\
  &   \leq &  \frac{1}{[F :\Q ]} \widehat{\deg} (\tilde{P}^* ({\nu'}^* 
  ({\cal L}(\Theta )^{\otimes N} ) )) 
 =  \frac{1}{[F :\Q ]} \widehat{\deg} 
 (P^* ({\cal L}(\Theta )^{\otimes N} )) =N \, {\height}_\Theta (P).
\hspace{0.2cm} \Box \nonumber
\end{eqnarray}
\end{proof}

\medskip
 
The following straightforward generalization to higher dimension will be useful in next section.
 
\begin{coro}
\label{LesHauteursInegales}
If $Y$ is a $d$-dimensional irreducible subvariety of $J_0 (p)$ 
then 
$$
\height_{\cal M} (Y) \leq  (d +1)\, N\, {\height}_\Theta (Y).
$$ 
\end{coro}

\begin{proof}
Combine Zhang's formulas (\ref{Zhang}) and (\ref{ZhangNT}) with  Proposition~\ref{InegalBost}. $\Box$
\end{proof}

\bigskip

Recall from~(\ref{pseudoprojection}) that one can define the ``pseudo-projection'' ${\cal P}_{\tilde{J}_{e^\perp}} 
(\iota_\infty (X_0 (p)))$ of the image of
$X_0 (p) \stackrel{\iota_\infty}{\hookrightarrow} J_0 (p)$ on the subabelian 
variety $\tilde{J}_{e^\perp} \subseteq J_0 (p)$. Let $X_{e^\perp}$ be any of its irreducible components.
Define similarly ${X}^{(2)}$, ${X}^{(2),-}$, ${X}^{(2)}_{e^\perp}$ and ${X}^{(2), -}_{e^\perp}$
as in Proposition~\ref{heightXSquareThetae}. 
Note that, by construction, the degree and normalized N\'eron-Tate height of $X_{e^\perp}$ (and other similar pseudo-projections: $X^{(2)}_{e^\perp}$ etc.), as an 
irreducible subvariety of $J_0 (p)$ endowed with ${\height}_\Theta$, are those of $\pi_{J_e^\perp} (X_0 (p)) =X^{(2),-}_{e^\perp}$ relative to the only natural 
hermitian sheaf of ${J}_e^\perp$, that is, the $\Theta_e^\perp =\Theta_{J_e^\perp}$ described in paragraph~\ref{PolarizationsAndHeights} and estimated in 
Proposition~\ref{heightXThetae}.

\begin{coro}
\label{DonnezNousDonnezNousDesHauteurs}
For any fixed integer $N\ge 3$, and any number field $K$ containing $\Q (J_0 (p) [2N] )$, 
let $( {\cal J} , {\cal M} )$ be the good model for $(J_0 (p) ,L(\Theta )^{\otimes N} )$, and 
$\height_{\cal M}$ the associated projective height, given in Definition~\ref{GoodModels}. Let $X$ be 
the image of $X_0 (p) \stackrel{\iota_\infty}{\hookrightarrow} J_0 (p)$, 
and more generally ${X}^{(2)}$,  ${X}^{(2),-}$, ${ X}^{(2)}_{e^\perp}$ and ${X}^{(2), -}_{e^\perp}$ be the 
objects ${X}^{(2)}, \dots$ defined in Proposition~\ref{heightXSquareThetae} (or their pseudo-projections). 
Then their ${\cal M}^{\otimes \frac{1}{N}}$-heights are bounded from above by similar 
functions as their N\'eron-Tate height (Proposition~\ref{heightXSquareThetae}). 
Explicitly, $\height_{{\cal M}^{\otimes  \frac{1}{N}}} ({X}_0 (p) )$ is less than $O(\log p)$, and 
$\height_{{\cal M}^{\otimes  \frac{1}{N}}} {X}^{(2)}$, etc., are all less than $O(\log p)$.
Similarly the ${\cal M}^{\otimes  \frac{1}{N}}$-degree of $X_0 (p)$ is $O(p)$,    
and the ${\cal M}^{\otimes  \frac{1}{N}}$-degrees of ${X}^{(2)}$, etc., are all $O(p^2 )$.
\end{coro}

\begin{proof}
Combine Zhang's formulas (\ref{Zhang}) and (\ref{ZhangNT}) with  Propositions \ref{heightXThetae}, 
\ref{heightXSquareThetae} and \ref{InegalBost}. $\Box$
\end{proof}

\subsection{Estimates on Green-Zhang functions for $J_0 (p)$}
\label{greenzhang}

We shall later on need some control on the $p$-adic N\'eron-Tate metric of $\Theta$ 
as alluded to in Remark~\ref{symmetricHodge}. (Those statements can probably be best formulated in the setting of Berkovich theory, for which one might check in 
particular~\cite{Du06}, Proposition~2.12, and~\cite{Th05}. A useful point of view is also proposed by that of ``tropical jacobians'', see~\cite{MZ08} and~\cite{dJ17}.
We will content ourselves here with our down-to-earth point of view). We therefore define 
$$
\hat{\Phi}_p :=\lim_{\stackrel{\longrightarrow}{K_{\mathfrak{P}} \supseteq \Q_p}} \Phi_{\mathfrak{P}}
$$ 
as the direct limit, on a tower of totally ramified extensions $K_{\mathfrak{P}} /\Q_p$, of the component groups $\Phi_{{\mathfrak{P}}}$ of 
the N\'eron models of $J_0 (p)$ at ${\mathfrak{P}}$, see~(\ref{ComponentGroup}). The compatible embeddings
$$
Z:=\langle C_0 -C_\infty \rangle \simeq \langle (0)-(\infty )\rangle \simeq \Z /N_0 \Z   \hookrightarrow \Phi_{{\mathfrak{P}}} 
$$ 
for each ${\mathfrak{P}}$ induce an exact sequence $0\to Z \to \hat{\Phi}_p \to \varinjlim_{e_{\mathfrak{P}}} (\Z /e_{\mathfrak{P}} \Z )^g \simeq (\Q /\Z )^g \to 0$.
Passing to the real completion yields a presentation:
\begin{eqnarray}
\label{presentationPhi} 
0\to Z \simeq \Z /N_0 \Z  \to \hat{\Phi}_{p,\R} \to  (\R  /\Z )^g  \to 0
\end{eqnarray}
(where $\hat{\Phi}_{p,\R}$ must be the ``skeleton'', in the sense of Berkovich, of the N\'eron model over $\overline{\Z}_p$ of $J_0 (p)$, and the tropical jacobian,
see~\cite{dJ17}, of the curve $X_0 (p)$ above $p$). The right-hand side of~(\ref{presentationPhi}) is more canonically written $(\R  /\Z )^g \simeq (\R /\Z )^s / \Delta (\R )$, for $\Delta$
the almost diagonal map 
$$
\Delta  (z) \mapsto (\frac{1}{w_i} z )_{1\le i\le g+1}
$$
(see~\cite{LF12}, Proposition~2.11.(c)).

\medskip

  We then sum-up useful properties about theta divisors and theta functions ``over $\overline{\Z}$''.

\medskip

  As $J_0 (p)$ is principally polarized over $\Q$, the complex extension of scalars $J_0 (p) (\C )$ can be given a classical complex 
uniformization $\C^g /(\Z ^g +\tau \Z^g )$ for some $\tau$ in Siegel's upper half plane. The associated Riemann theta function:  
\begin{eqnarray}
\label{theta}
\theta (z) =\sum_{m\in \Z^g } \exp (i\pi {}^t m\cdot \tau \cdot m +2i\pi  {}^t m\cdot z)
\end{eqnarray}
defines the tautological global section $1$ of a trivialization of ${\cal O}_{J_0 (p)}  ({\Theta}_{\C } ) (={\cal M}_{\C}^{\otimes 1/N} )$ for $\Theta_\C$ the image $W_{g-1}$
of some $(g-1)$st power of $X_0 (p)$ in $J_0 (p)$. More precisely, 
Riemann's classical results (e.g. \cite{GH78}, Theorem on p.~338) assert that $\mathrm{div} (\theta (z)) =\Theta_\C$ is the divisor with support $\{ \kappa_{P_0}  +
\sum_{i=1}^{g-1} \iota_{P_0} (P_i ), P_i \in X_0 (p) (\C ) \}$, where for any $P_0 \in X_0 (p)(\C )$ we write 
$\iota_{P_0} \colon X_0 (p) \hookrightarrow J_0 (p)$ for the 
Albanese morphism with base point $P_0$, and $\kappa =\kappa_{P_0} =``\frac{\iota_{P_0} (K_{X_0 (p)})}{2} "$ for the image of Riemann's characteristic, which 
is some pre-image under duplication in $J_0 (p)$ of the image of some canonical divisor: $\omega^0 =\iota_{P_0} (K_{X_0 (p)})$ 
(see Theorem~\ref{TheoremTheta} above).

Among the translates $\Theta_{D} =t^*_{D} \Theta$, for $D\in J_0 (p) (\C)$, of the above symmetric $\Theta$, the divisor
$\Theta_\kappa =t^*_{\kappa} \Theta  =\sum_{i=1}^{g-1} \iota_\infty (X_0 (p)_\Q )$ defines an invertible sheaf $L (\Theta_\kappa )$ on $J_0 (p)$ over $\Q$. 
If ${\cal N}_{J ,1}$ denotes the neutral component of the N\'eron model of $J$ over $\Z$ and ${\cal L}(\Theta_\kappa )$ is the cubist extension of $L (\Theta_\kappa )$ 
to ${\cal N}_{J ,1}$ (compare~\cite{MB86}, Proposition~II.1.2.2, as in Section~\ref{Moret} above), we know that $H^0 ({\cal N}_{J ,1} ,{\cal L}(\Theta_\kappa ) )$ is a (locally...) 
free ${\Z}$-module of rank $1$, so that the complex base-change $H^0 (J_0 (p) (\C ) ,L (\Theta_{\kappa ,\C} ))$ is similarly a complex line. This means that 
if $s_{\theta}$ is a generator of the former space, whose image in the later we denote by $s_{\theta ,\C}$, there is a nonzero complex number~$C_{\vartheta}$ such that 
\begin{eqnarray}
\label{constanteC}
s_{\theta , \C} (z )=C_{\vartheta} \cdot \theta (z +\kappa ) .
\end{eqnarray}
Up to making some base-change from $\Z$ to some ${{\cal O}_K}$  we can now
forget about $\kappa$ and come back to the symmetric $\Theta$: we define a global section 
\begin{eqnarray}
\label{laSM0}
s_{{\cal J}^0} :=  ( t^*_{-\kappa} )s_{\theta}  \in H^0 ({\cal N}_{J ,1 } , {{\cal L}(\Theta )}_{{\cal O}_K} ) \hspace{0.5cm} {\mathrm{so\ that}} \hspace{0.5cm} s_{{\cal J}^0 ,\C } (z) =C_\vartheta\cdot \theta(z) .
\end{eqnarray}

  If one replaces ${\cal N}_{J ,1 }$ by the N\'eron model, say ${\cal N}_{{\cal O}_{K_1}}$, of $J_0 (p)$ over any extension $K_1$ of $K$, then~\cite{MB86}, 
Proposition~II.1.2.2 insures that up to making some further field extension $K_2 /K_1$ the sheaf ${{L}(\Theta )}_{K_2}$ has a cubist 
extension ${{\cal L}(\Theta )}_{{\cal O}_{K_2}}$ to ${\cal N}_{{\cal O}_{K_1}} \times_{{\cal O}_{K_1}} {{\cal O}_{K_2}}$. Therefore $s_{{\cal J}^0}$ extends to 
a {\it rational} section (we shall sometimes write {\it meromorphic} section) of ${{\cal L}(\Theta )}_{{\cal O}_{K_2}}$
on ${\cal N}_{{\cal O}_{K_1}} \times_{{\cal O}_{K_1}} {{\cal O}_{K_2}}$. Abusing notations we still denote that extended section by $s_{{\cal J}^0}$, and write 
accordingly $\Theta$ for its divisor ${\mathrm{div}} (s_{{\cal J}^0} )$ on ${\cal N}_{{\cal O}_{K_1}} \times_{{\cal O}_{K_1}} {{\cal O}_{K_2}}$. Because $s_{{\cal J}^0}$ 
is well-defined (and non-zero) on the neutral component of the N\'eron model, its poles on ${\cal N}_{{\cal O}_{K_1}} \times_{{\cal O}_{K_1}} {{\cal O}_{K_2}}$ can 
only show-up at places of bad reduction. 

\begin{propo}
\label{padicmetric0}

The multiplicity of the $\Theta$-divisor at any component of the N\'eron model of $J_0 (p)$ over $\overline{\Z}$, normalized to be $0$ along the neutral component,
is $O(p )$. %

\end{propo}

\begin{proof}

We start by the following observations. Let us write $s_{{\cal J}^0 ,\C } (z) =C_\vartheta\cdot \theta(z)$ as in~(\ref{laSM0}).
Take $D$ in $J_0 (p)(\C )$ which can written as the linear equivalence class of some divisor 
$$ 
D= \sum_{i=1}^g -(Q_i -\infty ) 
$$
for points $Q_i$ in $X_0 (p)(\C )$. We associate to $D$ the embedding:
$$
\iota_{\kappa +D} \colon \left\{ 
\begin{array}{rcl}
X_0 (p) & \hookrightarrow & J_0 (p) \\
P & \mapsto & {\mathrm{cl}} (P-\infty +\kappa +D ) 
\end{array}
\right.
$$
where $\kappa$ is Riemann's characteristic (see just before~(\ref{ThetaHeight2})). For such a $D$ whose $Q_i$ are assumed to belong to $X_0 (p)(\overline{\Q})$, 
we know from the proof of Theorem~\ref{ThetaZhang} (see~(\ref{ThetaHeight})) that 
\begin{eqnarray}
\label{ThetaDkappa1}  
\height_\Theta (\iota_{\kappa +D} (P)) =\frac{1}{[K(P,D) :\Q ]} [P,{\tilde{\omega}}_D  ]_{\mu_0}
\end{eqnarray}
with 
\begin{eqnarray}
\label{explicitdivisor00}
{\tilde{\omega}}_D =\sum_i Q_i  +\Phi_D +c_D X_\infty
\end{eqnarray} 
and $\Phi_D$ is the explicit vertical divisor 
\begin{eqnarray}
\label{explicitdivisor}
\Phi_D = \frac{1}{2} \left( \Phi_\omega +\Phi_{\vartheta}  \right) - \sum_{i=1}^{g} \Phi_{Q_i}  
\end{eqnarray}
at each bad place, with notations as those of the proof of~Theorem~\ref{ThetaZhang}, see~(\ref{ThetaDivisor}).

  Moreover, it is well-known that there is a subset of $J_0 (p)(\C )$ which is open for the complex topology, and even the Zariski topology, in which all 
points $D=\sum_1^g -(Q_i -\infty)$ as above are such that 
\begin{eqnarray}
\label{RR}  
\dim_\C H^0 (X_0 (p)(\C ), L (-D+g\cdot \infty )_\C ) =\dim_\C H^0 (X_0 (p)(\C ), \iota_{\kappa +D}^* L (\Theta_\C ))  =1
\end{eqnarray}
so that $\iota_{\kappa +D}^* (\Theta_\C ) =\sum_i Q_{i,\C}$, the latter being an equality between effective divisors, not just a linear equivalence 
(\cite{GH78}, pp.~336--340). As the height $\height_\Theta$, in the N\'eron model of $J_0 (p)$, can be understood as the Arakelov intersection 
with $\Theta ={\mathrm{div}} (s_{{\cal J}^0} )$
it follows that, on the curve $X_0 (p)$, ${\mathrm{div}} (s_{{\cal J}^0 ,\C} )\cap \iota_{\kappa +D} (X_0 (p))(\C ) =
\cup_i \iota_{\kappa +D} (Q_{i,\C} )$, or ${\mathrm{div}} (\iota_{\kappa +D}^* (s_{{\cal J}^0 ,\C} ) )= \sum_i Q_i$ over $\C$. 
More precisely, extending base 
to some ring of integers ${\cal O}_K$ so that the $Q_i$ define sections of the minimal regular model ${\cal X}_0 (p)_{{\cal O}_K}$ of $X_0 (p)$ over ${\cal O}_K$, 
and making if necessary a further base extension such that ${\cal L}(\Theta )$ has a cubist extension on the whole N\'eron model of $J_0 (p)$ over ${{\cal O}_K}$ 
(as after~(\ref{laSM0})), one sees that $s_{{\cal J}^0 }$ defines a meromorphic section of ${\cal L}(\Theta )_{{\cal O}_K}$ and the restriction to the generic fiber
${X}_0 (p)_K$ of ${\mathrm{div}} ( \iota_{\kappa +D}^* (s_{{\cal J}^0 } ))$ has to be equal (and not merely linearly equivalent) to $\sum_i Q_i $.
Now in such a situation, the multiplicity of ${\mathrm{div}} ( s_{{\cal J}^0 } )$ on a component of the 
N\'eron model to which ${\cal X}_0 (p)_{{\cal O}_K}^{\mathrm{smooth}}$ is mapped via $\iota_{\kappa +D}$, can be read on the multiplicity 
of $\iota_{\kappa +D}^* (s_{{\cal J}^0} )$ along that component of ${\cal X}_0 (p)_{{\cal O}_K}^{\mathrm{smooth}}$. In turn, because of decompositions
of the arithmetic Chow group similar to that of Theorem~\ref{Hodge}, multiplicities of ${\mathrm{div}} ( s_{{\cal J}^0 } )$ are determined 
by the $\Phi_D$ of $(\ref{explicitdivisor00})$, up to constant addition of vertical fibers. The property that ${\mathrm{div}} ( s_{{\cal J}^0 } )$ has multiplicity $0$ 
along the neutral component of the N\'eron model (see~(\ref{laSM0})) fixes that last indetermination.   
Now if $\mathfrak{P}$ is a place of bad reduction for ${\cal X}_0 (p)_{{\cal O}_K}$, and if the $Q_i$ move sligthly in the 
$\mathfrak{P}$-adic topology (without modifying their specialization component at $\mathfrak{P}$),  the vertical divisor $\Phi_D$ does not change either at 
$\mathfrak{P}$, and the above reasoning regarding the components values of $\Theta$ is actually independent from the fact that condition~(\ref{RR}) holds 
true or not (provided, we insist, that the specialization components of the $Q_i$ at $\mathfrak{P}$ do not vary).

\medskip

We shall gain some flexibility with a last preliminary remark. If $k$ is any integer between $0$ and $N_0 -1$ (recall $N_0$ is the order of the Eisenstein 
element $(0-\infty )$), the divisor ${\tilde{\omega}}_D$ of~(\ref{explicitdivisor00}) can still be written as 
$$
{\tilde{\omega}}_D =  \left( k\cdot 0+ (g-k)\cdot \infty  -k\, \Phi_{C_0}  +\frac{1}{2} (\Phi_\omega + \Phi_{\vartheta} )   -\tilde{D} \right) +c_D X_\infty
$$
so that if 
$$
D= \left( \sum_{i=1}^g -(Q_i -\infty ) \right) +k(0-\infty ) =\sum_{i=1}^k -(Q_i -0) +\sum_{i=k+1}^g -(Q_i -\infty  )
$$ 
then ${\tilde{\omega}}_D =\sum_{i=1}^g Q_i  +\Phi_D +c_D X_\infty$
where $\Phi_D$ is still  
\begin{eqnarray}
\label{explicitdivisor2}
\Phi_D = \frac{1}{2} \left( \Phi_\omega +\Phi_{\vartheta}  \right)  - \sum_{i=1}^{g} \Phi_{Q_i}  .
\end{eqnarray}

\medskip 
 
   Coming back to the proof of the present Proposition~\ref{padicmetric0}, and assuming first $D=0$, it follows from what we have just discussed that the multiplicity of 
the $\Theta$-divisor on the components of the jacobian to which the components 
of ${\cal X}_0 (p)_{{\cal O}_K}^{\mathrm{smooth}}$ map under $\iota_{\kappa}$ is given by the functions $g_n$ and $G$ of~(\ref{ZhangMetric1}) 
and~(\ref{ZhangMetric2}), see Theorem~\ref{ThetaZhang}. To obtain the multiplicity of the $\Theta$-divisor on {\em all} components of the jacobian
we shall shift our Albanese embeddings $\iota_{\kappa +D}$ in order to explore all of $J_0 (p)/J_0 (p)^0$ with successive translations 
of ${\cal X}_0 (p)_{{\cal O}_K}^{\mathrm{smooth}}$ inside $J_0 (p)$.

   To be more explicit, let $\mathfrak{C}$ be an element of the component group $J_0 (p)/J_0 (p)^0$ at $\mathfrak{P}$, and $D= \sum_{i=1}^g (P_i -\infty )$ 
be a divisor, with all $P_i$ in $X_0 (p)(K)$, which reduces to $\mathfrak{C}$ at $\mathfrak{P}$. For all $r$ in $\{ 1,\dots ,g\}$, set $D_r =\sum_{i=1}^{r}  
(P_i -\infty )$ and let also $k_r$ in $\{ 1, \dots ,N_0 -1 \}$ and $Q_{i,r}$ be $g$ associated points on the curve such that one can write both
$$
D_r  = \sum_{i=1}^r (P_i -\infty ) \hspace{0.5cm} {\mathrm{and}} \hspace{0.5cm} D_r= \sum_{i=1}^g  -(Q_{i,r} -\infty ) +k_r (0-\infty ) .
$$
As always in this proof, up to making a finite base-field extension one can assume all points have values in $K$. 
Recall also from the discussion above that one can move slightly the $Q_i$ in the $\mathfrak{P}$-adic topology, as all that interests us here is the component
$\mathfrak{C}_r$, $1\le r\le g$, of~$(J_0 (p)/J_0 (p)^0 )_\mathfrak{P}$ to which $D_r$ maps. One can therefore  
assume if one wishes that $\iota_{\kappa +D_r}^* (\Theta_\C ) =\sum_i Q_{i,\C}$ (equality, not just linear equivalence). 
The presentation of $\Phi_{{\mathfrak{P}}}$ given in~(\ref{presentationPhi}) and above also shows one can 
assume that the specialization components at $\mathfrak{P}$ of the $Q_{i,r}$, in ${\cal X}_0 (p)_{{\cal O}_K}^{\mathrm{smooth}}$, which are 
not $C_\infty$, are all different (see Figure~\ref{figure}).

  Taking first $D=0$, that is, using the map $\iota_{\kappa}$, we already remarked that~(\ref{explicitdivisor}) implies the 
value $V_1$ of ${\mathrm{div}} (s_{{\cal J}^0 } )$ on $\mathfrak{C}_1$ is $V_1 =[\frac{1}{2} \left( \Phi_{\vartheta} +{\Phi_\omega} \right) , P_1 ] =\frac{1}{2} 
\left( [\Phi_{\omega} , P_1] +[\Phi_{P_1} ]^2 \right)$ (see~(\ref{PhiP-PP})). By Remark~\ref{CoeffPhi} and~(\ref{sizePhiOmega}), $\vert V_1 \vert \leq 2$. 

Going one step further we reach $\mathfrak{C}_2$ by considering the Albanese image $\iota_{\kappa + D_1 } ({\cal X}_0 (p)_{{\cal O}_K}^{\mathrm{smooth}} )$ and looking 
at the image of $P_2$. Here we need not to forget that the $\infty$-cusp in $X_0 (p)$ now maps to $\mathfrak{C}_1$, so the normalization of components-divisor on the {\it curve}  
${\cal X}_0 (p)_{{\cal O}_K}^{\mathrm{smooth}}$ at ${\mathfrak{P}}$ cannot be fixed to be $0$ along the $\infty$-component any longer: it needs to take the value~$V_1$ 
found above, in order to match with the normalization of the theta divisor on the jacobian. 
Applying the same reasoning as before with formula~(\ref{explicitdivisor2}) gives that the value of $\Theta$ on $\mathfrak{C}_2$ is 
\begin{eqnarray}
V_2 & = & [P_2 ,\frac{1}{2} \left(  \Phi_\omega +\Phi_{\vartheta} \right)   - \sum_{i=1}^{g} \Phi_{Q_{i,1}}  +V_1 ] = \frac{1}{2} \left( [\Phi_{\omega} , P_2 ] 
+[\Phi_{P_2} ]^2 \right) -  \sum_{i=1}^{g} [\Phi_{Q_{i,1}} ,P_2 ] +V_1 \nonumber
\end{eqnarray}
so that $\vert V_2 \vert \leq 9$ invoking Remark~\ref{CoeffPhi} again, and recalling the $Q_{i,1}$ specialize to different branches of Figure~\ref{figure}. 

 From there the inductive process is clear which yields that the value of $\Theta$ on $\mathfrak{C}_r$ has absolute value less or equal to $7 r$, whence the proof of  
 Proposition~\ref{padicmetric0}.  $\Box$

\end{proof}

\subsection{Explicit modular version of Mumford's repulsion principle}

We conclude this section by writing-down, for later use, an explicit version of Mumford's well-known ``repulsion principle'' for points, in the case of modular curves. 
\begin{propo}
\label{Mumford}
For $P$ and $Q$ two different points of $X_0 (p) (\overline{\Q} )$ one has  
\begin{eqnarray}
\label{effectiveMumford}
{\height}_\Theta (P -Q ) \geq \frac{g-2}{4g} \left( {\height}_\Theta (P -\infty ) +{\height}_\Theta (Q -\infty ) \right)  
- O(p\log p )  .
\end{eqnarray}
\end{propo}

\begin{proof}
Let $K$ be a number field such that both $P$ and $Q$ have values in $K$. Using notations of Section~\ref{Chow}, the adjunction formula and Hodge index theorem give
\begin{eqnarray}
2 [K:\Q ]{\height}_\Theta (P-Q) & = &  -\left[ P-Q-\Phi_{P} +\Phi_Q , P-Q-\Phi_{P} +\Phi_Q  \right]_{\mu_0} \nonumber \\
                                                 & = & [P +Q,\omega ]_{\mu_0}   +2[P,Q]_{\mu_0}      +[\Phi_P -\Phi_Q ]^2    \nonumber \\
                                                 & \geq & [P +Q,\omega ]_{\mu_0}   -2[K:\Q ] \sup g_{\mu_0}      +[\Phi_P -\Phi_Q ]^2 . \nonumber
\end{eqnarray}
In the same way,
\begin{eqnarray}
[P,\omega ]_{\mu_0}  & = &      2 [K:\Q ]{\height}_\Theta (P -\infty ) -2[P,\infty ]_{\mu_0} +[\infty ]_{\mu_0}^2 -[\Phi_{P} ]^2   \nonumber \\  
                    & \geq &  [K:\Q ] {\height}_\Theta (P -\infty +\frac{1}{2} \omega^0 )    -2[P,\infty ]_{\mu_0} +[\infty ]_{\mu_0}^2 -[\Phi_{P} ]^2   \nonumber  
\end{eqnarray}
where the last inequality comes from the quadratic nature of ${\height}_\Theta$, plus the fact that the 
error term of~(\ref{effectiveMumford}) allows us to 
assume ${\height}_\Theta (P-\infty )\geq \frac{1}{12-8\sqrt{2} } {\height}_\Theta ( \omega^0 )=O(\log p)$ (see~(\ref{tailleOmega}) and the end of proof of 
Theorem~\ref{TheoremTheta}). Now by~(\ref{lesestimees}),
$$
{\height}_\Theta (P -\infty +\frac{1}{2} \omega^0 ) =\frac{1}{[K:\Q ]} [P,g\cdot \infty ]_{\mu_0} +O(\log p)
$$ 
and using Remark~\ref{CoeffPhi}  and Lemma~\ref{0-infty} gives
$$
[P,\omega ]_{\mu_0} \geq \frac{g-2}{g} [K:\Q ] {\height}_\Theta (P -\infty +\frac{1}{2} \omega^0 ) +[K:\Q ] O( \log p).
$$
As $[\Phi_P ,\Phi_Q ]=[P,\Phi_Q ]=[Q ,\Phi_P ]$, we have $\vert [\Phi_P ,\Phi_Q ] \vert \leq 3 [K:\Q ]\log p$
using Remark~\ref{CoeffPhi} again. Putting everything together with Remark~\ref{CorrectOrder} about $\sup g_{\mu_0}$ we obtain
$$
{\height}_\Theta (P -Q ) \geq \frac{g-2}{2g} \left( {\height}_\Theta (P -\infty +\frac{1}{2} \omega^0) +
{\height}_\Theta (Q -\infty +\frac{1}{2} \omega^0) \right)  
- O(p\log p ) 
$$ 
which, by our previous remarks, can again be written as
$$
{\height}_\Theta (P -Q ) \geq \frac{g-2}{4g} \left( {\height}_\Theta (P -\infty ) +{\height}_\Theta (Q -\infty ) \right)  
- O(p\log p ) .\hspace{3cm} \Box
$$ 
\end{proof}

\medskip

(For large $p$, the angle between two points of equal large enough height is here therefore at least $\arccos (3/4 )-\varepsilon > \pi  /6$. Of course the 
natural value is $\pi/2$, to which one tends when sharpening the computations.)

\section{Arithmetic B\'ezout theorem with cubist metric}
\label{ArithmeticBezout}

We display in this section an explicit version of B\'ezout arithmetic theorem, in the sense of Philippon or 
Bost-Gillet-Soul\'e (\cite{Ph95+}, \cite{BGS94}), for intersections of cycles in our modular abelian varieties  
over number fields, with the following variants: we use Arakelov heights (as in 
Section~\ref{LaHauteurDeToutLeMonde} above, see~(\ref{lavraiehauteur}))
on higher-dimensional cycles, and we endow the implicit hermitian sheaf for this height with its cubist 
metric (instead of Fubini-Study). 

   It indeed seems that one generally uses Fubini-Study metrics for arithmetic B\'ezout because they are the only natural explicit ones available on 
a general projective space (a necessary frame for the approach we follow for B\'ezout-like statements). They moreover have the pleasant feature that the relevant projective 
embeddings have tautological basis of global 
sections with sup-norm less than $1$ which, for instance, allows for proving that the induced Faltings height is non-negative on effective
cycles (see~\cite{Fa91}, Proposition~2.6). For our present purposes however, we need bounds for the 
N\'eron-Tate heights of points, that is, Arakelov heights induced by cubist metrics.  One could in principle have tried working with Fubini-Study metrics 
as in \cite{BGS94} and then directly compare  with N\'eron-Tate heights, but comparison terms tend to be huge. In the case of rational points, for instance 
(that is, horizontal cycles of relative dimension $0$), within jacobians, those error terms are bounded by Manin and Zarhin (\cite{MZ72}) linearly in the ambient projective 
dimension, that is exponential in the dimension of the abelian variety. In other words, for our modular curves, the error terms would be exponential in the level $p$. 
It is therefore much preferable to stick to 
cubist metrics. This implies we avoid the use of joins as in \cite{BGS94}, as those need a sheaf metrization on the whole of the ambient projective 
spaces, and we instead use plain Segre embeddings. The extra numerical cost essentially consists of the appearance of modest binomial coefficients, 
which do not significantly alter the quantitative bounds we eventually obtain. 

  We also need to work with projective models which are ``almost'' compactifications of relevant N\'eron 
models of our jacobians. This we do with the help of Moret-Bailly theory as introduced in 
Section~\ref{LaHauteurDeToutLeMonde}.  

   Let us also recall that there still is another approach for such arithmetic B\'ezout theorems which uses Chow forms (\cite{Ph95+}, \cite{Re00}). That is 
however known to amount to working again with Faltings height relative to the Fubini-Study metrics (\cite{Ph95+}-I, \cite{So91}) that  we said we cannot afford.

   Finally, regarding generality: it would of course be desirable to have a proof available for 
arbitrary abelian varieties. Many of the present arguments are however quite particular to our 
application to $J_0 (p)$. We therefore prefer working  in our concrete setting from the beginning, 
instead of considering a somewhat artificial generality.

\begin{propo} {\bf (Arithmetic B\'ezout theorem for $J_0 (p)$).}
\label{Bezout}
Let $(J_0 (p),\Theta )$ be defined over some number field $K$, endowed with the principal and 
symmetric polarization $\Theta$. Let $V$ and $W$ be two irreducible $K$-subvarieties of $J_0 (p)$, of 
dimension $d_V:=\dim_K V$ and $d_W :=\dim_K W$ respectively, such that 
$$
d_V +d_W \leq g =\dim J_0 (p) 
$$
and assume $V\cap W$  has dimension $0$. %

\medskip

   If $P$ is an element of $(V\cap W) (K )$ then its N\'eron-Tate $\Theta$-height satisfies  
\begin{eqnarray}
\label{contentOfBezout}
{\height}_{\Theta} (P) & \leq & 
\frac{4^{d_V +d_W}}{2} \, \frac{(d_V +d_W +1)!}{d_V ! \, d_W !} \deg_{\Theta} (V) \deg_{\Theta} (W)
\Big[ (d_W +1){\height}_{\Theta}  (W)  + (d_V +1){\height}_{\Theta}  (V) \nonumber \\
 & & \hspace{9.5cm} +O(p\log p) \Big]  .
\end{eqnarray}
\end{propo}

\medskip

\begin{rema}
{\rm{The general aspect of the above release of arithmetic B\'ezout might look a bit different from the 
original ones, as can be found in~\cite{BGS94}: this is due to the fact that our definition of the 
height of some cycle $Y$ (see Section~\ref{LaHauteurDeToutLeMonde}, (\ref{lavraiehauteur})) amounts to
dividing its height in the sense of~\cite{BGS94} by the product of the degree and absolute dimension 
of $Y$.

}}
\end{rema}

\medskip

   Let us first sketch the strategy of proof, which occupies the rest of this 
Section~\ref{ArithmeticBezout}. We henceforth fix a prime number $p$ and some perfect square 
integer $N:= r^2$. (We shall eventually take $r=2$.) We write $({\cal J},{\cal M})$
for the Moret-Bailly projective model of $(J_0 (p),L(\Theta )^{\otimes N} )$ given by Definition~\ref{GoodModels}, 
relative to some given set of global sections ${\cal S}$ in $H^0 ({\cal N}_{J,N} , {\cal L}(\Theta 
)^{\otimes N} )$, of size $N^{g}$, to be described later (Lemma~\ref{lespetites}). That model is defined 
over some ring of integers ${\cal O}_K$. Consider the morphisms: 
\begin{eqnarray}
\label{Bezout0}
\begin{array}{rccl}
{\cal J} \stackrel{\Delta}{\longrightarrow} & {\cal J}\times {\cal J} &   & \\
                                                               &   {\cal P} \downarrow    &     \searrow \iota   &          \\
                                                               &  \PPP^n_{{\cal O}_K} \times \PPP^n_{{\cal O}_K} & \stackrel{S}{\longrightarrow}  & \PPP^{n^2 +2n}_{{\cal O}_K} 
\end{array}
\end{eqnarray}
where $\Delta$ is the diagonal map, $n=N^{g} -1$, ${\cal P}$ is the product of two ${{\cal S}}$-embeddings ${\cal J} \stackrel{\jmath}{\hookrightarrow} \PPP^n =
\PPP^n_{{\cal O}_K}$ and the application $\iota \colon {{\cal J}\times {\cal J} \to \PPP^{n^2 +2n}}$ is the composition of the Segre embedding $S$ with ${\cal P}$. 
As sheaves,  
$$
S^* ({\cal O}_{\PPP^{n^2 +2n}}(1)) = {\cal O}_{\PPP^n}(1) \otimes_{{\cal O}_K} {\cal O}_{\PPP^n} (1)  
$$
and
$$
{\cal P}^* ({\cal O}_{\PPP^n}(1) \otimes_{{\cal O}_K}  {\cal O}_{\PPP^n} (1)) ={\cal M} \otimes_{{\cal O}_K}  {\cal M} =: 
{{\cal M}^{\boxtimes 2}} 
$$
so that 
$$
 \iota^* ({\cal O}_{\PPP^{n^2 +2n}} (1)) ={{\cal M}^{\boxtimes 2}}
$$
and
\begin{eqnarray}
\label{mDeux}
\Delta^* \iota^* {\cal O}_{\PPP^{n^2 +2n}} (1) ={\cal M} \otimes_{{\cal O}_{\cal J}}  {\cal M} = {\cal M}^{\otimes 2} .
\end{eqnarray}
We naturally endow the sheaves ${\cal M}^{\boxtimes 2}$, ${\cal M}^{\otimes 2} $, and so on, with the hermitian structures 
induced by the cubist metric on the various ${\cal M}_{\sigma}$ for $\sigma\colon K\hookrightarrow \C$, denoted by $\| \cdot \|_{\mathrm{cub}}$.

  We then pick two copies $(x_{i} )_{0\leq i\leq n}$ and  $(y_{j} )_{0\leq j\leq n}$ of the canonical basis of global sections for each ${\cal O}_{{\PPP^n}} (1)$ on the two 
factors of $\PPP^n_{{\cal O}_K} \times \PPP^n_{{\cal O}_K}$ of~(\ref{Bezout0}), which give our basis ${\cal S}$ by restriction to ${\cal J}$. Then we provide the 
sheaf ${\cal O}_{\PPP^{n^2 +2n}} (1)$ on $\PPP^{n^2 +2n}_{{\cal O}_K}$ with the basis of global sections $(z_{i,j} )_{0\leq i,j\leq n}$, each of which is mapped 
to $x_i \otimes_{{\cal O}_K} y_j$ under $S^*$. Define ${\cal D}$ as the diagonal linear subspace of $\PPP^{n^2 +2n}_{{\cal O}_K}$ defined by the {\it linear} 
equations $z_{i,j} =z_{j,i}$ for all $i$ and $j$.

   Let $V,\, W\subseteq J={\cal J}_K$ be two closed subvarieties over $K$. The support of $V\cap W$ is 
the same as that of  $(\iota \circ \Delta )^{-1} ({\cal D}\cap \iota (V\times W))$. To bound from above 
the height of points in $V\cap W$ it is therefore sufficient to estimate Faltings' height of 
${\cal D}\cap \iota (V\times W)$, relative to the hermitian line bundle ${\cal O}_{\PPP^{n^2 +2n}} 
(1)_{\vert \iota (J\times J)}$ endowed with the cubist metric. As ${\cal D}$ is a linear subspace
that height is essentially the same as that of $(V\times W)$, up to an explicit error term which
depends on the degree. In turn this error term is a priori linear in the number of (relevant) equations for ${\cal D}$, and this is way too high. 
But if one knows $V\cap W$ has dimension $0$, it is enough to choose $(\dim V +\dim W)$ equations (up to perhaps increasing a bit the 
size of the set whose height we estimate), which makes the error term much smaller.

\medskip

   That is the basic strategy of proof for Proposition~\ref{Bezout}. To make it effective however we must control the ``error terms'' alluded to in the preceding
lines, and those crucially depend on the supremum, on the set $\cal S$, of values for the cubist metric of global sections defining the projective embedding ${\cal J} \hookrightarrow 
\PPP^n_{{\cal O}_K}$. We shall build that $\cal S$ using theta functions as follows.

\medskip

  Recall Riemann's theta function on $J_0 (p)$ introduced in Section~\ref{greenzhang}, see~(\ref{theta}). Its usual analytic norm is 
\begin{eqnarray}
\label{normtheta}
\| \theta (z ) \|_{\mathrm{an}} := \det ({\Im (\tau )})^{1/4} \exp (-\pi y \, \Im (\tau )^{-1} y) \vert \theta (z )\vert 
\end{eqnarray}
for $z=x+iy \in \C^g$  (see~\cite{Mo90}, (3.2.2)). That analytic metric  will have to be compared to the cubist one, about which we recall 
the following basic facts. 

\medskip

   Let $A$ be an abelian variety over a number field $K$, which extends to a semiabelian scheme ${\cal A}$ over the ring of integers ${\cal O}_K$.
We endow ${\cal A}$ with a symmetric ample invertible sheaf ${\cal L}$. Define, for $I\subseteq \{ 1,2,3\}$, the projection $p_I \colon {\cal A}^3 \to {\cal A}$, 
$p_I (x_1 ,x_2 ,x_3 )=\sum_{i\in I} x_i$. 
It is known to follow from the theorem of the cube (\cite{MB86}) that the sheaf ${\cal D}_3 ({\cal L}) :=\bigotimes_{I\subseteq \{ 1,2,3\}} 
p_I^* {\cal L}^{\otimes (-1)^{|I|}}$ is trivial on ${\cal A}^3$. Let us therefore fix an isomorphism $\phi \colon {\cal O}_{{\cal A}^3} \to {\cal D}_3 
({\cal L})$. For every complex place $\sigma$ of ${\cal O}_K$ one can endow ${\cal L}_\sigma$ with some cubist metric $\| \cdot \|_\sigma$ 
such that one obtains through $\phi$ the trivial metric on ${\cal O}_{{\cal A}^3}$. Each cubist metric $\| \cdot \|_\sigma$ is determined only up to multiplication 
by some constant factor so we perform the following rigidification to remove that ambiguity. If $0_{\cal A} \colon {\mathrm{Spec}} ({\cal O}_K )\to {\cal A}$ denotes the 
zero section, we replace ${\cal L}$ by ${\cal L}\otimes_{{\cal O}_K} (\pi^* 0_{\cal A}^* {\cal L}^{\otimes -1} )$ on ${\cal A}$ . Then 
$$
0_{\cal A}^* ({\cal L})\simeq {\cal O}_K
$$
and we demand that the $\| \cdot \|_\sigma$ be adjusted so that the above sheaf isomorphism is an isometry at each $\sigma$, where ${\cal O}_K$ is endowed with the 
trivial metric so that $\| 1\| =1$. This uniquely determines our cubist metrics $\| \cdot \|_\sigma$.
Now by construction the hermitian sheaf ${\cal L}$ on ${\cal A}$ defines a height $\height$ verifying the expected normalization condition $\height (0 )=0$. 

Having the same curvature form, the analytic and cubist metrics are known to differ by constant factors, at each complex place, on the Theta sheaf, as we shall use in the 
proof of Lemma~\ref{lacompacte} below.

\medskip

Recall we also defined in~(\ref{laSM0}) a ``meromorphic theta function $s_{{\cal J}^0}$ over ${\overline{\Z}}$'', which 
can be generalized: we have $[r ]^* {{\cal L}(\Theta )}_{\vert {\cal N}_{J ,1}} \simeq {{\cal L}
(\Theta )}^{\otimes r^2}$ on ${\cal N}_{J ,r}$ (\cite{Pa12}, Proposition~5.1) so we define a global section 
\begin{eqnarray}
\label{laSM}
s_{\cal M} :=  ([r ]^* t^*_{-\kappa} )s_{{\cal J}^0}  \in H^0 ({\cal N}_{J ,r } , [r ]^* 
{{\cal L}(\Theta )}_{{\cal O}_K} ).
\end{eqnarray}
We will shortly show how to control the supremum of $\| s_{{\cal J}^0} \|_{\mathrm{cub}}$, therefore of 
$\| s_{\cal M} \|_{\mathrm{cub}}$, on $J_0 (p) (\C )$ (see~Lemma~\ref{lacompacte}). 
Writing $N=r^2$, we shall moreover fix the morphism $\jmath_{\cal M} \colon \widetilde{{\cal N}}_{J,N} \to {\cal J} 
\hookrightarrow {\PPP}^n_{{\cal O}_K}$ of~(\ref{modelA}) 
by mapping the canonical coordinates $(x_i )_{0\leq i\leq n}$ to sections $(s_i )$ which will be 
translates by $r$-torsion points of a multiple of the above $s_{\cal M}$ by some constant, as explained in 
Lemma~\ref{lespetites} and its proof. 

This will allow us to control as well the supremum of those $s_i$, relative to the cubist metrics, on the 
complex base change of our abelian varieties, as is required by the proof of arithmetic B\'ezout theorems.  

\medskip

We now start the technical preparation for the proof of Proposition~\ref{Bezout}, for which we need some Lemmas on the behavior of 
heights and degree under Segre maps, comparison between cubist and analytic metrics on theta functions, and estimates for all. 

\medskip

\begin{lem}
\label{LeNouvelEssentiel}
There is an infinite sequence $(P_i )_{i\in \N}$ of points in $X_0 (p) (\overline{\Q} )$ which are ordinary at all places dividing $p$ and 
have everywhere integral $j$-invariant. Moreover their normalized theta height satisfies ${\height}_{\Theta}
(P_i -\infty +\frac{1}{2} \omega^0 )=O (p^3 )$, with notations of 
Theorem~\ref{ThetaZhang}. 
\end{lem}

\begin{proof}
Let $(\zeta_i )_\N$ be a infinite sequence of roots of unity. One can assume none are congruent to some supersingular $j$-invariant in 
characteristic $p$, modulo any place of $\overline{\Q}$ above $p$. (Indeed, as the supersingular $j$-invariants are quadratic over $\F_p$, 
it is enough for instance to choose for the $\zeta_i$ some primitive $\ell_i$-roots of unity, 
with $\ell_i$ running through the set of primes larger than $p^2 -1$.) 
Lift each $j$-invariant equal to $\zeta_i$ to some point $P_i$ in $X_0 (p)(\overline{\Q})$.  By construction, this makes a sequence of points 
with $j$-height $\height_j (P_i )$ equal to $0$. As for their (normalized) theta height one sees from Theorem~\ref{ThetaZhang} that %
$$
{\height}_\Theta (P_i -\infty +\frac{1}{2} \omega^0 ) = \frac{1}{[K(P_i ):\Q ]} [P_i ,\tilde{\omega}_{\Theta} ]_{\mu_0} =  \frac{-1}{[K(P_i ) :\Q ]} 
\sum_{\sigma \colon K (P_i )\hookrightarrow \C} g\cdot g_{\mu_0} (\infty ,\sigma (P_i )) +O(\log p) 
$$
as the contribution at finite places of $[P_i ,\infty ]$ is $0$.
It is therefore enough to bound the $\vert g_{\mu_0} (\infty ,\sigma (P_i)) \vert$. 

  Now $\vert j(P_i )\vert_\sigma =1$ for all $\sigma \colon K (P_i )\hookrightarrow \C$, so the corresponding elements 
$\tau$ in the usual fundamental domain in Poincar\'e upper half-plane for $X_0 (p)$ or $X (p)$ are absolutely bounded, 
and the same for the absolute values of $q_\tau =e^{2i\pi \tau}$. (For a useless explicit estimate of this bound, one 
can check Corollary~2.2 of~\cite{BP11b} which proposes $\vert q_\tau \vert \geq e^{-2500}$.)  From this, running 
through the proof of Theorem~11.3.1 of~\cite{EJ11c}, and adapting it to the case of $X_0 (p)$ instead of $X_1 (pl)$, we 
deduce that the $\sigma (P_i )$ do not belong to the
open neighborhood, in the atlas of loc. cit., of the cusp $\infty$ in $X_0 (p) (\C )$. Therefore Proposition~10.13 
of~\cite{Me11} applies and gives, with notations of that work,
\begin{eqnarray}
\vert g_{\mu_0} (\infty ,\sigma (P_i ))  \vert = \vert g_{\mu_0} (\infty ,\sigma (P_i ))  -h_\infty (\sigma (P_i )) \vert =O(p^2 )
\end{eqnarray}
(see Theorem~11.3.1 of~\cite{EJ11c} and its proof). $\Box$
\end{proof}

\medskip

\begin{lem}
\label{lacompacte}
Let $s_{\theta}$ be the ``theta function over $\Z$'', that is, the global section introduced just before~(\ref{constanteC}). 
One has:
\begin{eqnarray}
\label{supTheta2}
\sup_{J_0 (p) (\C )} (\log \| s_{\theta} \|_{\mathrm{cub}} )\leq  O(p\log p ).
\end{eqnarray}
\end{lem}
\begin{proof}

Writing $s_{{\theta}, \C} (z) =C_\vartheta\cdot \theta(z+\kappa )$ as in~(\ref{constanteC}), we shall bound from above
both $\vert C_\vartheta\vert$ and the contribution of the difference between cubist and analytic metrics. Then we will use upper
bounds for the analytic norm of the theta function due to P.~Autissier and proven in the Appendix of the present paper. 

 We invoke again some key arguments of the proof of Proposition~\ref{padicmetric0}. For $D$ 
in $J_0 (p)(\C )$, written as the linear equivalence class of some divisor $\sum_{i=1}^g (P_i -\infty )$ on $X_0 (p)(\C )$, we indeed 
once more consider the embedding
$$
\iota_{\kappa -D} \colon \left\{ 
\begin{array}{rcl}
X_0 (p) & \hookrightarrow & J_0 (p) \\
P & \mapsto & {\mathrm{cl}} (P-\infty +\kappa -D ) 
\end{array}
\right.
$$
as in Proposition~\ref{padicmetric0}. For such a $D$ whose $P_i$ are assumed to belong to $X_0 (p)(\overline{\Q})$, we recall~(\ref{ThetaDkappa1}) that
$$
\height_\Theta (\iota_{\kappa -D} (P)) =\frac{1}{[K(P,D) :\Q ]} [P,\sum_i P_i  +\Phi_D +c_D X_\infty ]_{\mu_0} .
$$
If the $P_i$ all have everywhere ordinary reduction, as will be the case in~(\ref{logC}) below, the vertical divisor~$\Phi_D$ will contribute 
at most $O(\log p)$ to the height of points (see Remark~\ref{CoeffPhi}).  

   Note that we can fulfill condition~(\ref{RR}) considering only points $P_i$ of same type as occurring in Lemma~\ref{LeNouvelEssentiel} 
(which, in particular, are ordinary and have integral $j$-invariants), because those $P_i$ make a Zariski-dense subset of $X_0 (p) 
(\overline{\Q}) $ (and the onto-ness of the map $X_0 (p)^{(g)} 
\stackrel{\iota_\infty^g}{\twoheadrightarrow} J_0 (p)$). We therefore conclude as in the proof of Proposition~\ref{padicmetric0} 
that  ${\mathrm{div}} ( \iota_{\kappa -D}^* (s_\theta ))$ has indeed to be $(\sum_i P_i  +\Phi_D )$ on $X_0 (p)_{{\cal O}_K}^{\mathrm{smooth}}
$.\footnote{Although we shall not use this, one can check that $\height_\Theta (\iota_{\kappa -D} (\infty )) 
=\| -(\sum_i P_i -\infty ) +\frac{1}{2} \omega^0 \|^2_\Theta =O(p^5 ) $ by Lemma~\ref{LeNouvelEssentiel} 
and~(\ref{tailleOmega}).}

  On the other hand, for some of those choices of $(P_i )_{1\leq i\leq g}$, our $\Z$-theta function $s_{\theta}$ does not vanish 
at $\iota_{\kappa -D} (\infty ) (\C )$, so $\height_\Theta (\iota_{\kappa -D} (\infty ))$ can also be computed as the Arakelov degree:
$$
\height_\Theta (\iota_{\kappa -D} (\infty ))=\widehat{\deg} (\infty^* \iota_{\kappa -D}^* (  {\cal L} (\Theta ) )) . %
$$ 
Integrality of the $P_i$ shows the intersection numbers $[\infty ,P_i]$ have trivial non-archimedean contribution. The only finite contribution 
to our Arakelov degree therefore comes from intersection with vertical components, that is, if $K_D$ is a sufficiently large field, over 
which $D$ is defined, then for a set of elements $(z_{\sigma} )_{\sigma \colon K_D \hookrightarrow {\overline{\Q}}}$ which lift $\sigma (-D)$ 
in the complex tangent space of $J_0 (p)$ to $0$ one has:
\begin{eqnarray}
\height_\Theta (\iota_{\kappa -D} (\infty )) & = & \widehat{\deg} (0_{{\cal J}_0 (p)}^* (t^*_{\kappa -D} {\cal L} (\Theta ) )) =\widehat{\deg} 
(0_{{\cal J}_0 (p)}^* (t^*_{-D} {\cal L} (\Theta_\kappa ) )) \nonumber \\
 & = &   -\frac{1}{[K_D :\Q ]} \sum_{ K_D \stackrel{\sigma}{\hookrightarrow} \C} \log \| s_{\theta} (z_\sigma ) \|_{\mathrm{cub}}   +
 O(\log p) \nonumber
\end{eqnarray}
whence as $s_{{\theta}, \C} (z) =C_\vartheta\cdot \theta (z+\kappa )$:
\begin{eqnarray}
\label{logC}
\log \vert C_\vartheta\vert = -\height_\Theta (\iota_{\kappa -D} (\infty ))  -\frac{1}{[K_D (\kappa ) :\Q ]} \sum_{ K_D (\kappa ) 
\stackrel{\sigma}{\hookrightarrow} \C} \log \|  \theta ((z +\kappa)_\sigma ) \|_{\mathrm{cub}} +O(\log p).
\end{eqnarray}
Following~\cite{GR11}, paragraph~8, we now write $J_0 (p)(\C ) =\C^g /(\Z^g +\tau \Z^g )$ for $\tau$ in Siegel's fundamental domain, 
write $z \in \C^g$ as $z=\tau \cdot p+q$ for $p,q \in \R^g$, and introduce the function $F\colon \C^g \to \C$ defined as
$$ 
F(z)= \det (2\Im (z))^{1/4} \sum_{n\in \Z^g} \exp ({i\pi {}^t (n+p)\tau (n+p) +2i\pi {}^t nq}) .
$$
One then has $\vert F(z) \vert =2^{g/4} \| \theta (z) \|_{\mathrm{an}}$. %
Indeed there is a constant $A\in \R^*_+$ such that $\vert F(z)\vert =A\cdot \| \theta (z) \|_{\mathrm{an}}$ (see the end of proof of Lemma~8.3 
of \cite{GR11}), $\int_{J_0 (p)(\C )} \vert F \vert^2 d\nu =1$ (where $d\nu$ is the probability Haar measure on $J_0 (p)(\C )$; see~\cite{GR11}, 
Lemma~8.2 (1)), and $\int_{J_0 (p)(\C )}\| \theta (z) \|_{\mathrm{an}}^2 d\nu =2^{-g/2}$ (see e.g~\cite{Mo90}, (3.2.1) and (3.2.2)). Therefore
Lemme~8.3 of \cite{GR11} gives, using definitions of loc. cit., Th\'eor\`eme~8.1,
$$
-\frac{1}{[K_D (\kappa ) :\Q ]} \sum_{ K_D (\kappa ) \stackrel{\sigma}{\hookrightarrow} \C} \left( \log \|  \theta ((z +\kappa)_\sigma ) 
\|_{\mathrm{an}} +\frac{g}{4} \log 2 \right)  \leq   \height_\Theta (\iota_{\kappa -D} (\infty )) +\frac{1}{2}  \height_F (J_0 (p))  + \frac{g}{4} \log 2\pi  .  \nonumber 
$$
Remember Faltings height of $J_0 (p)$ is known to satisfy $\height_F (J_0 (p)) =O(p\log p)$ by \cite{Ul00}, Th\'eor\`eme~1.2.
(We remark that Ullmo's normalization of Faltings' height differs from that of Gaudron-R\'emond, but the difference term is linear in $g=O(p)$
so the bound $O(p\log p)$ remains valid for the above $\height_F (J_0 (p)) $). Writing $\| \cdot \|_{\mathrm{cub}} =e^\varphi \| \cdot \|_{\mathrm{an}}$ 
we therefore see that (\ref{logC}) implies
$$
\log \vert C_\vartheta\vert  +\varphi \leq  \frac{1}{2} \height_F (J_0 (p))  + O(p)  \leq O(p\log p).
$$

   Given this upper bound for $e^\varphi \vert C_\vartheta\vert$ we can now go the other way round to derive an upper bound for 
$\| s_{\theta}\|_{\mathrm{cub}} =C_\vartheta\cdot \| \theta(z+\kappa ) \|_{\mathrm{cub}}$, by using 
estimates for analytic theta functions. For any principally polarized complex abelian variety whose complex invariant $\tau$ is chosen 
within Siegel's fundamental domain $F_g$, Autissier's result in the Appendix (Proposition~\ref{proposition8.1} below) indeed gives, 
with notations as in~(\ref{normtheta}), that:
\begin{eqnarray}
\label{PascalSup}
\frac{1}{\det ({\Im (\tau )})^{1/4}} \| \theta (z) \|_{\mathrm{an}} =\exp (-\pi y \, \Im (\tau )^{-1} y) \vert \theta (z )\vert \leq g^{g/2} .
\end{eqnarray}
We refer to the Appendix for a bound which is even slightly sharper.\footnote{Works of Igusa and Edixhoven-de Jong  
(\cite{EJ11c}, pp.~231-232) give $\frac{1}{\det ({\Im (\tau )})^{1/4}} \| \theta (z) \|_{\mathrm{an}}  \leq 2^{3g^3 +5g} .$
} 

As for the factor $\det ({\Im (\tau )})^{1/4}$, Lemma~11.2.2 of~\cite{EJ11c} gives the general result:  
$$
\det (\Im (z) )^{1/2} \leq \frac{(2g)! V_{2g}}{2^g V_g} \prod_{g+1 \leq i\leq 2g} \lambda_i
$$
where for any $k$ we write $V_k$ for the volume of the unit ball in $\R^k$ endowed with its standard Euclidean structure, and 
the $\lambda_r$ are the successive minima, relative to the Riemann form, of the lattice $\Lambda =\Z^g +\tau \cdot \Z^g$. To bound 
the $\lambda_i$ we need to invoke an avatar of loc. cit., Lemma~11.2.3. But the very same proof shows that for any integer $N$, the 
group $\Gamma_0 (N)$ has a set of generators having entries of absolute value less or equal to the very same bound $N^6 /4$. (That term 
could be improved, but this would have an invisible impact on the final bounds so we here content ourselves with it.) We can therefore rewrite 
the proof of Lemma~11.2.4 verbatim.  
This gives that $\Lambda$ is generated by elements having naive hermitian 
norm $\| x\|_E^2$ less or equal to $g p^{46}$. Finally, in our case the Gram matrix is diagonal (no $2\times 2$-blocks, at the difference of 
Lemma~11.1.4 of loc. cit.) so Lemma~11.2.5 a fortiori holds: if $\| \cdot \|_P$ denotes the hermitian product on $\C^g$ induced by the 
polarization, $\| \cdot \|^2_P \leq \frac{e^{4\pi}}{\pi} \| \cdot \|^2 _E$. This allows to conclude as in p.~228 of~\cite{EJ11c}: 
$$
(\prod_{i=g+1}^{2g} \lambda_i )^2 \leq (\frac{e^{4\pi}}{\pi} g p^{46 } )^g
$$
so that 
$$
\log (\det ({\Im (\tau )})) \le O (p \log p )  
$$
and combining with (\ref{PascalSup}),
$$
\log \| \theta (z)\|_{\mathrm{an}} \le O(p\log p) .
$$
Putting everything together finally yields:
\begin{eqnarray}
\sup_{z\in J_0 (p)(\C )} \log \| s_{{\theta}, \C} (z) \|_{\mathrm{cub}}  &=&\sup_{z\in J_0 (p)(\C )}  \log \| C_\vartheta\cdot \theta (z+\kappa ) \|_{\mathrm{cub}} \nonumber \\
 & = & \left( \log |C_\vartheta | +\varphi \right) +\sup_{z\in J_0 (p)(\C )}  \log \| \theta (z+\kappa ) \|_{\mathrm{an}} \leq O(p\log p). \hspace{0.6cm} \Box \nonumber
\end{eqnarray}
\end{proof}

\medskip

\begin{lem}
\label{lespetites}
Assume the same hypothesis and notations as in Definition~\ref{GoodModels}. 
After possibly making some finite base extension one can pick a set ${\cal S}$ in $H^0 ({\cal N}_{J,4} , {\cal L}(\Theta )^{\otimes 4} )$ 
of $4^{g}$ global sections $(s_{i} )_{1\leq i\leq 4^{g}}$, which span ${\cal L}(\Theta )^{\otimes 4}$ on ${\cal N}_{J,4} [1/2p]$,
and verify

\begin{eqnarray}
\label{petitpetit}
\sup_{{J_0 (p)}} (\log \| s_{i} \|_{\mathrm{cub}} ) \leq  O(p \log p ) .
\end{eqnarray} 
\end{lem}

\begin{proof}
We fix $N=r^2 =4$ for the construction of a good model as in Definition~\ref{GoodModels}. Up to making a base
extension, we can assume ${{L}(\Theta )}^{\otimes 4}$ and
$[2 ]^* {{L}(\Theta )}$ have cubist extensions ${\cal L}(\Theta )^{\otimes 4}$ and $[2]^* {\cal L}
(\Theta )$) on ${\cal N}_{J,4}$, respectively. 
As $\Theta$ is symmetric one knows there is an isomorphism $[2 ]^* {{\cal L}(\Theta )} \to 
{{\cal L}(\Theta )}^{\otimes 4}$ which actually is an isometry (\cite{Pa12}, 
Proposition~5.1), by which we identify those two objects from now on. On the other hand,  
every element $x$ of $J_0 (p)[4 ] (\overline{\Q} )=J_0 (p)[4 ] (K)$ defines a section $\tilde{x}$ in 
${\cal N}_{J,4} ({\mathrm{Spec}} ({\cal O}_K ))$. 
Letting $t_{\tilde{x}}$ denote the translation by $\tilde{x}$ on ${\cal N}_{J,4}$ we have 
\begin{eqnarray}
\label{isometrie}
t_{\tilde{x}}^* {{\cal L}(\Theta )}^{\otimes 4} \simeq {{\cal L}(\Theta )}^{\otimes 4} . 
\end{eqnarray}
(This is indeed true over $\C$ by Lemma 2.4.7.c) of~\cite{BL04}, hence over $K$, then over ${\mathrm{Spec}} 
({\cal O}_K )$ by uniqueness of cubist extensions.)
The interpretation as N\'eron-Tate heights shows that as ${{\cal L}(\Theta )}$ is endowed with its cubist 
metric, 
this isomorphism even is an isometry. Recall the section $s_{{\cal M}}$ defined in~(\ref{laSM}), belonging to $H^0 ({\cal N}_{J,2} ,[2 ]^*{{\cal L}(\Theta )} ) $.
Up to making an extension to some larger base ring of integer, we may assume  $s_{{\cal M}}$ extends as a 
meromorphic section on ${\cal N}_{J,4 }$ and Proposition~\ref{padicmetric0}, which gives estimates on the poles of $s_{{\cal J}^0}$ at bad components, 
implies that  $s_{{\cal M}}$ is actually holomorphic (has no pole on the new components) after multiplication by some power $C_1$ of $p$ with $\log C_1 =O(p\log p )$. 
%
%
We can therefore define a set $(s_i )_{1\le i\le 4^{g}}$ in $H^0 ({\cal N}_{J,4} ,[2 ]^*{{\cal L}(\Theta )} )$ 
made of $4^{g}$ elements of shape 
\begin{eqnarray}
\label{carrement}
s_i := t^*_{\tilde{x}_i} C_1 \cdot s_{\cal M} %
\end{eqnarray}
for $\tilde{x}_{i}$ running through a set of representatives, in $J_0 (p)[4 ](K)$, of $J_0 (p)[4 ]/J_0 (p)[2  ]$. Note that one can explicitly 
lift $s_{\cal M}$ on the complex tangent space at $0$ of $J_0 (p)(\C )$ as 
\begin{eqnarray}
\label{laSMcomplexe}
s_{{\cal M},\C} (z) =C_{\vartheta} \cdot \theta (2 \cdot z)
\end{eqnarray} 
where $C_{\vartheta}$ is defined in the proof of~Lemma~\ref{lacompacte} and the $s_{i,\C}$ are constant multiple of the basis denoted by ${\mathrm h}_{\vec a,\vec b} 
(\vec z)$ in~\cite{Mu84}, Proposition~II.1.3.iii) on p.~124\footnote{where it seems by the way that the expression ``${\mathrm h}_{\vec a,\vec b} (\vec z)=\vartheta 
[{{\vec a}/k \atop {\vec b}/k}] (\ell \cdot \vec z, \Omega )$'' should read ``$\dots =\vartheta [{{\vec a}/k \atop {\vec b}/k}] (k \cdot \vec z ,\Omega )$'' (notations of loc. cit.).}. 
From here, Lemma~\ref{lacompacte} and Proposition~\ref{padicmetric0} give~(\ref{petitpetit}).

By the theory of theta functions (\cite{Pa12}, Proposition~2.5 and its proof, \cite{Mu66} 
and \cite{MB86},  Chapitre VI) the $s_i$ make a generic basis of global sections, which span ${{\cal L}(\Theta 
)}^{\otimes 4}$ over ${\mathrm{Spec}} ({\cal O}_K [1/2p])$.
\hspace{2cm} $\Box$
\end{proof}
\medskip

\begin{lem}
\label{degreeheight}
Let $V$ and $W$ be two closed $K$-subvarieties, with dimension $d_V$ and $d_W$ respectively, of a smooth projective variety $A$ 
over a number field $K$, endowed with an ample sheaf $M$. Assume the flat projective scheme $({\cal A}, {\cal M})$ over ${\mathrm{Spec}} 
({\cal O}_K )$, with $\cal M$ an hermitian sheaf on ${\cal A}$, is a model for $(A,M)$. Let ${\cal V}$ and ${\cal W}$ be the Zariski closure in ${\cal A}$ 
of $V$ and $W$ respectively. Then, with definitions as in~\cite{BGS94}, Section~3.1, 
\begin{eqnarray}
\label{degree}
(c_1 ({ M}^{\boxtimes 2} )^{d_V +d_W} \vert (V\times W)) = \left(  {{d_V +d_W} \atop {d_V}}  \right) ( c_1 
({ M})^{d_V} \vert V) ( c_1 ({M})^{d_W } \vert  W)
\end{eqnarray}
and 
\begin{eqnarray}
\label{height2}
(\hat{c}_1 ({\cal M}^{\boxtimes 2} )^{d_V +d_W +1} \vert {\cal V}\times {\cal W}) 
& = & \left(  {{d_V +d_W +1} \atop {d_V}} \right) (c_1 ({M})^{d_V} \vert {V}) \ 
 (\hat{c}_1 ({\cal M})^{d_W +1} \vert {\cal W}) +   \nonumber \\
& &  \hspace{0.5cm} \left(  {{d_V +d_W +1} \atop {d_W}} \right) (\hat{c}_1 ({\cal M})^{d_V +1} \vert {\cal V})  (c_1 ({ M})^{d_W} \vert {W} ) . 
\end{eqnarray}
\end{lem}

\begin{rema}
{\rm Equation~(\ref{degree}) can be read as
$$
\deg_{{M}^{\boxtimes 2}} (V \times W )  =\left(  {{d_V +d_W} \atop {d_V}}  \right) \deg_{M} (V) \deg_{M}( W)
$$
Equation~(\ref{height2}) in turn fits with Zhang's interpretation (\ref{ZhangNT}) in terms of 
essential minima, compare the proof of Proposition~\ref{Bezout} below.}
\end{rema}

\begin{proof} {\bf (of Lemma~\ref{degreeheight}).}
For (\ref{degree}), one can realize it is elementary, or refer to Lemme 2.2 of \cite{Re10}, or proceed 
as follows. Using (2.3.18), (2.3.19), and Proposition 3.2.1, (iii) of~\cite{BGS94}, and noticing 
$$
{c}_1 ({ M}^{\boxtimes 2})={c}_1 ({M})\times {\bf 1} +{\bf 1}\times {c}_1 ({ M})
$$ 
(and same with $\hat{c}_1 ({\cal M})$ and $\hat{c}_1 ({\cal M}^{\boxtimes 2})$ instead) one computes
\begin{eqnarray}
(c_1 ({  M}^{\boxtimes 2} )^{d_V +d_W} \vert (V\times W))  &= &   ( \sum_{k=0}^{d_V +d_W} \left(  {{d_V +d_W} 
\atop {k}}  \right) c_1 ({  M})^k \times  c_1 ({  M})^{d_V +d_W -k} \vert V\times W)  \nonumber \\
 &= & \sum_{k=0}^{d_V +d_W} \left(  {{d_V +d_W} \atop {k}}  \right) (c_1 ({  M})^k \times  c_1 ({  M})^{d_V 
 +d_W -k} \vert V\times W)  \nonumber \\
  &= & \sum_{k=0}^{d_V +d_W} \left(  {{d_V +d_W} \atop {k}}  \right) (c_1 ({  M})^k \vert V) 
(c_1 ({  M})^{d_V +d_W -k}  \vert  W) \nonumber \\
 &= & \left(  {{d_V +d_W} \atop {d_V}}  \right) ( c_1 ({  M})^{d_V} \vert V) ( c_1 ({  M})^{d_W } \vert  
 W) \nonumber
\end{eqnarray}
where the last equality comes from the fact that the only nonzero term in the line before occurs for $k=d_V$. 

\medskip 
 
 An analogous computation, using~\cite{BGS94}, (2.3.19), can be used for the arithmetic degree: 
\begin{eqnarray}
(\hat{c}_1 ({\cal M}^{\boxtimes 2} )^{d_V +d_W +1} \vert {\cal V}\times {\cal W}) 
& = &  \sum_{k=0}^{d_V +d_W +1} \left(  {{d_V +d_W +1} \atop {k}}  \right)(\hat{c}_1 ({\cal M})^k \times  
\hat{c}_1 ({\cal M})^{d_V +d_W +1 -k} \vert  {\cal V}\times {\cal W}) \nonumber \\
 & = & \left(  {{d_V +d_W +1} \atop {d_V}} \right) (c_1 ({M})^{d_V} \vert {V}) \ 
 (\hat{c}_1 ({\cal M})^{d_W +1} \vert {\cal W}) +   \nonumber \\
& &  \hspace{1.7cm} \left(  {{d_V +d_W +1} \atop {d_W}} \right) (\hat{c}_1 ({\cal M})^{d_V +1} \vert {\cal V})  (c_1 ({ M})^{d_W} \vert {W} ) . \hspace{0.3cm} \Box \nonumber  
\end{eqnarray}
\end{proof}

\bigskip

For the rest of this Section we fix the model $({\cal J} ,{\cal M})$ for $(J_0 (p), \Theta )$ 
(see~(\ref{Bezout0})) as the one built with the set $\cal S$ of $N^{g} =4^g$ sections provided by 
Lemma~\ref{lespetites}. Before settling the proof of the arithmetic B\'ezout theorem, we need a last 
lemma on comparison between the projective height  on $({\cal J} ,{\cal M})$ and its normalized 
N\'eron-Tate avatar.

\begin{lem}
\label{n'estpasraison}
Up to translation by torsion points, the projective height $\height_{\cal M}$ on points in
$J_0 (p)(\overline{\Q})$ (associated with the good model $({\cal J}, {\cal M})$) differs from the 
N\'eron-Tate theta-height $4{\height}_\Theta$ by an error term of shape $O (p\log p)$. 
\end{lem}

\begin{proof}

Lemma~\ref{lespetites} implies 
that the elements of ${\cal S}$ extend as holomorphic sections 
to {\it any} component of the N\'eron model ${\overline{\cal N}}$ of $J_0 (p)$ over $\overline{\Z}$ (see (\ref{carrement})).
As remarked in the proof of Lemma~\ref{lespetites}, Mumford's algebraic theory of theta-functions implies 
that the sections in ${\cal S}$ do define a projective embedding of ${\overline{\cal N}}$  
over $\overline{\Z} [1/2p]$: the only fibers of ${\overline{\cal N}}$ over $\overline{\Z}$ where base points for ${\cal S}$ can 
show up are above $2$ and $p$. If one seeks to approximate 
the N\'eron-Tate height of a given point $P$ in $J_0 (p)(\overline{\Q} )$ by the projective height of our 
good model $({\cal J} ,{\cal M})$, one needs the section of the N\'eron model ${\overline{\cal N}}$ defined 
by $P$ to avoid those base points, or at least control their length.

Given $P$ in $J_0 (p)(\overline{\Q} )$, we claim one can translate $P$ by some torsion point in $J_0 (p)
(\overline{\Q} )$ so that the translated new point $P+t$ does avoid base points in characteristic $2$. 
Indeed, choose a Galois extension $F/\Q$ such that the base locus is defined over ${\mathrm{Spec}} ({\cal O}_F \otimes \F_2 )$. Summing-up, as divisors, 
all the Galois conjugates of that base locus in each fiber of characteristic $2$, one obtains a constant cycle $C_{\kappa}$, in each fiber at $\kappa$, which is 
defined over $\F_2$. (In our case one actually could have taken $F=\Q$.) Density of torsion points then shows that  one can replace our point $P$ by $P+t$, 
for some torsion point $t$, such that $P+t$ does not belong to $C_{\kappa_0}$ for some $\kappa_0$, then for all $\kappa$ of characteristic $2$ because $C_\kappa$ 
is constant. This proves our claim. Now in characteristic $p$, we know from Proposition~\ref{padicmetric0} again that possible base points have length at most $O(p)$, 
which gives an estimate of size $O(p\log p)$ for  the difference error term between projective height on ${\cal J}$ and N\'eron-Tate height (\cite{Pa12}, 
Proposition~4.1). \hspace{13cm} $\Box$

\end{proof}

\medskip

\begin{proof} {\bf of Proposition~\ref{Bezout}}.
Before proceeding we will allow ourselves, for this proof only, and in the hope not to weighten too much 
the computations, to work with heights defined as in~\cite{BGS94}, Section~3.1. Namely, for ${\cal Y}$ a
cycle of dimension $(d+1)$ in a regular arithmetic variety endowed with a hermitian sheaf ${\cal F}$, 
we multiply our definition~(\ref{lavraiehauteur}) of its height by degree and absolute 
dimension and we set:
$$
{\height'}_{\cal F} ({\cal Y})  =\frac{(\hat{c}_1 ({\cal F} )^{d+1} \vert {\cal Y} )}{[K:\Q ]}.
$$
Note that ${\height}$ and ${\height'}$ coincide on $K$-rational points, in which case we might use either notation.

Construction~(\ref{Bezout0}) gives a $\Q$-embedding $V\times W \stackrel{\iota}{\hookrightarrow} 
\PPP^{n^2 +2n}$ via a Segre map. We set 
$$
s_{\underline{i},\underline{j}} :=\iota^* (z_{\underline{i},\underline{j}} -z_{\underline{j},\underline{i}} )
$$
for all $({\underline{i},\underline{j}} )$, and denote by ${\cal O}_{N}$ the ambiant line bundle $\iota^* ({\cal O}_{\PPP^{n^2 +2n}} (1)) ={{\cal M}^{\boxtimes 2}}$ 
as before~(\ref{mDeux}). (Recall we will eventually specialize to $N=4$.) 
Set also $O_{N} :={\cal O}_{N} \otimes \Q$. We intersect $\iota (V\times W)$ with one of the $\mathrm{div}  (z_{\underline{i_0},\underline{j_0}} 
-z_{\underline{j_0},\underline{i_0}} )_\Q $ such that the two cycles meet properly: define 
$$
J_1 =\mathrm{div} ({s_{\underline{i_0},\underline{j_0}}}_\Q ) \cap (V \times W )
$$
in the generic fiber $(J_0 (p)\times J_0 (p))_\Q$. As $\mathrm{div}  (z_{\underline{i_0},\underline{j_0}} -z_{\underline{j_0},\underline{i_0}} )$ is a projective hyperplane we have 
by definition 
$$
\deg_{{O}_{N}} (J_1 ) =\deg_{{O}_{N}} (V \times W ) .
$$
For the same linearity reason, a similar statement is true for heights. 
Indeed, let ${\cal V}$ and ${\cal W}$ denote the schematic closure  in ${\cal J}$ 
of $V$ and $W$ respectively, and ${\cal J}_1$ the schematic closure of $J_1$ in ${\cal J}\times {\cal J}$, which satisfies  
$$
{\height'}_{{\cal O}_{N}} ({\cal J}_1 )   \leq {\height'}_{{\cal O}_{N}} (\mathrm{div} ({s_{\underline{i_0},
\underline{j_0}}} ) \cap ({\cal V}\times {\cal W})) 
$$
(as there might be vertical components in the intersection of the right-hand side which do not 
intervene in the left, and contribute positively to the height).

   Proposition~3.2.1~(iv)  of \cite{BGS94} gives, with notations of loc. cit., that:
\begin{eqnarray}
\label{inductionBezout}
{\height'}_{{\cal O}_{N}} (\mathrm{div} ({s_{\underline{i_0},\underline{j_0}}} ) \cap ({\cal V}\times {\cal W})) & = & {\height'}_{{\cal O}_{N}} ( {\cal V}\times {\cal W} ) \nonumber \\
 & & +  \frac{1}{[K:\Q ]} \sum_{\sigma \colon K\hookrightarrow \C } \int_{(V \times W)_\sigma (\C )} \log \| {s_{\underline{i_0},\underline{j_0}}}_\C  \|  
  c_1 ({{\cal O}_{N}})^{d_V +d_W}  \hspace{0.7cm}
\end{eqnarray}
where  $\| \cdot \| =\| \cdot \|_{\mathrm{cub}}$ shall denote  the cubist metric, or the metric induced 
by the cubist metric on products or powers of relevant sheaves.  To estimate the last integral we note that at any point of $( V\times W)_\sigma  
(\C )$ and for any $(\underline{i},\underline{j})$, 
\begin{eqnarray}
\| s_{\underline{i},\underline{j}}  \| & = & \| z_{\underline{i},\underline{j}}  -z_{\underline{j},
\underline{i}}    \|_{{\cal M }^{\boxtimes 2}}  \leq  \| 
z_{\underline{i},\underline{j}} \|_{{\cal M}^{\boxtimes 2}} +\| z_{\underline{j},\underline{i}} 
\|_{{\cal M}^{\boxtimes 2}} \nonumber  \\
    & \le & \| x_{\underline{i}} \|_{\cal M} \| y_{\underline{j}}  \|_{\cal M}  + \| x_{\underline{j}} 
    \|_{\cal M}  
    \| y_{\underline{i}} \|_{\cal M}  \leq 2 (\sup_{\underline{i}} \| x_{\underline{i}} \|_{\cal M} )^2  
    \nonumber \\
     & \leq & \exp ({2  \log (\sup {\| s_{i} \|_{\mathrm{cub}}}  ) +\log 2}) \nonumber 
\end{eqnarray}
with notations of Lemma~\ref{lespetites}. Setting $M_{{\cal J} ,{\cal M}} = \log (\sup {\| s_{i} \|_{\mathrm{cub}}} )$ we obtain
$$
{{\height'}}_{{\cal O}_{N}} ({\cal J}_1 ) \leq {{\height'}}_{{\cal O}_{N_0 }} ( {\cal V}\times {\cal W} ) 
+(2 M_{{\cal J} ,{\cal M}} +\log 2) \deg_{{\cal O}_{N}} (V\times W ) . \nonumber 
$$
Call $I_1$ one of the reduced irreducible components of $J_1$ containing the point $\iota (\Delta (P))$ 
of $V\cap W$ considered in the statement of~Proposition~\ref{Bezout}, and 
let ${\cal I}_1$ denote its Zariski closure in ${\cal J}$. It has  ${{\cal O}_{N}}$-height (and degree) 
less than or equal to those of ${\cal J}_1$, so that again 
\begin{eqnarray}
{\height'}_{{\cal O}_{N}} ({\cal I}_1 ) \leq {\height'}_{{\cal O}_{N}} ({\cal V}\times {\cal W})  +(2  
M_{{\cal J} ,{\cal M}} +\log 2) \deg_{{\cal O}_{N}} ({V}\times {W} )  \nonumber 
\end{eqnarray}
and we can iterate the process with $I_1$ in place of $V\times W$: we obtain some $J_2$, ${\cal J}_2$, $I_2$, 
${\cal I}_2$ such that
\begin{eqnarray}
{\height'}_{{\cal O}_{N}} ({\cal I}_2 ) & \leq & {\height'}_{{\cal O}_{N}} ({\cal I}_1 )  +(2  M_{{\cal J} ,
{\cal M}} +\log 2) 
\deg_{{\cal O}_{N}} (I_1 )   \nonumber \\
 &\leq & {\height'}_{{\cal O}_{N}} ({\cal V}\times {\cal W})  +2(2 M_{{\cal J} ,{\cal M}} +\log 2) 
\deg_{{\cal O}_{N}} ({V}\times {W} ) . \nonumber  
\end{eqnarray}
(The only obstruction to this step is if all the $s_{\underline{k} ,\underline{l}}$ vanish on $I_1$, which 
implies it is contained in the diagonal of ${J_0 (p)}\times {J_0 (p)}$ - so that $I_1 =\iota (\Delta (P))$ 
by construction and that means we are already done.) Processing, one builds a sequence $({\cal I}_k )$ of 
integral closed subschemes of ${\cal J}\times {\cal J}$, with decreasing dimension, such that the last step 
gives
\begin{eqnarray}
{\height'}_{{\cal O}_{N}} ({\cal I}_{d_V +d_W })  & \leq  &  {\height'}_{{\cal O}_{N}} ({\cal V}\times 
{\cal W} ) 
 +(d_V +d_W )(2 M_{{\cal J} ,{\cal M}} +\log 2) \deg_{{O}_{N}} (V\times W ) . \nonumber 
\end{eqnarray}
Now 
$$
{\height'}_{{\cal O}_{N}} ({\cal I}_{d_V +d_W }) \geq {\height'}_{{\cal O}_{N}} (\Delta (P ,P) ) =  
{\height'}_{{\cal M}^{\otimes 2}} (P ) 
={\height}_{{\cal L}^{\otimes 2N}} (P ) =2N \, {\height}_{\Theta} (P ) +O(p\log p) ,
$$
for ${\height}_{\Theta} (P )$ the N\'eron-Tate normalized theta height.
Indeed the statement of the present Proposition~\ref{Bezout} is invariant by translation of every 
object by some fixed torsion point, so that one can apply Lemma~\ref{n'estpasraison}.

   Using Lemma~\ref{degreeheight} and Corollary~\ref{LesHauteursInegales} and writing 
${\height'}_{\Theta} ({Y}) = (\dim (Y)+1)\deg_\Theta (Y) \height_{\Theta} ( Y)$ we therefore obtain  
\begin{eqnarray}
2N {\height}_\Theta (P ) & \leq &  N^{d_v +d_W +1}  \left[  (d_W +1) \left( {{d_V +d_W +1} \atop {d_V}}  \right) {\height'}_{\Theta}  (W) \deg_{\Theta} (V)  +\right. \nonumber \\  
 & & \hspace{5cm} + \left. (d_V +1) \left(  {{d_V +d_W +1} \atop {d_W}}  \right) {\height'}_{\Theta}  (V) \deg_{\Theta} (W) \right]   \nonumber \\
&  &   \hspace{0.6cm} +N^{d_V +d_W} (d_V +d_W )(2  M_{{\cal J} ,{\cal M}} +\log 2 )  \left(  {{d_V +d_W} \atop {d_V}}  \right)  \deg_{\Theta}  (V)  \deg_{\Theta}  (W )  \nonumber \\
 & &  \hspace{3cm} + O(p\log p) . \nonumber
\end{eqnarray} 
From here, fixing $N=4$, the bound $M_{{\cal J} ,{\cal M}} \le O(p\log p)$ (Lemma~\ref{lespetites}) 
concludes the proof, after expressing quantities ${\height'}_\Theta$ back into ${\height}_\Theta$. 
\hspace{8cm} $\Box$ 

\end{proof}

\bigskip

That arithmetic B\'ezout theorem will be  our principal tool in the sequel.  

\section{Height bounds for quadratic points on $X_0 (p )$}
\label{7}

\begin{propo}
\label{orderdegree}
Let $\iota \colon X\hookrightarrow J$ be some Albanese map from a curve (of positive genus) over some field $K$ to its jacobian $J$.
Let $\pi \colon J\to A$ be some quotient of $J$, with $\dim (A)>1$, and $X'$ be the normalization of the image $\pi \circ \iota (X)$ of $X$ in $A$. 
Then the map $\pi' \colon X\to X'$ induced by $\pi \circ \iota$ verifies
$$
\deg (\pi' ) \leq \frac{\dim (J)-1}{\dim (A ) -1} .
$$
\end{propo}
\begin{proof}
The map $\pi \circ \iota$ induces an inclusion of function fields which defines the map $\pi' \colon X\to X'$. If $J'$ is the jacobian  
of $X'$, Albanese functoriality says that $\pi$ factorizes through surjective morphisms $J\to J' \to A$. Hurwitz formula writes:
$$
\deg (\pi' ) =\frac{\dim (J)-1 -\frac{1}{2} \deg R}{\dim (J' ) -1} 
$$
for $R$ the ramification divisor of $\pi'$, whence the result. \hspace{6cm} $\Box$
\end{proof}

\bigskip

\begin{lem}
\label{afterBrumer}
For all large enough prime $p$, let $X:=X_0 (p)$ and $\pi_e  \colon J_0 (p) \twoheadrightarrow J_e$ be the projection. Let 
$$
\iota_{P_0}  \colon 
\left\{
\begin{array}{rcl}
X_0 (p)& \hookrightarrow & J_0 (p) \\
P & \mapsto & {\mathrm{cl}} (P-P_0 )
\end{array}
\right.
$$
for some $P_0$ in $X_0 (p)({\overline{\Q}}) $ such that $w_p (P_0 )=P_0$  (there are roughly $\sqrt{p}$ such points, see Proposition~3.1 of~\cite{Gr87b}),
and set $\varphi_e :=\pi_e \circ \iota_{P_0}$. Then: 
\begin{itemize}
\item if $a \in J_e (\Q )$ is some (necessarily torsion) point, the equality $\varphi_e  (X_0 (p)) =a-\varphi_e (X_0 (p))$ implies
\begin{eqnarray}
\label{condition}
\varphi_e  (X_0 (p)) =a+\varphi_e  ( X_0 (p))
\end{eqnarray}
and $a=0$; 

\item  If $d$ is the degree of the map ${X_0 (p) \to \widetilde{\varphi_e (X_0 (p))}}$ to the normalization 
of $\varphi_e (X_0 (p))$, then $d$ is either $1$, $3$ or $4$; 

\item Assuming moreover Brumer's conjecture (see~(\ref{Brumer}) and (\ref{TuM'asDitDeLeFaireMoinsFort}))
equality~(\ref{condition})  implies~$d=1$ for large enough~$p$.
\end{itemize}
\end{lem}
\begin{proof}
Notice first that, by our choice of $P_0$ (whence $\iota$), and because $J_e$ belongs to the $w_p $-minus part of $J_0 (p)$, one has:
$$
\varphi_e  (w_p (P)) =w_p (\varphi_e (P)) =- \varphi_e  (P)
$$
for all $P\in X_0 (p)(\C)$, whence equality~(\ref{condition}). So let $n$ be the order of $a$, which also is that of the automorphism ``translation by $a$ restricted 
to $\varphi_e (X_0 (p))$'' . We remark that the degree $d$ cannot be equal to $2$, as otherwise the extension of fraction fields $K(X_0 (p))/K(\varphi_e (X_0 (p)))$ 
would be Galois and $X_0(p)$ would possess an involution different from $w_p $, which it does not  by Ogg's 
theorem (\cite{Og77}, or~\cite{KM88}). If $d=1$, the same reason that ${\mathrm{Aut}} (X_0 (p)) =\langle w_p \rangle$ implies that $n=1$. 
Let now $X'$ be the normalization of the quotient of $\varphi_e (X_0 (p))$ by the automorphism $P\mapsto P+a$, that is, the image 
of $\varphi_e (X_0 (p))$ by the quotient morphism $J_e \twoheadrightarrow J_e /\langle a\rangle$. Let $\pi$ be the composed map $J_0 (p) 
\stackrel{\varphi_e}{\longrightarrow} J_e \to J_e /\langle a\rangle$. The degree of $X_0 (p)\to X'$ is $d\cdot n$ and Proposition~\ref{orderdegree} 
together with the left part of inequalities~(\ref{KM}) implies:
$$
d\cdot n \leq \frac{g -1}{(\frac{1}{4} -o(1) ) g -1} \leq 4+o(1) 
$$ 
for large enough $p$. This shows that if $d=3$ or $4$ one still has $a=0$, whence the Proposition's first two statements. 
Assuming (\ref{TuM'asDitDeLeFaireMoinsFort}) we have $d\cdot n<3$, so that $d=1$ and $a=0$ by previous arguments. \hspace{13cm}  $\Box$
\end{proof}

\medskip

\begin{rema}
{\rm Replace, in Lemma~\ref{afterBrumer}, the map $X_0 (p) \to J_e$ by $X_0 (p) \stackrel{\varphi}{\longrightarrow}  J_0 (p)^-$ (by which the 
former factorizes, by the way). The above proof 
shows that the map $X_0 (p) \to \varphi (X_0 (p))$ is of generic degree $1$ (independently on any conjecture), but of course it needs not be
injective on points:
a finite number of points can be mapped together to singular points on $\varphi (X_0 (p))$. In our case one checks those are 
among the Heegner points $P$ such 
that $P=w_p (P)$ (for which we again refer to Proposition~3.1 of~\cite{Gr87b}). Indeed, 
the endomorphism of $J_0 (p)$ defined by multiplication by $(1-w_p )$ factorizes through $\varphi$, and $\cdot (1-w_p )$ is the map 
considered in~(\ref{chain}) and what follows, inducing multiplication by $2$ on tangent spaces. Therefore, if $P$ maps to a multiple point 
of $\varphi  (X_0 (p))$, it also maps to a multiple point of $(1-w_p )\circ \iota (X_0 (p))$. Now assuming $X_0 (p)$ has gonality larger 
than $2$ (which is true as soon as $p> 71$, \cite{Og74}, Theorem~2), the 
equality ${\mathrm{cl}} ( (1-w_p )P)={\mathrm{cl}} ((1-w_p )P' )$ in $J_0 (p)$, for some $P'$ on $X_0 (p)$ different from $P$, implies $P=w_p P$
and $P' =w_p P'$. That is, $P$ and $P'$ are  Heegner points. 
}
\end{rema}

\medskip

\begin{lem}
\label{notX+}
Suppose $P$ belongs to $X_0 (p^2 ) (K)$ for some quadratic number field $K$, and $P$ is not a complex multiplication point. Then for one of 
the two natural degeneracy morphisms $\pi$ from $X_0 (p^2 )$ to $X_0 (p)$, the point $Q:=\pi (P)$ in $X_0 (p) (K)$ does not define a $\Q$-valued 
point of the quotient curve $X_0^+ (p) :=X_0 (p)/w_p$. 
\end{lem}
\begin{proof}
Using the modular interpretation, we write $P=(E,C_{p^2} )$ for $E$ an elliptic curve over $K$ and $C_{p^2}$ a cyclic $K$-isogeny
of degree $p^2$, from which we obtain the two points $Q_1 :=(E,p\cdot C_{p^2} )$ and $Q_2 :=(E/p\cdot C_{p^2} ,C_{p^2 } \mod p\cdot C_{p^2})$ 
in~$X_0 (p)(K)$. Assume both $Q_1$ and $Q_2$ do define elements of~$X_0^+ (p)(\Q )$. If~$\sigma$ denotes a generator 
of ${\mathrm{Gal}} (K/\Q )$ we then have
$$
w_p (Q_1 )=(E/p\cdot C_{p^2} ,E[p] \!\!\!\mod p\cdot C_{p^2} )\simeq \sigma (Q_1 )
$$ 
and 
$$
w_p (Q_2 )=(E/ C_{p^2} ,E[p] +C_{p^2} \!\!\!\mod  C_{p^2} )\simeq \sigma (Q_2 ).
$$
Therefore $E \simeq {}^\sigma (E/p\cdot C_{p^2} ) \simeq E/C_{p^2}$, which means $E$ has complex multiplication. $\Box$
\end{proof}

\bigskip

We can now conclude with the main result of this paper.

\bigskip
\begin{theo}
\label{theoheight}
There is an integer $C$ such that the following holds. If $p$ is a prime number such that (\ref{TuM'asDitDeLeFaireMoinsFort}), the weak form of Brumer's conjecture, holds, 
and $P$ is a quadratic point of $X_0 (p)$ (that is: $P$ is an element of $X_0 (p )(K)$ for some quadratic number field $K$) which does not come from $X_0 (p)^+ (\Q )$, 
then its $j$-height satisfies
\begin{eqnarray}
\label{finalheightA}
\height_j (P) <C\cdot p^{5} \log p  .
\end{eqnarray}  
If $P$ is a quadratic point of $X_0 (p^2 )$ then the same conclusion holds without further assumption apart from (\ref{TuM'asDitDeLeFaireMoinsFort}). 
\end{theo}
\begin{proof}
In the case $P$ is a quadratic point of $X_0 (p^2 )$, by Lemma~\ref{notX+} one can deduce from $P$ a point $P'$ in $X_0 (p) (K)$ which does not induce 
an element of~$X_0^+ (p)(\Q )$, and whose $j$-height, say, is 
equal to $\height_j (P )+O(\log p)$ for an explicit function $O(\log p)$ (see e.g. \cite{Pe01}, 
inequality (51) on p.~240 and \cite{BPR13}, Proposition 4.4 (i)). 
Replace $P$ by $P'$ if necessary. By Theorem~\ref{TheoremTheta} it is now sufficient to prove that $\height_\Theta (P-\infty )=O(p^5 \log p)$. 

   Keep the notation of Lemma~\ref{afterBrumer}. By construction, the point:
$$
a:=\varphi_e (P) +\varphi_e ({}^\sigma P) =\varphi_e (P) -\varphi_e (w_p ({}^\sigma P)) =\varphi_e (P-w_p  ({}^\sigma P))
$$ 
is torsion. First assume $a=0$. Set $X^{(2), -} :=\left\{ \iota_\infty  (x) -\iota_\infty  (y), (x,y)\in X_0 (p)^2 \right\}$ as in Proposition~\ref{heightXSquareThetae}. 
Recall from Section~\ref{preliminaries} that $\tilde{I}_{J_e^\perp ,N_e^\perp} \colon  J_e^\perp \to \tilde{J}_e^\perp$ is the map defined as in~(\ref{OnTangentSpaces}),
that $\iota_{\tilde{J}_e^\perp ,N_e^\perp}$ is the embedding $\tilde{J}_e^\perp  \hookrightarrow J_0 (p)$, and denote by $[N_{\tilde{J}_e^\perp}]_{\tilde{J}_e^\perp}$ the multiplication 
by $N_{\tilde{J}_e^\perp}$ restricted to ${\tilde{J}_e^\perp}$. As in~(\ref{pseudoprojection})
and before Corollary~\ref{DonnezNousDonnezNousDesHauteurs} we use our pseudo-projections and define 
$$
{\widetilde{X}}^{(2), -} := \iota_{\tilde{J}_e^\perp ,N_e^\perp} [N_{\tilde{J}_e^\perp}]_{\tilde{J}_e^\perp} ^{-1} \tilde{I}_{J_e^\perp ,N_e^\perp} \pi_{J_e^\perp} (X^{(2), -} )  .
$$ 
Then $P-w_p   ({}^\sigma P)$ belongs to $X^{(2), -} \cap \tilde{J}_e^\perp$, and even to the intersection of surfaces (in the generic fiber):
$$
X^{(2), -} \cap {\widetilde{X}}^{(2), -}   .
$$
Recall (see~(\ref{pseudoprojection})) that ${\widetilde{X}}^{(2), -}$ is a priori highly non-connected, 
being the inverse image of multiplication by $N_{\tilde{J}_e^\perp}$ in $\tilde{J}_e^\perp$ of the 
(irreducible) surface $\tilde{I}_{J_e^\perp ,N_e^\perp} \pi_{J_e^\perp} (X^{(2), -} )$.  However, in 
what follows we can replace ${\widetilde{X}}^{(2), -}$ by one of its connected components 
containing $P-w_p   ({}^\sigma P)$. Denote that component by ${\widetilde{X}}^{(2), -}_P$.

By construction, the theta degree and height of ${\widetilde{X}}^{(2), -}_P$, as an 
irreducible subvariety of $J_0 (p)$ endowed with $\Theta$, are those of $\pi_{J_e^\perp} (X^{(2), -} ) =X^{(2),-}_{e^\perp}$ relative to the only natural hermitian 
sheaf of ${J}_e^\perp$, that is, the $\Theta_e^\perp =\Theta_{J_e^\perp}$ described in paragraph~\ref{PolarizationsAndHeights}. One can therefore apply 
Proposition~\ref{heightXSquareThetae} to obtain that all theta degrees are $O(p^2 )$, all 
N\'eron-Tate theta heights are $O(\log p)$. We claim the dimension of $(X^{(2), -} \cap 
{\widetilde{X}}^{(2), -}_P )$ is zero. That intersection indeed corresponds to pairs of distinct points on $X_0 (p)$ having same image ($0$) under $\varphi_e$.
On the other hand, Brumer's conjecture implies $X_0 (p)\to \varphi_e (X_0 (p))$ has generic degree one (see Lemma~\ref{afterBrumer}), so our intersection points 
correspond to singular points on $\varphi_e (X_0 (p))$, which of course make a finite set.

We therefore are in position to apply our arithmetic B\'ezout theorem (Proposition~\ref{Bezout}), 
which yields $\height_{\Theta} (P-w_p  ({}^\sigma P) ) \leq O(p^5  \log p)$. The two points $(P-\infty )$ and $(w_p ({}^\sigma P) -\infty )$ have 
same $\Theta$-height (recall $w_p$ is an isometry on $J_0 (p)$ for $\height_\Theta$, compare the end of Remark~\ref{symmetricHodge}), and are 
by hypothesis different, so one can apply them Mumford's repulsion principle (Proposition~\ref{Mumford}) to obtain  
\begin{eqnarray}
\label{Lafinale}
\height_{\Theta} (P-\infty  )\leq   O(p^5 \log p).
\end{eqnarray}

       Let us finally deal with the case when the torsion point $a=\varphi_e (P) +\varphi_e ({}^\sigma P)$ 
is nonzero. We adapt the previous argument: pick a lift $\tilde{a} \in J_0 (p)(\overline{\Q})$ of $a$ 
by $\pi_e^\perp$ which also is torsion, and let $t_{\tilde{a}}$ be the translation by $\tilde{a}$ 
in $J_0 (p)$. Replace $(P-w_p ({}^\sigma P))$ by $t_{\tilde{a}}^* (P-w_p ({}^\sigma P))$,  $X^{(2), -}$ 
by $t_{\tilde{a}}^* X^{(2), -}$ and ${\widetilde{X}}^{(2), -}$ by
$$
{\widetilde{t_{\tilde{a}}^* {X}}^{(2), -}} = \iota_{\tilde{J}_e^\perp ,N_e^\perp} [N_{\tilde{J}_e^\perp}]_{\tilde{J}_e^\perp} ^{-1} 
\tilde{I}_{J_e^\perp ,N_e^\perp} \pi_{J_e^\perp}  
(t_{\tilde{a}}^* X^{(2), -}).
$$  
Now $t_{\tilde{a}}^*  (P-w_p ({}^\sigma P))$ belongs to  $( t_{\tilde{a}}^* X^{(2), -} \cap {\widetilde{t_{\tilde{a}}^* {X}}^{(2), -}} )$.    
The theta degree and height of $t_{\tilde{a}}^* X^{(2), -}$ and ${\widetilde{t_{\tilde{a}}^* {X}}^{(2), -}}$ (or rather, as above, some connected component 
${\widetilde{t_{\tilde{a}}^* {X}}^{(2), -}}_P $ of it containing $t_{\tilde{a}}^* (P-w_p  ({}^\sigma P))$) are
the same as for the former objects in the case $a=0$. The fact that the intersection 
$$
 t_{\tilde{a}}^* X^{(2), -} \cap {\widetilde{t_{\tilde{a}}^* {X}}^{(2), -}}_P 
$$ 
is zero-dimensional comes from the fact that otherwise, we would have $\varphi_e (X_0 (p)) =a-\varphi_e (X_0 (p))$, a contradiction with our present 
hypothesis $a\neq 0$ by Proposition~\ref{afterBrumer}. The height bound for $P$ is therefore the same as (\ref{Lafinale}). 
$\Box$

\end{proof}

\medskip

\begin{coro}
\label{corolaBrum}
Under the assumptions of Theorem~\ref{theoheight}, if $p$ is a large enough prime number and $P$ is a quadratic point of $X_0 (p^\gamma )$ 
for some integer $\gamma$, such that $P$ is not a cusp nor a complex multiplication point, then $\gamma \leq 10$.  
\end{coro}
\begin{proof}
Let $P$ be a point in $X_0 (p^\gamma )(K)$, which is not a cusp nor a CM point, for some quadratic number field $K$. 
Then the isogeny bounds of~\cite{GR11}, Theorem~1.4 imply there is some real $\kappa$ with 
$$
p^\gamma < \kappa (\height_j (P))^2 .
$$  
Now Theorem~\ref{theoheight} gives that there is some absolute real constant $B$ such that, if $p\geq B$ then $\gamma \leq 10$. $\Box$ 
\end{proof}

\begin{rema}
{\rm A similar (but technically simpler) approach for the morphism $X_0 (p)\to J_e$ over $\Q$ should give 
(independently of any conjecture) a bound of shape
$O(p^3 \log p )$ for the $j$-height of $\Q$-rational (non-cuspidal) points of $X_0 (p)$ (which are known not 
to exist for $p>163$ by Mazur's theorem). The same
should apply for $\Q$-points of $X_{\mathrm{split}} (p)$ (and here again, we obtain a weak version of known 
results).

Actually, sharpening results directly coming from Section~\ref{jETtheta} (that is, avoiding the use of 
B\'ezout) might even yield the full strength of the above results about $X_0 (p)(\Q )$ and 
$X_{\mathrm{split}} (p) (\Q )$, with more straightforward (unconditional) proofs.}
\end{rema}

\bigskip 
 
\section{Appendix: An upper bound for the theta function, by P. Autissier}
\label{Section8}

In this appendix, I give a new upper bound for the norm of the classical theta
function on any complex abelian variety. This result, apart from its role in the present paper (see Section~\ref{ArithmeticBezout}),
has been used by Wilms \cite{Wi16} to bound the Green-Arakelov function on curves. 

\medskip

\subsection{Result}

Let $g$ be a positive integer. Write ${\mathbb{H}_g}$ for the Siegel space of
symmetric matrices $Z\in{\rm M}_g(\mathbb{C})$ such that ${\rm Im}Z$ is
positive definite. To every $Z\in\mathbb{H}_g$ is associated the theta function
defined by
$$\theta_Z(z)=\sum_{m\in\mathbb{Z}^g}\exp(i\pi ^{\rm t}\!mZm+2i\pi ^{\rm t}\!mz) ,\quad \forall z\in\mathbb{C}^g ,$$
and its norm defined by
$$\|\theta_Z(z)\|=\sqrt[4]{\det Y}\exp(-\pi ^{\rm t}\!yY^{-1}y)|\theta_Z(z)| , \quad \forall
z=x+iy\in\mathbb{C}^g ,$$
where $Y={\rm Im} Z$.

\medskip

My contribution here is the following:

\begin{propo}
\label{proposition8.1} 
Let $Z\in\mathbb{H}_g$ and assume that $Z$ is
Siegel-reduced. Put $\displaystyle c_g=\frac{g+2}{2}$ if $g\le3$ and
$\displaystyle c_g=\frac{g+2}{2}\Bigl(\frac{g+2}{\pi\sqrt{3}}\Bigr)^{g/2}$ if
$g\ge4$. The upper bound $\|\theta_Z(z)\|\le c_g(\det{\rm Im}Z)^{1/4}$ holds for
every $z\in\mathbb{C}^g$.
\end{propo}

Let us remark that $c_g\le g^{g/2}$ for every $g\ge2$. In comparison, Edixhoven
and de Jong (\cite{EJ11c} page 231) obtained the statement of 
Proposition~\ref{proposition8.1} with $c_g$ replaced by $\displaystyle2^{3g^3+5g}$.

\subsection{Proof}

Fix a positive integer $g$. Denote by $\mathbb{S}_g$ the set of symmetric
matrices $Y\in{\rm M}_g(\mathbb{R})$ that are positive definite. Let us recall
a special case of the functional equation for the theta function (see equation
(5.6) of \cite{Mum} page 195): for every $Y\in\mathbb{S}_g$ and every
$z\in\mathbb{C}^g$, one has
\begin{eqnarray}
\label{(1)}
\theta_{iY^{-1}}(-iY^{-1}z)=\sqrt{\det Y}\exp(\pi ^{\rm t}\!zY^{-1}z)\theta_{iY}(z) .
\end{eqnarray}

\begin{lemma}
\label{lemma8.2} Let $Z\in\mathbb{H}_g$ and $z\in\mathbb{C}^g$. Putting
$Y={\rm Im}Z$, one has the inequality
$$\|\theta_Z(z)\|\le\|\theta_{iY}(0)\|=\theta_{iY}(0)\sqrt[4]{\det Y}.$$
\end{lemma}

\begin{proof}
Put $y={\rm Im}z$. One has
$$|\theta_Z(z)|=\Bigl|\sum_{m\in\mathbb{Z}^g}\exp(i\pi ^{\rm t}\!mZm+2i\pi ^{\rm t}\!mz)\Bigr|\le\sum_{m\in\mathbb{Z}^g}\Bigl|\exp(i\pi ^{\rm t}\!mZm+2i\pi ^{\rm t}\!mz)\Bigr|=\theta_{iY}(iy)\quad,$$
that is, $\|\theta_Z(z)\|\le\|\theta_{iY}(iy)\|$. The functional equation~(\ref{(1)}) gives
$\|\theta_{iY^{-1}}(Y^{-1}y)\|=\|\theta_{iY}(iy)\|$, and one deduces
\begin{eqnarray}
\label{(2)}
\|\theta_Z(z)\|\le\|\theta_{iY^{-1}}(Y^{-1}y)\| .
\end{eqnarray}

Applying again~(\ref{(2)}) with $Z$ replaced by $iY^{-1}$ and $z$ by $Y^{-1}y$, one
gets 
$$
\|\theta_{iY^{-1}}(Y^{-1}y)\|\le\|\theta_{iY}(0)\| .
$$ 
Whence the result. $\Box$
\end{proof}

\medskip

Let $Y\in\mathbb{S}_g$. Define
$\displaystyle\lambda(Y)=\min_{m\in\mathbb{Z}^g-\{0\}}{}^{\rm t}\!mYm$. For every
$t\in\mathbb{R}^*_+$, put
$$f_Y(t)=\theta_{itY}(0)=\sum_{m\in\mathbb{Z}^g}\exp(-\pi t ^{\rm t}\!mYm) .$$

\begin{lemma} 
\label{lemma8.3} 
Let $Y\in\mathbb{S}_g$ and put $\lambda=\lambda(Y)$. The
following properties hold.
\begin{enumerate}

\item The function $\mathbb{R}^*_+\rightarrow\mathbb{R}$ that maps $t$ to
$t^{g/2}f_Y(t)$ is increasing. 
\item One has the estimate
$\displaystyle f_Y\Bigl(\frac{g+2}{2\pi\lambda}\Bigr)\le\frac{g+2}{2}$.
\end{enumerate}
\end{lemma}

\begin{proof} $(a)$ The functional equation~(\ref{(1)}) implies
$\sqrt{\det Y}t^{g/2}f_Y(t)=f_{Y^{-1}}(1/t)$ for every $t\in\mathbb{R}^*_+$;
conclude by remarking that $f_{Y^{-1}}$ is decreasing.

\medskip

$(b)$ Part $(a)$ gives
$\displaystyle\frac{\rm d}{{\rm d}t}[t^{g/2}f_Y(t)]\ge0$, that is,
$\displaystyle\frac{g}{2t}f_Y(t)\ge-f_Y'(t)$ for every $t>0$. On the other hand,
$$-\frac{1}{\pi}f_Y'(t)=\sum_{m\in\mathbb{Z}^g}{}^{\rm t}\!mYm\exp(-\pi t ^{\rm t}\!mYm)\ge\sum_{m\in\mathbb{Z}^g-\{0\}}\lambda\exp(-\pi t ^{\rm t}\!mYm)=\lambda[f_Y(t)-1] .$$

One infers $\displaystyle\frac{g}{2t}f_Y(t)\ge\pi\lambda[f_Y(t)-1]$. Choosing
$\displaystyle t=\frac{g+2}{2\pi\lambda}$, one obtains the result. $\Box$
\end{proof}

\begin{propo} 
\label{proposition8.4} 
Let $Y\in\mathbb{S}_g$. Putting
$\lambda=\lambda(Y)$, one has the upper bound
$$\theta_{iY}(0)\le\frac{g+2}{2}\max\Big[\Bigl(\frac{g+2}{2\pi\lambda}\Bigr)^{g/2},1\Bigr].
$$
\end{propo}

\begin{proof} Put $\displaystyle t=\frac{g+2}{2\pi\lambda}$. If $t\ge1$, then
Lemma~\ref{lemma8.3} $(a)$ implies the inequality $f_Y(1)\le t^{g/2}f_Y(t)$. If $t\le1$,
then $f_Y(1)\le f_Y(t)$ since $f_Y$ is decreasing. In any case, one obtains
$$
\theta_{iY}(0)=f_Y(1)\le\max(t^{g/2},1)f_Y(t) .
$$
Conclude by applying 
Lemma~\ref{lemma8.3} $(b)$. $\Box$
\end{proof}

\medskip

Now, to prove Proposition~\ref{proposition8.1} from Lemma~\ref{lemma8.2} and 
Proposition~\ref{proposition8.4}, it suffices
to observe that if $Z\in\mathbb{H}_g$ is Siegel-reduced, then
$\displaystyle\lambda({\rm Im}Z)\ge\frac{\sqrt{3}}{2}$ (see lemma 15 of
\cite{Igu} page 195).

\bigskip

\medskip
  
\noindent {\bf Acknowledgments} 
The main body of this work (by P.P) benefited from hours of discussions with the author of the Appendix
(P.A.), who shared with great generosity his expertise in Arakelov geometry, provided extremely valuable
advices, references, explanations, critics, insights, and even read large parts of preliminary releases of the present 
paper\footnote{Although, as goes without saying, he bears no responsibility for the mistakes which remain.}. 
Pascal actually ended writing the present Appendix, and the
bounds its displays for theta functions should definitely be useful in a much wider context than the
present work\footnote{They have already been used by R.~Wilms in~\cite{Wi16}, see the introduction to Autissier's Appendix.}.  

  Many thanks are also due to Qing Liu for clarifying some points of algebraic geometry, Fabien Pazuki for explaining general diophantine geometry issues, and to
Ga\"el R\'emond for describing us his own approach to Vojta's method, which under some guise plays a crucial role  here. 

   As already stressed, the influence of the orange book~\cite{CE11} should be obvious all over this text. We have used many results of the deep effective 
Arakelov study of modular curves led there by Bas Edixhoven, Jean-Marc Couveignes and their coauthors. We also benefited from a visit to Leiden 
University in June of 2015, where we had very enlightening discussions with Bas, Peter Bruin, Robin de Jong and David Holmes.    

   Olga Balkanova, Samuel Le Fourn and Guillaume Ricotta helped a lot with references and explanations about some results of analytic number theory, and
Jean-Beno\^{\i}t Bost kindly answered some questions about his own arithmetic B\'ezout theorem.

  Finally, many thanks are due to the referee for her or his substantial and helpful work.

\bigskip

\bigskip

{\small Pascal Autissier

\medskip

I.M.B., Universit\'e de Bordeaux, 351, cours de la
Lib\'eration, 33405 Talence cedex, France.

\medskip

pascal.autissier@math.u-bordeaux.fr}

\bigskip

{\small Pierre Parent

\medskip

I.M.B., Universit\'e de Bordeaux, 351, cours de la
Lib\'eration, 33405 Talence cedex, France.

\medskip

pierre.parent@math.u-bordeaux.fr}

\end{document}